\documentclass[aop,preprint]{imsart}

\RequirePackage[OT1]{fontenc}
\RequirePackage{amsthm,amsmath}
\RequirePackage[numbers]{natbib}
\RequirePackage[colorlinks,citecolor=blue,urlcolor=blue]{hyperref}
\usepackage{graphicx}
\usepackage{amssymb}
\usepackage{epstopdf}
\usepackage{color}
\usepackage{times}

\usepackage{subfigure}
\usepackage{epsfig}

\arxiv{arXiv:1503.08416}

\startlocaldefs
\numberwithin{equation}{section}
\theoremstyle{plain}

\def\:{:\,}\numberwithin{equation}{section}
\allowdisplaybreaks
\newtheorem{theorem}{Theorem}[section]
\newtheorem{lemma}[theorem]{Lemma}

\newtheorem{definition}[theorem]{Definition}
\newtheorem{remark}[theorem]{Remark}
\newtheorem{example}[theorem]{Example}

\def\beqn{\begin{equation}}
\def\beqn*{$$}
\def\eeqn{\end{equation}}
\def\eeqn*{$$}

\def\P{\mathbb{P}}
\def\E{\mathbb{E}}
\newcommand{\beq}{\begin{eqnarray}}
\newcommand{\eeq}{\end{eqnarray}}
\newcommand{\beqq}{\begin{eqnarray*}}
\newcommand{\eeqq}{\end{eqnarray*}}

\newcommand{\bx}{{\bf x}}
\newcommand{\by}{{\bf y}}
\newcommand{\bi}{{\bf i}}
\newcommand{\bj}{{\bf j}}

\newcommand{\reals}{{\mathbb R}}
\newcommand{\bbr}{\reals}

\newcommand{\one}{{\bf 1}}
\def\indic{\one}

\newcommand{\psiinv}{\psi^{\leftarrow}}
\def\definedas{\stackrel{\Delta}{=}}

\def\X{\mathcal X}
\def\Cech{\v{C}ech\ }
\def\real{\mathbb R}
\def\cM{\mathcal M}
\def\bi{\mathbf i}
\def\Xi{\X_\bi}

\def\IPk{\mathcal I_{|\mathcal P_n|, k}}
\def\IP2{\mathcal I_{|\mathcal P_n|, 2}}
\def\Pn{\mathcal P_n}
\def\Xsupn{\X^{(n)}}
\def\Xsupnk{\X^{(n)}_k}
\def\dkn{d_{k,n}}
\def\ckn{c_{k,n}}
\def\tildeN{\widetilde{N}}
\def\regN{N}

\endlocaldefs

\begin{document}

\begin{frontmatter}
\title{Limit Theorems for Point Processes under Geometric Constraints (and Topological Crackle)}
\runtitle{Point processes under geometric constraints}

\begin{aug}

\author{\fnms{Takashi} \snm{Owada} \thanksref{t4}
  \ead[label=e2]{takashiowada@ee.technion.ac.il}
\ead[label=u2,url]{} }
\and
\author{\fnms{Robert J.} \snm{Adler}\thanksref{t1,t4}
\ead[label=e1]{robert@ee.technion.ac.il}
\ead[label=u1,url]{webee.technion.ac.il/people/adler}}

\thankstext{t1}{Research supported in part by SATA, AFOSR Award FA9550-15-1-0032.}
\thankstext{t4}{Research supported in part  by  URSAT, ERC Advanced Grant 320422.}
\runauthor{Owada and Adler}
\affiliation{Technion -- Israel Institute of Technology}

\address{Takashi Owada\\Electrical Engineering\\
 Technion, Haifa, Israel 32000\\
\printead{e2}\\
}

\address{Robert Adler\\Electrical Engineering\\
 Technion, Haifa, Israel 32000\\
\printead{e1}\\
\printead{u1}}

\end{aug}

\begin{abstract}
We study the asymptotic nature of geometric structures formed from a point cloud of  observations of (generally heavy tailed) distributions in a Euclidean space of dimension greater than one. A typical  example is given by the Betti numbers of \v{C}ech complexes built over the cloud. The structure of dependence and  sparcity (away from the origin) generated by these distributions leads to limit laws expressible via 
non-homogeneous, random, Poisson measures. The parametrisation of the limits depends on both the tail decay rate of the observations and the particular geometric constraint being considered. 

The main theorems of the paper generate a new class of results in the well established  theory of extreme values, while their applications are of significance for the fledgling area of rigorous results in topological data analysis. In particular, they provide a broad theory for the empirically well-known  phenomenon of homological `crackle';  the continued presence of spurious homology in samples of topological structures, despite increased sample size.

\end{abstract}



\end{frontmatter}

\section{Introduction} \label{sec:intro}

The main results in this paper lie in two seemingly unrelated areas, those of classical Extreme Value Theory (EVT), and the fledgling area of rigorous results in Topological Data Analysis (TDA).  

The bulk of the paper, including its main theorems, are in the domain of EVT, with many of the proofs coming from Geometric Graph Theory. However, the consequences for TDA, which will not appear in detail until the penultimate section of the paper,  actually provided our initial motivation, and so we shall start with a brief description of this 
aspect of the paper.

Many problems in TDA start with a `point cloud', a collection $\mathcal{X}=\{x_1,\dots,x_n\}$ of points in $\bbr^d$, from which more complex sets are constructed. Two simple examples are the simple union of balls
$$
U(\mathcal X,r) \ \definedas \  \bigcup_{k=1}^n B(x_k; r),
$$
where $B(x; r)$ is a closed ball of radius $r$ about the point $x$, and the  \textit{\v{C}ech complex}, 
$\check{C}(\mathcal{X},r)$.
\begin{definition}
\label{cech:defn}
Let $\mathcal{X}$ be a collection of points in $\bbr^d$ and $r$ be a positive number. Then the \v{C}ech complex $\check{C}(\mathcal{X},r)$ is defined as follows. 
\begin{enumerate}
\item The $0$-simplices are the points in $\mathcal{X}$. 
\item A $p$-simplex $\sigma=[x_{i_0}, \dots, x_{i_p}]$ belongs to $\check{C}(\mathcal{X},r)$ whenever a family of closed balls $\bigl\{ B(x_{i_j}; r/2), \, j=0,\dots,p \bigr\}$ has a nonempty intersection. 
\end{enumerate}
\end{definition}

\v{C}ech complexes are higher-dimensional analogues of \textit{geometric graphs}, a notion more familiar to probabilists.
\begin{definition}  \label{def.geometric.graph}
Given a finite set $\mathcal{X} \subset \bbr^d$ and a real number $r>0$, the geometric graph $G(\mathcal{X},r)$ is the undirected graph with vertex set $\mathcal{X}$ and  edges $[x,y]$ for all pairs $x,y\in\X$ for which $\|x-y\| \leq r$. 
\end{definition}
For a given a \v{C}ech complex, it is immediate from the definitions that its $1$-skeleton is actually a geometric graph. Examples of both are given in Figure \ref{f:cech.geometric}.

\begin{figure}[!htb]
  \label{f:cech.geometric}
\includegraphics[height=4.5cm,width=11cm]{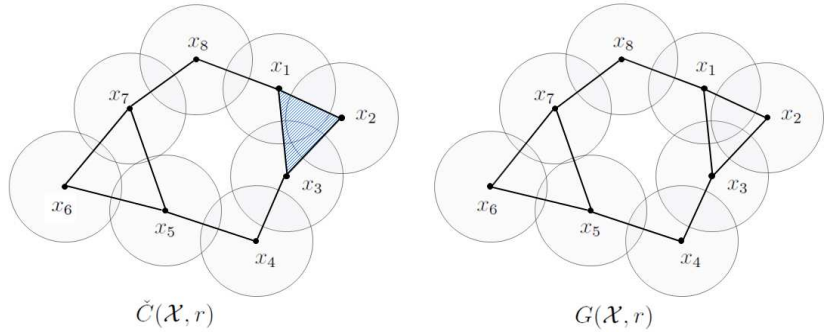}
\caption{Take $\mathcal{X} = \{x_1,\dots,x_8\}\subset \bbr^2$. Note that the $2$-simplex $[x_1,x_2,x_3]$ belongs to $\check{C}(\mathcal{X},r)$, since the three balls with radius $r$ centred at $x_1,x_2,x_3$ have a common intersection. The Betti numbers of $\check{C}(\mathcal{X},r)$ are $\beta_0=1$ (one connected component), $\beta_1=2$ (two closed loops), while all  others are zero. The geometric graph $G(\mathcal{X}, r)$  shows the $1$-skeleton of $\check{C}(\mathcal{X},r)$. }
\end{figure}

A typical TDA paradigm is to create either $U(\X,r)$ or $\check{C}(\X,r)$ as an estimate of some underlying 
sub-manifold, $\mathcal M\subset \real^d$, from which $\X$ is actually sampled, and then consider its homology, typically via what is known as persistent homology (with which we shall not concern ourselves in this paper) or through Betti numbers.

We shall concern ourselves primarily  with Betti numbers, a basic quantifier in Algebraic Topology. Roughly speaking, given a topological space $X$, the $0$-th Betti number $\beta_0(X)$ counts the number of its connected components, while for $k \geq 1$, the $k$-th Betti number $\beta_k(X)$ counts the number of  $k$-dimensional `holes' or `cycles' in $X$. For example, a one-dimensional sphere -- i.e.\ a circle -- has $\beta_0=1$, $\beta_1=1$,  and $\beta_k=0$ for all $k\geq 2$. A two-dimensional sphere has $\beta_0=1$, $\beta_1=0$, and $\beta_2=1$,  and all others  zero. In the case of a two-dimensional torus, the non-zero Betti numbers are  $\beta_0=1$, $\beta_1=2$, and $\beta_2=1$. At a more formal level, $\beta_k(X)$ can be  defined as the dimension of the \textit{$k$-th homology group}, however the essence of this paper can be captured without  knowledge of homology theory. In the sequel, simply viewing $\beta_k(X)$ as the number of $k$-dimensional holes will suffice. The readers who are interested in a thorough coverage of homology theory may refer to \cite{hatcher:2002} or \cite{vick:1994}.

Returning to the two sets  $U(\X, r/2)$ and $\check{C}(\X,r)$,  by a classical result known as the   Nerve Theorem (\cite{borsuk:1948})  they are homotopy equivalent, and so from the point of view of TDA they are conceptually equivalent. However, since the definition of the  \Cech complex  is essentially combinatorial, it is computationally more accessible, and so of more use in applications. Hence we shall concentrate on it from now on.

There is now a substantial literature, with  \cite{niyogi:smale:weinberger:2008,niyogi:smale:weinberger:2011}  being the papers that motivated us, that shows that, given  a nice enough $\cM$, and any $\delta>0$,  there are explicit conditions on $n$ and $r$ such that the homologies of $U$ and $\check{C}$ are equal to the homology of $\cM$ with a probability of at least $(1-\delta)$.  
Typically, these results hold when the sample $\X$ is either taken from $\cM$ itself, or from $\cM$ with small (e.g.\ Gaussian)  random,  perturbative error. However, if the error is allowed to become large (e.g.\ has a heavy tailed distribution) then these results fail, and a phenomenon known as {\it crackle} occurs. An example is given in 
Figure 2, 
taken from  \cite{adler:bobrowski:weinberger:2014},
 in which the underlying $\cM$ (a stratified manifold in this case) is an annulus in $\real^2$. From the point of view of TDA, the main aim of the current paper is to derive rigorous results describing the distribution of the  `extraneous homology elements' appearing in Figure 2(d).

\begin{figure}[htb!]
\label{fig:mfld}
\centering
\subfigure[]
{
  \centering
    {\includegraphics[scale=0.2]{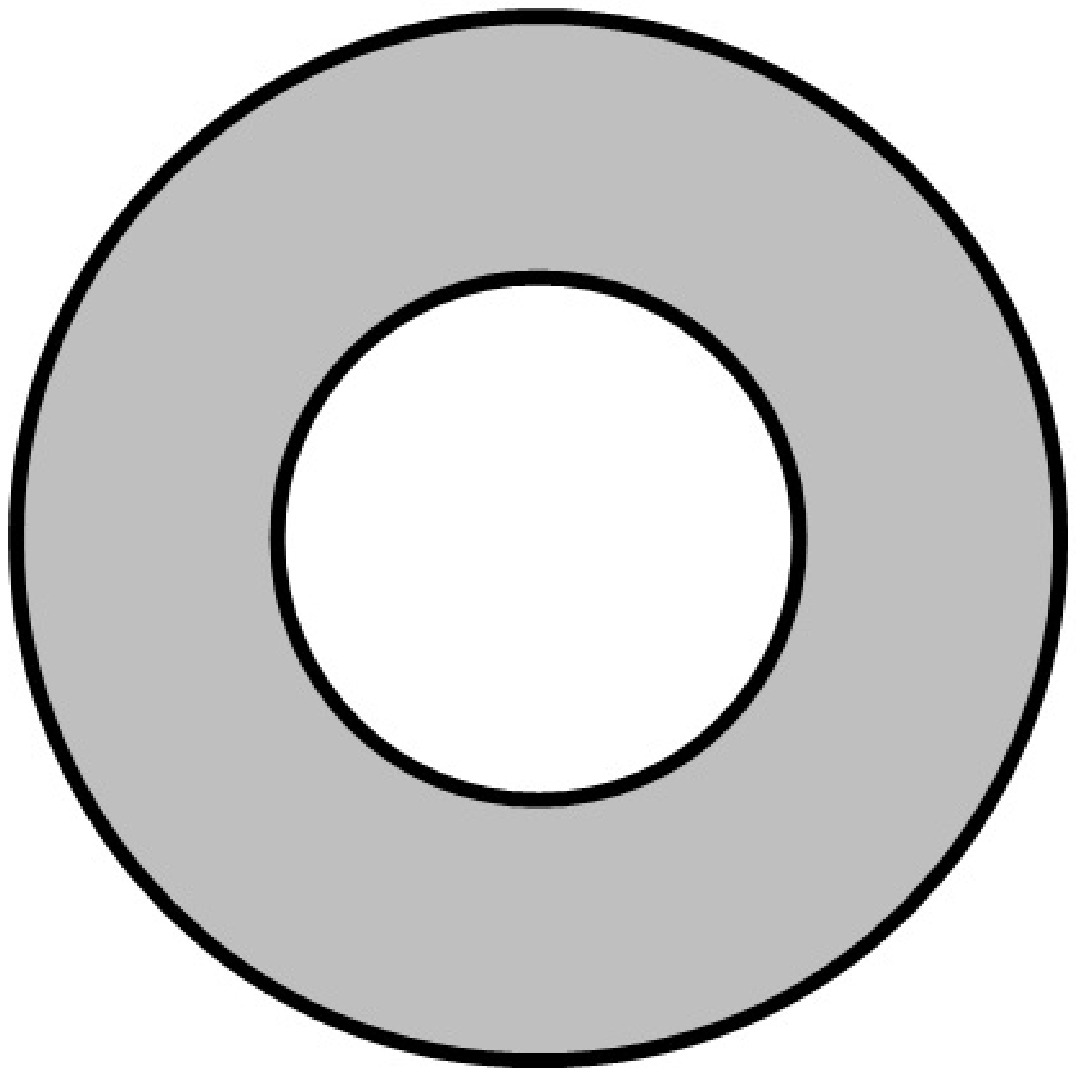}}
    \label{fig:mfld_orig}
}
\subfigure[]
{
  \centering
    {\includegraphics[scale=0.2]{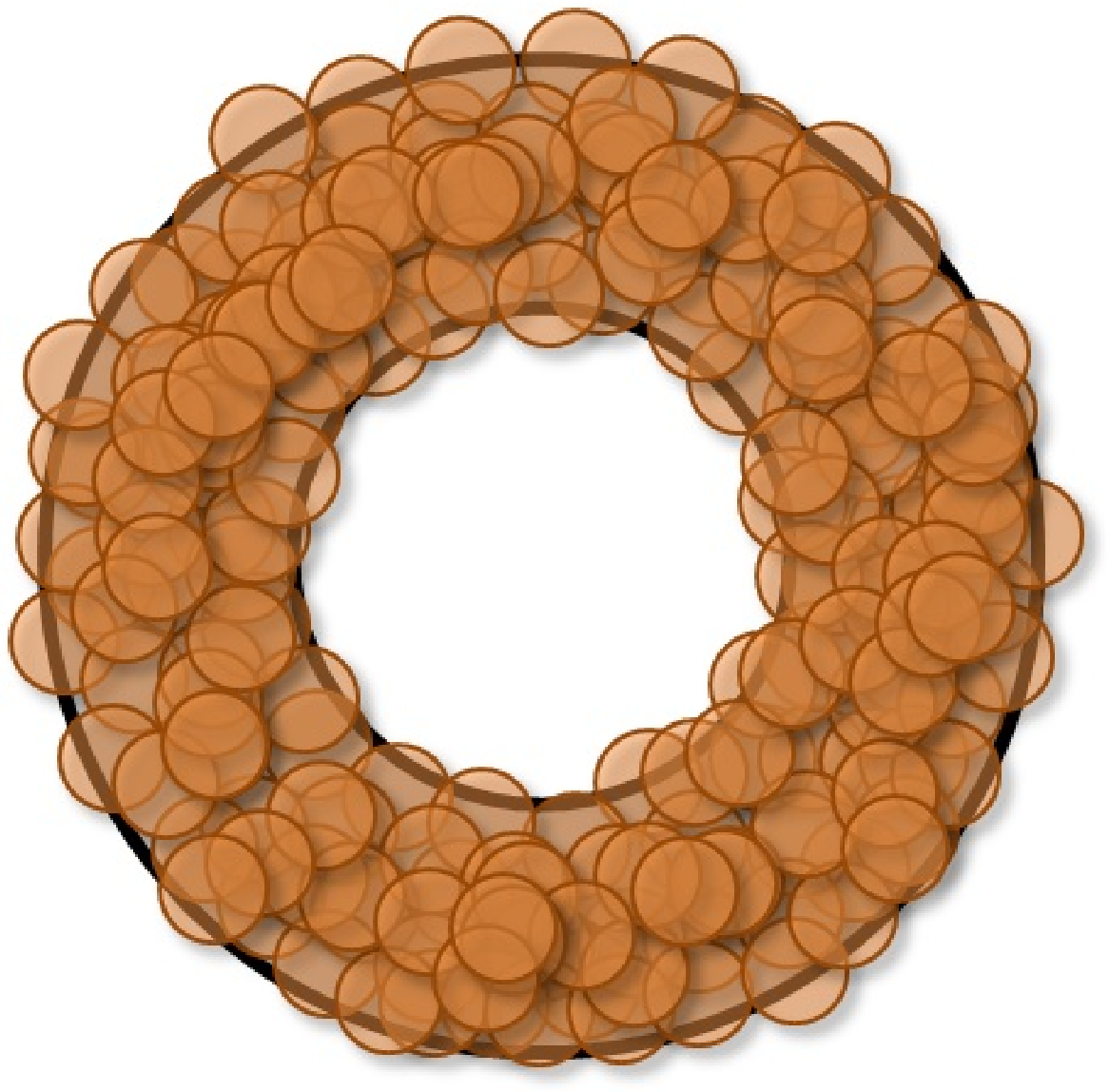}}
    \label{fig:mfld_no_noise}
}
\subfigure[]
{
  \centering
    {\includegraphics[scale=0.2]{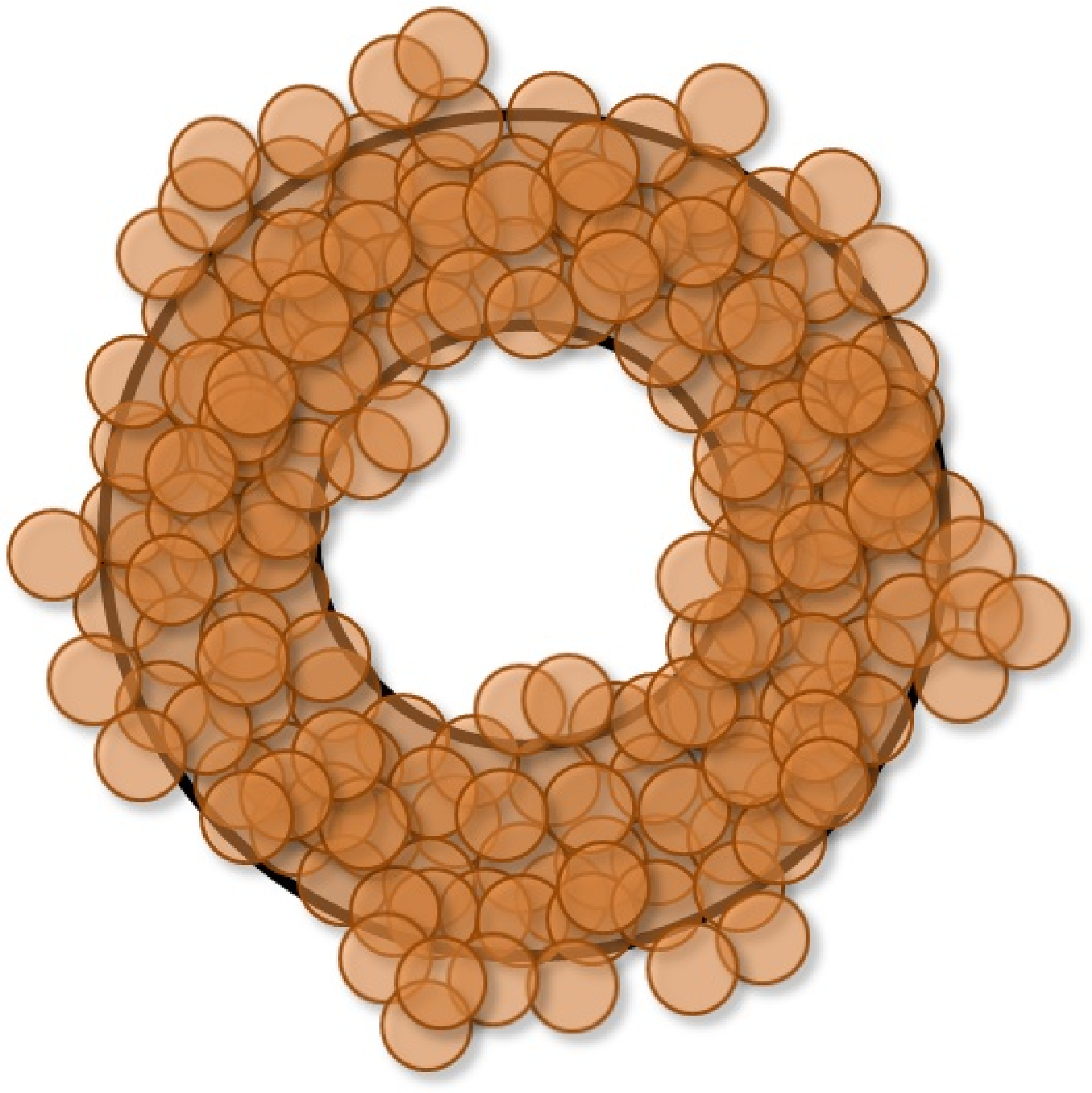}}
    \label{fig:mfld_noise_low}
}
\subfigure[]
{
  \centering
    {\includegraphics[scale=0.2]{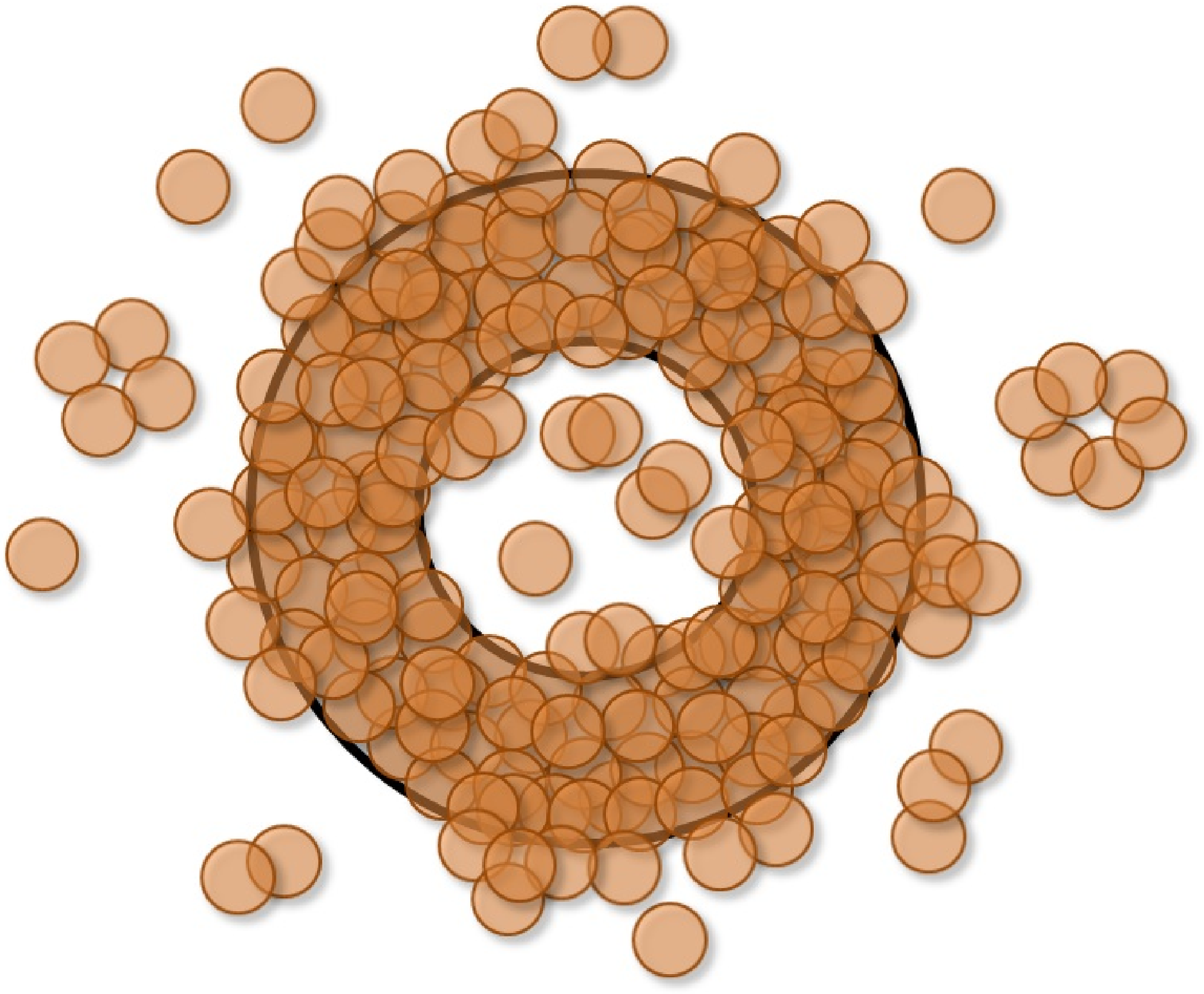}}
    \label{fig:mfld_noise_high}
}

\caption{(a) The original space  $\mathcal M$ (an annulus) that we wish to recover from random samples. (b) With the appropriate choice of radius, we can easily recover the homology of the original space from the union of balls around a random sample from $\mathcal M$. (c) In the presence of bounded noise, homology recovery is undamaged. (d) In the presence of unbounded noise, many extraneous homology elements appear, and significantly interfere with homology recovery.}
\label{fig:complexes:simp_comp}
\end{figure}

These results, as central as they are to the paper, will  be given in detail only in  Section \ref{s:application}.  There we shall show that the Betti numbers of the sets generating this spurious homology, which occurs far from the support of $\mathcal M$, satisfy Poisson limit laws. 
The limit laws are stated in terms of the distance of these sets from $\mathcal M$, so that they are in terms of limiting Poisson laws, the intensities of which depend significantly on the tail decay rate of the perturbative error. The practical importance of this is that it gives useful information on  {\it how much} crackle occurs, what its  {\it topological nature} is, and {\it where} it occurs, at least in a limiting scenario.  
The results of Section \ref{s:application}, therefore, are the main ones from the point of view of TDA.

The reason that we need to wait so long to get to these results is that the natural probabilistic setting for both stating and proving the theorems from which they follows is that of EVT. Recall that the classical setup of (multivariate) EVT takes a point process generated by independently and identically distributed (i.i.d.) random variables in  $\real^d$ -- a random point cloud --  scales and centers the cloud, and then, typically,  shows weak convergence of the new cloud to a Poisson random measure. The literature dealing with this setup is rich, and impossible to summarise here, but the main references for EVT are probably \cite{embrechts:kluppelberg:mikosch:1997}, \cite{resnick:1987}, and \cite{dehaan:ferreira:2006}, while a comprehensive introduction to point process theory is provided by \cite{daley:verejones:2003}.

%

Beyond these more classical papers, over the last decade or so there have  been a number of papers treating geometric descriptions of multivariate extremes from the perspective of point process theory, among them \cite{balkema:embrechts:2007}, \cite{balkema:embrechts:nolde:2010}, and \cite{balkema:embrechts:nolde:2013}.  In particular, Poisson limits of point processes possessing a U-statistic structure were investigated by \cite{dabrowski:dehling:mikosch:sharipov:2002} and \cite{schulte:thale:2012}, the latter also treating  a number of examples in stochastic geometry. 

All of these papers focused on the (limiting) distributions of extreme sample clouds, but so far, surprisingly few attempts have been made at analyzing their topological features. The objective of the present work is to present a unified, geometric, and topological approach to the point process convergence of extreme sample clouds. 
%
%
Results on the limiting behaviour of  geometric complexes are quite new, with results not unrelated to our own, include
  \cite{kahle:2011}, \cite{kahle:meckes:2013}, \cite{yogeshwaran:adler:2015}, and \cite{yogeshwaran:subag:adler:2014}. In particular \cite{kahle:meckes:2013} and \cite{yogeshwaran:subag:adler:2014} derived various central and Poisson limit theorems for the Betti numbers of the  random \v{C}ech complexes $\check{C}(\mathcal X_n, r_n)$, with $\mathcal X_n$ a  random point set in $\bbr^d$ and $r_n$  a threshold radius. The resulting  limit theorems depend heavily on the asymptotics of $nr_n^d$, as $n \to \infty$. For example, \cite{kahle:meckes:2013} investigated the \textit{sparse regime} (i.e.\ $nr_n^d \to 0$) so that the spatial distribution of complexes is sparse, and they are observed mostly as isolated components. In contrast, the main focus of \cite{yogeshwaran:subag:adler:2014} was the \textit{thermodynamic regime} (i.e.\ $nr_n^d \to \xi \in (0,\infty)$) in which complexes are large and highly connected. 
  
However, with the single exception of \cite{adler:bobrowski:weinberger:2014}, already discussed above, none of  these papers 
has results related to  crackle.  The main discovery of \cite{adler:bobrowski:weinberger:2014} was that the topological behaviour of the individual components of geometric complexes 
built over point clouds of (typically)  heavy tailed random variables  is determined 
by how far away they lie from $\mathcal M$. In the simple case in which $\mathcal M$ is a single point at the origin, this is exhibited by the 
 formation of `rings' around the origin, each ring containing points of the cloud which exhibit quite different geometric and topological properties. The positions of the rings depend crucially on the tail decay rate of the  underlying probability distribution of the observations as well as the topological property being studied. 
 The contribution of the present paper is to take the discovery of \cite{adler:bobrowski:weinberger:2014}, made on the basis of only two examples, and develop it into a full theory, covering wide classes of distributions and developing far more detailed descriptions of the crackle phenomenon. 
  
%
%
%
%

The  remainder of the paper is structured as follows: In Section \ref{s:general}, we establish general results on  point process convergence under certain geometric constraints.  Since the results here are very general, the conditions of the theorems may at first seem rather unnatural and hard to check. That this is actually not the case becomes clear when we turn to specific scenarios  depending on the tail properties of the underlying distributions.  In doing so, it turns out to be natural to distinguish  between what we shall henceforth call `heavy-tailed' (e.g.\ power law tails) and `light-tailed' (e.g.\ $ e^{-||x||^\tau}$, $\tau >0$) densities. Section \ref{s:heavy.tail} defines heavy-tailed densities more precisely,  via regular variation of tails   (giving marginal distributions in the max-domain of attraction of the 
Fr\'echet law) and then applies the general results of Section 
\ref{s:general} to obtain limiting Poisson results for this case. 
 Section \ref{s:light.tail}  then does the same for  light-tailed densities (with tails controlled by the von Mises function). 
 Typically, a standard argument in EVT characterises distributional tails via distribution functions (not densities), but our calculation will proceed relying heavily on the treatment of densities. Therefore, in both Sections \ref{s:heavy.tail} and \ref{s:light.tail}, we shall  distinguish the underlying distributions using density functions. 

The final Section \ref{s:application} presents several applications of the point process convergence, one of which is the limit theory of  Betti numbers alluded to above. In addition, we give some independently interesting  limit theorems on partial sums and maxima under geometric constraints. To the best of our knowledge, these results have no precursors in the random topology literature, since, while they are simple corollaries of the Poisson convergence that is our main theme, proving them {\it ab initio} would be quite hard. 

A final remark is that all random points in the present paper are generated from an inhomogeneous Poisson point process in $\bbr^d$ with intensity $nf$ ($f$ is a density function). However, all the results obtained can be carried over to the  situation of an i.i.d.\ random sample without difficulty. Further, we only consider spherically symmetric densities. Although the spherical symmetry assumption has very little to do with our results, we adopted it in order to avoid unnecessary technicalities, which would otherwise blur the message of the paper as well as making the treatment much longer and more cumbersome.

Since the paper is rather long, and unavoidably heavy on notation, we conclude the Introduction with a short notation guide.

\begin{table}[!h]
\begin{tabular}{|l|l|c|}
\hline
Symbol      & Definition or Meaning & Place Defined\\ 
\hline 
 $\Check{C}(\X,r)$       &      \Cech complex          &  Definition \ref{cech:defn}            \\
 $G(\X,r)$       &      Geometric graph          &  Definition \ref{def.geometric.graph}            \\
   $\X$  & Point cloud $\{X_1,\dots\} $       &              \\
     $\mathcal P_n$     & Poisson point process with intensity $nf$               &              \\
   $\X_k$     &  First $k$ points in $\X$               &       \eqref{Xk:def}       \\
   $\X_{\bi}$, $\bi = \{i_1,\dots, i_k\}$  & Subset $\{X_{i_1},\dots,X_{i_k}\}$ of $\X$ &   \eqref{Xi:def}   \\ 
    $\mathcal I_{n,k}$       &  All subsets of $\{1,\dots,n\}$ of size $k$         &      \eqref{In:def}        \\
          $r_n, \ c_{k,n},\ d_{k,n},\ R_{k,n}$ etc.   & Scaling sequences, specific to section               &              \\
           $X_i^{(n)}$     &  $X_i$, translated and scaled              &       \eqref{e:abb.X1}       \\
                     $\Xsupnk$     &  First $k$ points of normalised $X_i$s              &       \eqref{e:abb.X2}       \\
          $\Xsupn_\bi$, $\bi = \{i_1,\dots, i_k\}$    &  Subset $\{ X_{i_1}^{(n)}, \dots, X_{i_k}^{(n)} \}$          &       \eqref{e:abb.X3}       \\
         $S(x)$   & Maps  $x\in\real^d$ onto +ve orthant of   $S^{d-1}$               &  \eqref{SX:def}            \\
         $L_k$   &   Upper-diagonal subset of a $k$-dim. cube       &   \eqref{e:upper.diagonal}     \\
    $\regN_n^{(k)}$   &  No.\ of $k$-tuples satisfying a  condition               &    \eqref{Nnk:defn}          \\
    $\tildeN_n^{(k)}$   &    No.\ of {\it isolated} $k$-tuples as above           &       \eqref{e:related.pp}       \\ 
 $p(\bx;r)$  & Local integral of density               &        \eqref{e:comp.func}      \\
     $J(\theta)$      &  Polar Jacobian                &    \eqref{J:def}          \\
          $x \preceq y$, for $x,y\in\real^d$          &  $x_i\leq y_i$, $i=1,\dots,d$ &               \\
$x \prec y$, for $x,y\in\real^d$          &  $x_i < y_i$, $i=1,\dots,d$ &               \\ 
$a \vee b$, for $a,b \in\real$          &  $\max \{ a,b \}$ &               \\ 
$a \wedge b$, for $a,b \in\real$          &  $\min \{ a,b \}$ &               \\ 
$B(x; r)$      &  Ball, centre $x$,    radius $r$            &              \\
$||\cdot||$      &  Euclidean norm           &              \\
    $S^d$    &  $d$-dimensional unit radius sphere               &              \\
      $s_d$    &   Surface area of $S^d$             &              \\
        $\lambda_k$     & $k$-dimensional Lebesgue measure               &              \\
 $Poi(\lambda)$ & Poisson distribution with mean $\lambda$& \\
 $\Rightarrow$ & Weak convergence & \\
 $\mathbf{0}$       & Vector of zeros       &              \\
$e_1$       & Vector with 1st comp. $1$ and others zero       &              \\
    Ann$(a,b)$    & Annulus, $\{x\in\real^d\: a\leq \|x\|\leq b\}$               &              \\
     \hline
\end{tabular}
\end{table}

\section{Point Process Convergence: General Theory}  \label{s:general}

In this section we prove a general convergence result for point processes with geometric constraints. The main result,
Theorem \ref{t:general.thm}, which holds under somewhat opaque conditions,  will then be applied in the two following sections to obtain results under specific, and quite natural,  distributional assumptions. Before getting to results, however, we need to set up a framework and notation, which we do in the following subsections.

\subsection{Point clouds and indicator functions for geometry}
\label{pc:subsec}

We start with a collection  
$\X= \{X_i, \, i \geq 1\}$ of i.i.d., $\bbr^d$-valued random variables with spherically symmetric probability density $f$. For $n \geq 1$, let $N_n$ be a Poisson random variable with mean $n$, and independent of the $X_i$. Then the random measure with atoms at the points   $X_1,  \dots, X_{N_n}$, which we denote by $\mathcal P_n$, is an inhomogeneous  Poisson point process with intensity $nf$. 
 
 Now choose a positive integer $k$, which will remain fixed for the remainder of the paper. Although some of our results hold for $k=1$, we shall  henceforth  take $k\geq 2$, unless otherwise stated. As will soon be clear, in the $k=1$ case many of the objects and functions we build will degenerate, in the sense that the corresponding results will become known results of non-topological EVT.
 
 Let 
 \beq   \label{Xk:def}
 \X_k \ \definedas \ \{ X_1, \dots, X_k \}
 \eeq
 be the first $k$ points and for an ordered $k$-tuple $\bi = (i_1,\dots,i_k)$,  $1 \leq i_1 < \dots < i_k$, consider the $k$-tuple of points
   \beq
 \label{Xi:def}
 \X_\bi \ \definedas \  \left\{X_{i_1}, \dots, X_{i_k} \right\}.
  \eeq
We are interested in the geometric objects built over the $\X_\bi$, for all $\bi$ in the set $\mathcal I_{|\Pn|,k}$, where 
\beq
\label{In:def}
\mathcal I_{m,k}\  \definedas \  \left\{\bi=(i_1,\dots,i_k)\: 1 \leq i_1 < \dots < i_k \leq m \right\},
 \eeq
 and $|\mathcal  P_n|=N_n$ is the number of atoms of $\mathcal P_n$.

 To set this up formally, we define an indictor function  $h\:(\bbr^d)^k \to \{ 0,1 \}$ which will capture the geometry.  
 Two concrete examples related to problems that will be considered in detail later sections are 
\beq
h(x_1, \dots, x_k) &= \indic\bigl( G \left( \{ x_1, \dots, x_k \}, \, 1 \right) \cong \Gamma \bigr),  \label{e:geometric.graph} 
\eeq
and
\beq
h(x_1, \dots, x_k) &= \indic \bigl(\beta_{k-2} \left( \check{C} \left( \{ x_1, \dots, x_k \}, \, 1 \right) \right) = 1  \bigr), \label{e:betti.number}
\eeq
where $\indic$ is the indicator function, $G$ is a geometric graph, $\Gamma$ is a fixed connected graph with $k$ vertices, $\cong$ denotes graph isomorphism, and  $\check{C}$ is a \v{C}ech complex.

Throughout the paper we shall require that $h$ is shift invariant, viz.
\begin{equation}  \label{e:invariance1}
h(x_1, \dots, x_k) \ =\  h(x_1+y, \dots, x_k+y), \quad\text{for all } x_1, \dots, x_k, y \in \bbr^d.
\end{equation}
We shall refer to this as the {\it shift invariance condition in what follows.
Further, we shall typically require a {\it points in proximity condition}, defined by the requirement } that there exists a finite $M > 0$ for which
\begin{equation}  \label{e:close.enough}
h(0,x_1,\dots,x_{k-1}) = 0 \ \ \text{if } \|x_i\| > M \ \, \text{for some } i \in \{ 1,\dots,k-1 \}.
\end{equation}
In addition, given a non-increasing sequence of positive numbers $(r_n, \, n \geq 1)$ (the case $r_n \equiv$ constant is permissible) we define  scaled versions of $h$, $h_n\:(\bbr^d)^k \to \{ 0,1 \}$, by setting
\begin{equation}  \label{e:indep}
h_n(x_1,\dots,x_k) \ \definedas\  h( x_1/r_n,\dots,x_k/r_n). 
\end{equation}
Clearly, $h_n$ is also shift invariant. Furthermore,  if $r_n$ is small, then  \eqref{e:close.enough} implies   that     $h_n(x_1,\dots,x_k) = 1$ only when  all the points $x_1,\dots,x_k$ are close to one another. 

In what follows, we shall be particularly interested in $k$-tuples $\Xi$, $\bi\in \mathcal I_{|\mathcal P_n|,k}$ which not only satisfy  the constraint
implicit in the indicator  $h_n$ but are also separated (by at least $r_n$) from the other points in  $\Pn$. Thus, we  need to introduce another sequence of  indicator functions. Given a  subset  $\mathcal{X}$  of $k$ points in $\bbr^d$, and a larger, but  finite, 
set $\mathcal{Y} \supset \mathcal{X}$ of points in the same space, define
\begin{equation}  \label{e:def.gn}
 g_n \left(\mathcal{X}, \mathcal{Y} \right) \ \definedas h_n (\mathcal{X} ) \times \one \left(G ( \mathcal{X} ,r_n ) \ \text{is an isolated component of } G (\mathcal{Y},r_n) \right). 
\end{equation}

\subsection{Sequences of point processes}
\label{evt:subsec}
Since we are going to be dealing with the geometry of random variables `in the tail' of $f$, as in standard EVT we are going to need to normalise them before trying to formulate any results.

Thus, with $k$ still fixed,  we assume that we have two sequences of constants,   $d_{k,n} \in \bbr$ and $c_{k,n}>0$, with which to shift and then scale the $X_i$, and, in  notations that will save a  lot of space, define:
\beq
X_i^{(n)} &\definedas& \frac{X_i - \dkn S(X_i)}{\ckn}  \label{e:abb.X1} \\
\Xsupnk &\definedas&   \left\{X_1^{(n)},\dots,X_k^{(n)}\right\}  \label{e:abb.X2}\\ 
\Xsupn_\bi &\definedas&   \left\{X_{i_1}^{(n)},\dots,X_{i_k}^{(n)}\right\} \label{e:abb.X3}
\eeq
where
\beq
\label{SX:def}
S(x) \ \definedas \  \frac{\left( \, |x^{(1)}|, \dots, |x^{(d)}| \, \right)}{\|x\|}
\eeq
maps points $x = (x^{(1)}, \dots, x^{(d)})$ in $\real^d$ onto the positive orthant of $S^{d-1}$.

We can now define the two random measures, that will be at the centre of all that follows, over the space 
 $E_k \times L_k$, where $E_k$ is a measurable set in $\left([-\infty,\infty]^d\right)^k$ and $L_k$ is the upper-diagonal subset of a $k$-dimensional cube, viz.
\begin{equation}  \label{e:upper.diagonal}
L_k  \ \definedas \  \bigl\{ (z_1,\dots,z_k) \in [0,1]^k: 0 \leq z_1 < \dots < z_k  \leq 1 \bigr\}\,.
\end{equation}

The measures, for Borel $A\subset E_k\times L_k$, and $\epsilon_x$ a Dirac measure at $x\in E_k\times L_k$,  are given by
\beq
\label{Nnk:defn}
 N_n^{(k)} (A) \ \definedas  \ \sum_{  \bi\in \IPk} h_n( \Xi) \,
 \epsilon_{(\Xsupn_\bi, \bi/ |\mathcal P_n|)}(A),
\eeq
and 
\beq
  \label{e:related.pp}
\tildeN_n^{(k)}(A)   \ \definedas  \ \sum_{  \bi\in \IPk} g_n( \Xi,  \Pn) \,
 \epsilon_{(\Xsupn_\bi, \bi/ |\Pn|)}(A).
\eeq
(Note that  $N_n^{(k)}  = \widetilde N_n^{(k)} \equiv 0$ whenever $k>n$.)

The point process ${N}_n^{(k)}$   counts $k$-tuples satisfying the geometric condition implicit in $h_n$, while 
  $\widetilde N_n^{(k)}$ adds the additional restriction that  the geometric graphs generated by such $k$-tuples be distance at least $r_n$ from all the remaining points of $\Pn$. It will turn out that, while ${N}_n^{(k)}$  always dominates 
  $\tildeN_n^{(k)}$,  both exhibit  very similar limiting behaviour. In fact, the proof of the general result on point process convergence will follow the route of first establishing a result for  $N_n^{(k)}$, which is somewhat easier because the geometric condition implicit in the indicator does not include $\Pn$, and then showing that since the two processes are close to one another the same limit holds for $\tildeN_n^{(k)}$.

\subsection{A general theorem}
\label{general:subsec}
Retaining the notation of the previous subsection, we now set up the main conditions that will ensure the weak convergence of the processes  ${N}_n^{(k)}$   and   $\widetilde N_n^{(k)}$. Although we ultimately want theorems in which conditions relate to the tail behaviour of the underlying density of the $X_i$, the conditions in this section relate both to this and to  what amounts to  asymptotic sparsity for  certain structures in $\Pn$. In  Sections \ref{s:heavy.tail} and  \ref{s:light.tail} we shall see how to derive this independence from tail considerations.

For our first condition, which we refer to as the {\it vague convergence condition}, we assume that there exist a Radon measure $\nu_k$ on a measurable set 
$E_k \subset ( [-\infty,\infty]^d )^k$, not assigning mass to points at infinity, and normalising constants, as in the 
previous subsection, such that
 \begin{equation} 
 \label{e:conv.in.exp}
 {n \choose k}
\ \P \left\{ h_n(\X_k) = 1,\  \Xsupnk \in \cdot  \right\}  \ \stackrel{v}{\to} \nu_k(\cdot),   
\end{equation}
where $\stackrel{v}{\to}$ denotes vague convergence in $E_k$.

The second, {\it sparsity condition,} is given by 
\begin{equation}
  \label{e:anti.cluster} 
n^{2k-l} \P \left\{h_n(\X_k) = h_n(X_1,\dots,X_l,X_{k+1},\dots,X_{2k-l})=1, 
\  \X^{(n)}_{2k-l} \in K \, \right\} \to  0, 
\end{equation}
as $n \to \infty$,
for every $l=1,\dots,k-1$ and every compact set $K$ in $E_{2k-l}$. This condition prevents the random $k$-tuples from clustering too much, and will lead to a Poisson limit for ${N}_n^{(k)}$. The same kind of sparsity, or anti-clustering, condition can be found in, for example, \cite{dabrowski:dehling:mikosch:sharipov:2002} and \cite{schulte:thale:2012}.

An additional  condition will be needed for the convergence of $\tildeN_n^{(k)}$. For this,   let  
$\mathcal{P}_n^{\prime}$
be an independent copy of $\mathcal{P}_n$. Then this condition, combined with  \eqref{e:conv.in.exp}, requires that, asymptotically in $n$, the $k$ points in $\mathcal X_k$ all be distance at least  $r_n$ from the points in 
$\Pn^{\prime}$.  Specifically,
\begin{equation}
\label{e:component}
{n \choose k}  \P \left\{ g_n ( \X_k,\, \X_k \cup \mathcal{P}_n^{\prime}) = 1\,,
\Xsupnk  \in \cdot \right\} \ \stackrel{v}{\to}  \ \nu_k(\cdot).
\end{equation} 

As one might expect from these conditions, when they hold the weak limits of $N_n^{(k)}$ and $\tildeN_n^{(k)}$  will typically  be expressible via Poisson random measures, which we denote by $N^{(k)}$, and now define.

Writing $\lambda_k$ for $k$-dimensional Lebesgue measure,   the Poisson random measure $N^{(k)}$ on $E_k \times L_k$, with mean measure 
$\nu_k \times \lambda_k$ (with $\lambda_k$ restricted to $L_k$) is defined by the finite dimensional distributions 
$$
\P \bigl( N^{(k)} (A) =m \bigr) = e^{-(\nu_k \times \lambda_k)(A)} \bigl( (\nu_k \times \lambda_k)(A) \bigr)^m/ m!\,, \ \ \ m=0,1,2,\dots,
$$
 for all measurable  $A \subset E_k \times L_k$ with $(\nu_k \times \lambda_k)(A) < \infty$. Furthermore, if $A_1, \dots, A_m$ are disjoint subsets in $E_k \times L_k$, then $N^{(k)}(A_1), \dots, N^{(k)}(A_m)$ are independent. 
 
 The random measures $N_n^{(k)}$, $\tildeN_n^{(k)}$, and $N^{(k)}$ are all regarded as random elements in the space $M_p (E_k \times L_k)$ of Radon point measures in $E_k \times L_k$.    Details on such random measures, including issues related to their weak convergence in the space of Radon measures, along with the  vague convergence of  \eqref{e:conv.in.exp} and \eqref{e:component}, can be found, for example, in \cite{resnick:1987} and \cite{resnick:2007}. 

Before stating the main result of this section, a few words about the above three conditions are in order.
Firstly, we note that more  specific definitions of the state space $E_k$ and concrete definitions of the normalising constants $(c_{k,n}, d_{k,n})$ will be given in the subsequent sections, where they will be seen to be dependent on whether the underlying density $f$ has heavy tail or light tail.

 In the heavy-tailed density case, we take $E_k = ([-\infty,\infty]^d)^k \setminus \{ \mathbf{0} \}$ and $d_{k,n} \equiv 0$. Thus there is no shift of the points $X_i$ by $S(X_i)$, and all the scaling does is bring points closer to the origin.  Consequently the point processes $N_n^{(k)}$ and $\tildeN_n^{(k)}$ asymptotically count geometric events equally likely in all the orthants.

On the other hand, in the case of light-tailed densities, one needs to take $E_k = ((-\infty,\infty]^d)^k$ and $d_{k,n}$ is generally non-zero.  Since  $X_i$ is translated by $d_{k,n}S(X_i)$, this implies that, at least  when $n$  is large,  
$N_n^{(k)}$ and $\tildeN_n^{(k)}$  count geometric events basically only when they occur in the non-negative orthant.
Nevertheless, the assumed spherical symmetry of the density allows our main results to be extended to general  domains with only minor additional arguments. 

Thus, since the aim of this section is to establish general theory on point process convergence,  we shall  leave this interperative complication until it arises later in specific examples.


As remarked above, throughout the paper we take $k \geq 2$.  Given the definitions above, we can now explain why this is the case. If $k=1$, there is no  shift invariant indicator  to associate with the point process, and we have
$$
{N}_n^{(1)} (\cdot)\  =\ \sum_{i=1}^{|\mathcal{P}_n|} \epsilon_{\bigl(  c_{1,n}^{-1} (X_i - d_{1,n} S(X_i))\,, \ i /|\mathcal{P}_n| \bigr)} (\cdot)\,.
$$
This type of the point process has been well studied in classical EVT, and we have nothing to add about it.  

We can now state the main result of this section. The proof is deferred to the  Appendix.

\begin{theorem}  \label{t:general.thm}
Let $(X_i, \, i \geq 1)$ be a sequence of i.i.d. $\bbr^d$-valued spherically symmetric random variables. Let $h_n\:(\bbr^d)^k \to \{ 0,1 \}$ be a sequence of  indicators as  in  \eqref{e:indep}, satisfying  shift invariance as in \eqref{e:invariance1} and the points in proximity condition \eqref{e:close.enough}. Assume also that the vague convergence  condition \eqref{e:conv.in.exp} and the sparcity condition \eqref{e:anti.cluster} hold.  Then
$$
{N}_n^{(k)} \   \Rightarrow  \  N^{(k)},
$$
where $\Rightarrow$ denotes weak convergence in the space $M_p (E_k \times L_k)$. \\
 In addition, suppose that the global separation condition \eqref{e:component} holds. Then we also have that 
$$
 \widetilde N_n^{(k)}  \Rightarrow N^{(k)}.
$$ 
\end{theorem}

\begin{remark}
Note that the case $h \equiv 1$ is precluded by \eqref{e:close.enough}. In fact, if $h \equiv 1$, then it is easy to see that  \eqref{e:conv.in.exp} and \eqref{e:anti.cluster} cannot hold simultaneously.  
\end{remark}

\section{Heavy Tail Cases}  \label{s:heavy.tail}

As a special case of the general theory introduced in the last section, we shall  discuss  point process convergence when the underlying density on $\bbr^d$ has a heavy tail. More specifically, we assume that the density has a regularly varying tail (at infinity) in the sense that, for every $\theta \in S^{d-1}$ (equivalently, for some $\theta \in S^{d-1}$, due to the spherical symmetry of $f$), and  some $\alpha > d$,
\beqq
\lim_{r \to \infty} \frac{f(tr\theta)}{f(r\theta)} = t^{-\alpha}, \ \ \ \text{for every } t>0.
\eeqq
Writing $RV_{-\alpha}$ to denote the family of regularly varying functions (at infinity) of exponent $-\alpha$, then the above is equivalent to the requirement 
\begin{equation}  \label{e:RV.tail}
f\  \in\  RV_{-\alpha}\,.
\end{equation}
In the  one-dimensional case ($d=1$) it is known that  regular variation of the tail in $f$ suffices for the distribution to be in the max-domain of attraction of the Fr\'echet law; see for example Theorem 3.3.7 in \cite{embrechts:kluppelberg:mikosch:1997}. 

In general, given a non-increasing sequence of positive numbers $(r_n, \, n \geq 1)$, we define $h_n$, $g_n$ as in Section \ref{s:general}. Then, for a fixed positive integer $k \geq 2$, the point processes we are going to explore are
\begin{equation}  \label{e:def.pp.heavy2}
\regN_n^{(k)}(\cdot) = \sum_{\mathbf{i} \in \IPk} h_n(\mathcal{X}_{\bi}) \epsilon_{\bigl( R_{k,n}^{-1} \mathcal{X}_{\bi}\,, \ \mathbf{i} /|\mathcal{P}_n| \bigr)} (\cdot)\,,  
\end{equation}
and 
\begin{equation}  \label{e:def.pp.heavy1}
\tildeN_n^{(k)}(\cdot) = \sum_{\mathbf{i} \in \IPk} g_n(\mathcal{X}_{\bi}, \mathcal{P}_n) \epsilon_{\bigl( R_{k,n}^{-1} \mathcal{X}_{\bi}\,, \ \mathbf{i} /|\mathcal{P}_n| \bigr)} (\cdot)\,,  
\end{equation}
where $(R_{k,n}, \, n \geq 1)$ is asymptotically determined by 
\begin{equation}  \label{e:normalizing.heavy}
n^k r_n^{d(k-1)} R_{k,n}^d f(R_{k,n}e_1)^k \to 1 \ \ \text{as } n \to \infty\,,
\end{equation}
with $e_1 = (1,0,\dots,0)^{\prime}$. 

\begin{remark}  \label{re:no.rapid.decay.heavy} Note that \eqref{e:normalizing.heavy} implicitly precludes 
 very fast decay of $r_n$ to zero. In fact, if $n^kr_n^{d(k-1)} \to 0$ as $n \to \infty$, then \eqref{e:normalizing.heavy} implies that we must also have  that $R_{k,n}^d f(R_{k,n}e_1)^k \to \infty$, which contradicts \eqref{e:RV.tail}.  However, under the implicit restriction that such rapid convergence does not occur, \eqref{e:normalizing.heavy} effectively determines the $R_{k,n}$ as an intrinsic function of $k$ and $r_n$. 

 Note also that the case  $r_n \equiv 1$ was the one treated in \cite{adler:bobrowski:weinberger:2014}, although for a far more limited class of densities. In the case  $r_n\equiv 1$,  it is easy to check that  $R_{k,n} \in RV_{(\alpha-d/k)^{-1}}$ as a function of $n$.
  \end{remark} 

The convergence in law of $\regN_n^{(k)}$ and $\tildeN_n^{(k)}$ takes place in the space $M_p(E_k \times L_k)$, where $E_k = \bigl([-\infty,\infty]^d\bigr)^k \setminus \{ \mathbf{0} \}$ and $L_k$ is given in \eqref{e:upper.diagonal}. Let $N^{(k)}$ be a limiting Poisson random measure, whose intensity is, as in Section \ref{s:general}, denoted by $\nu_k \times \lambda_k$, where $\lambda_k$ is the $k$-dimensional Lebesgue measure concentrated on $L_k$. As a consequence  of the heavy tail of $f$, the measure $\nu_k$ also exhibits a 
power-law tail structure. More specifically,  for a rectangle $(a_0, b_0] \times \dots \times (a_{k-1}, b_{k-1}] \subset E_k = \bigl([-\infty,\infty]^d\bigr)^k \setminus \{ \mathbf{0} \}$, which  is bounded away from the origin, 
\begin{align} 
\nu_k \bigl( (&a_0, b_0] \times \dots \times (a_{k-1}, b_{k-1}] \bigr)  \label{e:def.nu.heavy} \\
&= \frac{1}{k!} \int_{(\bbr^d)^{k-1}} h(0,\by) d\by \int_{a_i  \prec x \preceq b_i\,, \ i=0,\dots,k-1} \|x\|^{-\alpha k} dx  \notag
\end{align}
where $\prec$  and $\preceq$ indicate  componentwise inequalities.

It is worth mentioning that $\nu_k$ exhibits the scaling property:
\begin{equation}  \label{e:homo.exp}
\nu_k(sA) = s^{-(\alpha k - d)} \nu_k(A)
\end{equation}
for all $s>0$ and measurable  $A$. We shall also need  one additional function. For 
  $\bx = (x_1,\dots,x_k) \in (\bbr^d)^k$, $r>0$,  we define
\begin{equation} 
 \label{e:comp.func}
p(\bx; r) \ \definedas \ \int_{\bigcup_{i=1}^k B(x_i; \, r)} f(z) \,dz.
\end{equation}

We now have all  we need to formulate  the following result. The proof is given in the Appendix. 

\begin{theorem}  \label{t:main.heavy}
Let $(X_i, \, i \geq 1)$ be a sequence of i.i.d. $\bbr^d$-valued spherically symmetric random variables with density $f$ having regularly varying tail as in \eqref{e:RV.tail}. Let $h_n\:(\bbr^d)^k \to \{ 0,1 \}$ be a sequence of  indicators as in  \eqref{e:indep}, satisfying the shift invariance of \eqref{e:invariance1} and the points in proximity condition \eqref{e:close.enough}. Then, the point processes $N_n^{(k)}$ and $\tildeN_n^{(k)}$ given by \eqref{e:def.pp.heavy2} and \eqref{e:def.pp.heavy1} weakly converge to $N^{(k)}$ in the space $M_p(E_k \times L_k)$. 
\end{theorem}

To make the implications of the theorem a little more transparent, consider the simple  case for which 
\begin{equation}  \label{e:heavy.example}
 f(x) \  = \  \frac{C}{ 1 + \|x\|^{\alpha} },
\end{equation}
  and assume that $r_n = n^{s}$ with $-k/d(k-1) < s \leq 0$. Then, \eqref{e:normalizing.heavy} reduces to $n^{k+sd(k-1)} C^k R_{k,n}^{d-\alpha k} \to 1$ and solving this  with respect to $R_{k,n}$ gives
  $$
   R_{k,n} = C^{(\alpha - d/k)^{-1}} n^{[1+sd(1-k^{-1})] / (\alpha - d/k)}, 
  $$ 
  which in turn implies $\dots << R_{k,n} << R_{k-1,n} << \dots << R_{2,n}$. 

Being more specific on the geometric side, consider the geometric graph example of  \eqref{e:geometric.graph}, where connected graphs $\Gamma_k$ with $k$ vertices are fixed for $k=2,3,\dots$. Then the theorem  implies  that 
 $\bbr^d$ can be divided into annuli of increasing radii, and the $\Gamma_k$, $k=2,3,\dots$ are (asymptotically)  distributed in these annuli in a very specific fashion. Letting $\text{Ann}(K,L)$ denote an annulus of outer radius $L$ and inner radius $K$, we have (in an asymptotic sense)
\begin{itemize}
\item Inside Ann$(R_{2,n},\infty)$, there are finitely many $\Gamma_2$, but  none of $\Gamma_3, \Gamma_4, \dots$.
\item Inside Ann$(R_{3,n},R_{2,n})$, there are infinitely many $\Gamma_2$ and finitely many $\Gamma_3$,  but none of 
 $\Gamma_4, \Gamma_5, \dots$.
\end{itemize}
In general, 
\begin{itemize}
\item Inside Ann$(R_{k,n},R_{k-1,n})$, there are infinitely many $\Gamma_2, \dots, \Gamma_{k-1}$ and finitely many $\Gamma_k$,  but none of  $\Gamma_{k+1}, \Gamma_{k+2}, \dots$.
\end{itemize}
\begin{figure}[!t]
  \label{f:layer1}
\includegraphics[width=12.5cm]{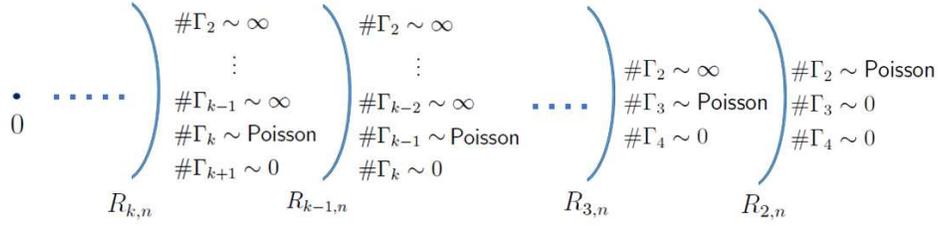}
\caption{Layer structure of annuli for random geometric graphs. For the density \eqref{e:heavy.example}, $R_{k,n}$ is a regularly varying sequence with exponent $[1+sd(1-k^{-1})] / (\alpha - d/k)$. The number of subgraphs isomorphic to $\Gamma_k$ outside $B(0; R_{k,n})$ is approximately Poisson.}
\end{figure}

\section{Light Tail Cases}  \label{s:light.tail}

This section treats  point process convergence when the underlying density on $\bbr^d$ possesses a relatively lighter tail than in the previous section. Typically, in the spirit of extreme value theory, light-tailed densities can be formulated by the so-called von Mises function. In particular, in a one-dimensional case ($d=1$), the von Mises function plays a decisive role in a characterisation of the max-domain of attraction of the Gumbel law. See Proposition 1.4 in \cite{resnick:1987}. The density formulation similar to our setup can be found, for example in \cite{balkema:embrechts:2004} and \cite{balkema:embrechts:2007}. 

Our results for the light tailed case fall into two categories. In the first, we have results more or less paralleling those of the previous section, although the normalisations and limiting distributions are somewhat different.    Nevertheless, the `annuli structure' in the preceding section carries through. This case is treated in Section \ref{subsec:light:annuli}.   Recall, however, that in the heavy tail case these annuli are rather thick, especially when compared the typical size of the geometric objects, which is of order $r_n$.  In the light tail scenario, this is not the case, and the annuli are actually quite thin, which leads to questions that   are not present in the heavy tail case. These questions are posed and solved  in Section \ref{subset:light:core}.

\subsection{Annuli} \label{subsec:light:annuli}
The light tail case considers $(X_i, \, i \geq 1)$ of  $\bbr^d$-valued i.i.d. spherically symmetric random variables with density given by
\begin{equation}  \label{e:density.light}
f(x) = L \bigl( \|x\| \bigr) \exp \bigl\{ -\psi \bigl( \|x\| \bigr) \bigr\}\,, \ \ x \in \bbr^d.
\end{equation}
Here, $\psi\: \bbr_+ \to \bbr$ is a $C^2$ function of von Mises type, in  that
$$
\psi^{\prime}(z) > 0, \quad \psi(z) \to \infty, \quad \bigl( 1/\psi^{\prime} \bigr)^{\prime}(z) \to 0
$$
as $z \to z_{\infty}$,  for some $z_\infty  \in (0,\infty]$. The main objects of this paper are the geometric objects far away from the origin, so the bounded support of the density (i.e., $z_{\infty} < \infty$) is relatively less interesting to us. Therefore, in what follows, we restrict to the unbounded support case, in which $z_{\infty} = \infty$. For notational convenience, we introduce  the function 
$$a(z) = \frac{1}{\psi^{\prime}(z)}.$$
 Since $a^{\prime}(z) \to 0$ as $z \to \infty$, the Ces\`{a}ro mean of $a^{\prime}$ converges as well. That is, 
$$
\frac{a(z)}{z} = \frac{1}{z} \int_0^z a^{\prime}(r)dr \to 0\,, \ \ \text{as } z \to \infty\,.
$$
Suppose that the function  $L:\bbr_+ \to \bbr_+$  in \eqref{e:density.light} is \textit{flat} for $a$; viz.
\begin{equation}  \label{e:flat}
\frac{L \bigl( t + a(t) v \bigr)}{L(t)} \to 1 \ \ \text{as } t \to \infty \ \text{uniformly on bounded } v\text{-sets}.
\end{equation}
This assumption implies that  $L$ is negligible in its tail, and hence the tail behaviour  of $f$ is determined  only by $\psi$. Here, we need to introduce an extra regularity condition on $L$. Assume that there exist $\gamma \geq 0$, $z_0 > 0$, and $C \geq 1$ such that 
\begin{equation}  \label{e:poly.upper}
\frac{L(zt)}{L(z)} \leq C\, t^{\gamma} \ \ \text{for all } t > 1, \, z \geq z_0\,.
\end{equation}
In view of Corollary 2.0.5 in \cite{bingham:goldie:teugels:1987} (or Theorem 2.0.1 there), 
\begin{equation}  \label{e:poly.growth}
\limsup_{z \to \infty} \frac{L(zt)}{L(z)} < \infty \ \ \text{for all } t > 1
\end{equation}
suffices for \eqref{e:poly.upper}. Observe that if $L$ is a polynomial function, then both \eqref{e:flat} and \eqref{e:poly.growth} are satisfied. On the other hand, \eqref{e:poly.growth} does not hold if $L$ grows exponentially.

Given a non-increasing sequence $(r_n, \, n \geq 1)$ of positive numbers (as usual constant $r_n$ is permissible), we define $h_n$ and $g_n$ as in the previous sections. Suppose that there exists a non-decreasing sequence $(R_{k,n}, \, n \geq 1)$ determined by 
\begin{equation}  \label{e:normalizing.light}
n^k r_n^{d(k-1)} \, a(R_{k,n})\, R_{k,n}^{d-1}\, f(R_{k,n}e_1)^k \to 1, \ \ \ n \to \infty\,.
\end{equation}
As pointed out in Remark \ref{re:no.rapid.decay.heavy}, note that a rapid decay of $(r_n)$, e.g., $n^k r_n^{d(k-1)} \to 0$, $n \to \infty$ is not permissible. 
For $k \geq 2$, using abbreviations \eqref{e:abb.X1}, \eqref{e:abb.X2}, and \eqref{e:abb.X3}, so that
$$
c_{k,n} = a(R_{k,n})\,, \quad \quad d_{k,n} = R_{k,n}\,,
$$ 
we shall consider the following point processes in the space $M_p(E_k \times L_k)$ with $E_k = \bigl((-\infty,\infty]^d\bigr)^k$ and $L_k$ given by  \eqref{e:upper.diagonal}: 
\begin{equation}  \label{e:def.pp.light2}
\regN_n^{(k)}(\cdot) = \sum_{\mathbf{i} \in \IPk} h_n(\mathcal{X}_{\bi}) \ \epsilon_{\bigl( \Xsupn_\bi, \ \mathbf{i} /|\mathcal{P}_n| \bigr)} (\cdot)\,.
\end{equation}
and 
\begin{equation}  \label{e:def.pp.light}
\tildeN_n^{(k)}(\cdot) = \sum_{\mathbf{i} \in \IPk} g_n(\mathcal{X}_{\bi}, \mathcal{P}_n) \ \epsilon_{\bigl( \Xsupn_\bi, \ \mathbf{i} /|\mathcal{P}_n| \bigr)} (\cdot)\,,
\end{equation}

As seen below in more detail, the nontriviality of weak limits of $N_n^{(k)}$ and $\tildeN_n^{(k)}$ mainly depends on the limit value of $a(R_{k,n})/r_n$. Indeed, if $a(R_{k,n})/r_n$ has a non-zero limit, then $N_n^{(k)}$ and $\tildeN_n^{(k)}$ both converge to the same Poisson random measure. On the contrary, if $a(R_{k,n})/r_n \to 0$, then $\regN_n^{(k)}$ goes to $0$ in probability (and of course, so does $\tildeN_n^{(k)}$), which implies that the geometric objects implicit in $h_n$ cannot be observed outside $B(0; R_{k,n})$, regardless of whether or not those objects are isolated from other points. 

 The following is the main limit theorem of this subsection. The proof is deferred to the Appendix. 

\begin{theorem}  \label{t:main.light}
Let $(X_i, \, i \geq 1)$ be a sequence of i.i.d. $\bbr^d$-valued spherically symmetric random variables with density $f$ provided in \eqref{e:density.light}. Let $h_n\:(\bbr^d)^k \to \{ 0,1 \}$ be a sequence of indicators as in  \eqref{e:indep}, satisfying the shift invariance of \eqref{e:invariance1} and the points in proximity condition \eqref{e:close.enough}. \\
$(i)$ If $a(R_{k,n})/r_n \to c \in (0,\infty]$ as $n \to \infty$, then the point processes $\regN_n^{(k)}$ and $\tildeN_n^{(k)}$ given by \eqref{e:def.pp.light2} and \eqref{e:def.pp.light} converge  weakly to a Poisson random measure with intensity $\nu_k \times \lambda_k$. Here, $\lambda_k$ is the $k$-dimensional Lebesgue measure concentrated on $L_k$, and the measure $\nu_k$ is given by 
\begin{align}
 \label{e:intensity.finite} 
\nu_k &\bigl( (a_0, b_0] \times \dots \times (a_{k-1}, b_{k-1}] \bigr) \\
&\quad \, = \frac{1}{k!} \int_{(\bbr^d)^{k-1}} \int_{\theta \succeq 0} \int_0^{\infty} \one \bigl(a_0 \prec \rho \theta \preceq b_0 \bigr) \notag  \\
&\qquad\qquad \, \times \one \Bigl( a_i \prec \bigl(\rho + c^{-1} \langle \theta,y_i \rangle \bigr)\theta \preceq b_i, \ i=1,\dots,k-1 \Bigr)  \notag  \\
&\qquad\qquad\qquad \, \times \exp \Bigl\{ -k\rho - c^{-1} \sum_{i=1}^{k-1} \langle \theta, y_i \rangle \Bigr\}\, h(0,\by)\, d\rho J(\theta) d\theta d\by,  \notag
\end{align}
where $a_i$, $b_i$, $i=0,\dots,k-1$ are $d$-dimensional real vectors with $-\infty \prec a_i \preceq b_i \preceq \infty$, and 
$J(\theta) = |\partial x / \partial \theta|$ is the Jacobian
\beq
\label{J:def}
J(\theta) \ = \  \sin^{k-2}(\theta_1) \, \sin^{k-3}(\theta_2) \, \dots\, \sin (\theta_{k-2}).
\eeq 
$(ii)$ Suppose $a(R_{k,n})/r_n \to 0$ as $n \to \infty$ and $a$ is eventually non-increasing (i.e., $a$ is non-increasing on $(x_0,\infty)$ for some large $x_0$). Then, the point processes $\regN_n^{(k)}$ and $\tildeN_n^{(k)}$ converge to $0$ in probability.  
\end{theorem}

As was the case in the heavy tailed scenario, Theorem \ref{t:main.light} $(i)$  again implies the layer structure of annuli at different radii $\dots << R_{k,n} << R_{k-1,n} << \dots << R_{2,n}$ in which different structures can be found. In the current case, however, this is not as immediate, since now the normalisation also involves a translation, the asymptotic effect of which is to count geometric events only when they occur in the nonnegative orthant. 

Recall that we are particularly interested in the distribution of geometric events outside the ball $B(0; R_{k,n})$, $k=2,3,\dots$, and this  is captured by
$$
\sum_{\bi \in \IPk} g_n(\mathcal{X}_{\bi}, \mathcal{P}_n)\, \one \bigl( \|X_{i_j}\| \geq R_{k,n}, \ j=1,\dots,k\, \bigr)\,.
$$
However, by the spherical symmetry of the underlying  density, along with the independence of the sample points,   the weak limit of this quantity is the same as that of 
$$
2^d \sum_{\bi \in \IPk} g_n(\mathcal{X}_{\bi}, \mathcal{P}_n)\, \epsilon_{\Xsupn_\bi} \Bigl( \bigl( [0,\infty]^d \bigr)^k \Bigr),
$$
the limit of which can be computed directly from Theorem \ref{t:main.light} $(i)$. In other words, the spherical symmetry of a density allows one to extend the basic result on the weak convergence occurring in the nonnegative orthant to that in all orthants. Consequently we can obtain the same qualitative separation of geometric events into distinct annuli that we saw in the heavy tailed case. 

To see how this works, consider the simple example  $f(x) = C e^{-\|x\|^{\tau}/\tau}$, for some $0 < \tau \leq 1$, and  take $r_n\equiv 1$. Then, $a(z) = z^{1-\tau}$ clearly has a non-zero limit as $z \to \infty$, and so Theorem \ref{t:main.light} $(i)$ applies. Then, 
\eqref{e:normalizing.light} becomes  $$n^k R_{k,n}^{d-\tau} C^k e^{-k R_{k,n}^{\tau} / \tau} \to 1,$$ and the solution 
\begin{equation}  \label{e:log.grow.ex}
 R_{k,n} = \bigl( \tau \log n + k^{-1}(d-\tau) \log (\tau \log n) + \tau \log C \bigr)^{1/\tau} 
\end{equation}
 grows only logarithmically in $n$, whereas the $R_{k,n}$ of the previous section exhibited, essentially, power law growth. 
Thus, the description inherent in Figure \ref{f:layer1} remains unchanged, except for the change in the values of $R_{k,n}$.

\begin{remark}
In particular, if $a(R_{k,n})/r_n \to \infty$, the limiting intensity measure $\nu_k$ can be simplified to 
\beq  \label{e:intensity.infty}
&&\nu_k \bigl( (a_0, b_0] \times \dots \times (a_{k-1}, b_{k-1}] \bigr) 
\\ &&\qquad \quad =\ \frac{1}{k!} \int_{(\bbr^d)^{k-1}} h(0,\by) d\by \iint_{a_i \prec \rho \theta \preceq b_i, \ i=0,\dots,k-1, \ \theta \succeq 0} \hspace{-20pt} e^{-k\rho} d\rho J(\theta) d\theta\,,   \notag
\eeq
where $a_i$, $b_i$, $i=0,\dots,k-1$ are $d$-dimensional vectors such that $-\infty \prec a_i \preceq b_i \preceq \infty$. 
\end{remark}
\begin{remark}
In the case of $k=1$, the point process $\regN_n^{(1)}$ is no longer associated with the indicator $h_n$ and is given by
$$
\regN_n^{(1)}(\cdot) = \sum_{i=1}^{|\mathcal{P}_n|} \epsilon_{\bigl(  a(R_{1,n})^{-1} (X_i - R_{1,n} S(X_i))\,, \ i /|\mathcal{P}_n| \bigr)} (\cdot)\,.
$$
Interestingly, a simpler argument than the proof of Theorem \ref{t:main.light} (or combining  standard arguments of point process theory --  e.g.\ Chapters 5 and 6 of \cite{resnick:2007}  -- with the Palm theory in the Appendix) shows that regardless of whether $a(R_{1,n})$ has a zero or a non-zero limit, $\regN_n^{(1)}$  converges weakly  to a Poisson random measure with a non-trivial intensity. 
\end{remark}

Before concluding this subsection, we look at  several examples of light-tailed densities for which the corresponding point processes exhibit  different limiting behaviours. 
\begin{example}
Suppose,  for simplicity, that $r_n \equiv 1$ and consider the following probability densities on $\bbr^d$: 
\begin{align*}
f_1(x) &= L_1 \bigl( \|x\| \bigr) e^{-\|x\|^{\tau}}\,, \ \ 0 < \tau < 1\,, \ \ x \in \bbr^d\,,  \\     
f_2(x) &= L_2 \bigl( \|x\| \bigr) e^{-\|x\| \log \log \|x\| / \log \|x\|}\,,  \ \ x \in \bbr^d\,,
\end{align*}
where the $L_1$ and $L_2$  satisfy \eqref{e:flat} and \eqref{e:poly.upper}. The densities $f_1$ and $f_2$ are usually referred to as subexponential densities, since their tails decay more slowly than that of  an exponential distribution. For  both $f_1$ and $f_2$, it is easy to check that $a(z) \to \infty$ as $z \to \infty$, so, by Theorem \ref{t:main.light} $(i)$, the point processes $\regN_n^{(k)}$ and $\tildeN_n^{(k)}$ weakly converge to a Poisson random measure with intensity $\nu_k \times \lambda_k$, where $\nu_k$ is given by \eqref{e:intensity.infty}. 
On the other hand, if the density has the same tail as an exponential distribution; for example,
$$
f_3(x) = L_3 \bigl( \|x\| \bigr) e^{-\|x\|}\,, \ \ x \in \bbr^d\,,
$$
then $\psi(z) = z$ and $a(z) = 1$. In this case, $\regN_n^{(k)}$ and $\tildeN_n^{(k)}$ once again weakly converge to a Poisson random measure. However, its intensity measure $\nu_k \times \lambda_k$ is more complicated, where $\nu_k$ is given by \eqref{e:intensity.finite} with $c=1$. We shall also consider the densities with more rapidly decaying tails than an exponential distribution (they are sometimes called superexponential densities). Two examples of superexponential densities are 
\begin{align*}
f_4(x) &= L_4 \bigl( \|x\| \bigr)e^{-\|x\|^{\tau}}\,, \ \ \tau > 1\,, \ \ x \in \bbr^d\,, \\     
f_5(x) &= L_5 \bigl( \|x\| \bigr) e^{-\|x\| \log \|x\| / \log \log \|x\|}\,, \ \ x \in \bbr^d\,.
\end{align*}
For $f_4$ and $f_5$, it follows that $a(z) \to 0$ as $z \to \infty$, and so Theorem \ref{t:main.light} $(ii)$ implies  that the point process $\regN_n^{(k)}$ goes to $0$ in probability. The densities $f_2$, $f_3$, and $f_5$ differ only slightly in their tail behaviours, but the corresponding point processes possess totally different limits. Finally, we point out that even for $f_4$ and $f_5$, if one chooses $(r_n, \, n \geq 1)$ so that $a(R_{k,n})/r_n \to c \in (0,\infty]$, then the point process can converge to a non-trivial Poisson random measure. 
\end{example}

\subsection{At the annuli boundaries}
\label{subset:light:core}
The claim of Part $(ii)$  of  Theorem \ref{t:main.light}  is that,   when $a(R_{k,n})/r_n\to 0$,  the geometric objects being counted do not exist outside of the ball $B(0; R_{k,n})$, at least from the view of the point process convergence. We next want to explore the existence of the same objects inside $B(0; R_{k,n})$, under the condition that these objects must be isolated  from other random points by at least $r_n$. This question was partially and negatively answered in \cite{adler:bobrowski:weinberger:2014}, in the framework of the asymptotics of the expected Betti numbers of the \v{C}ech complexes associated with a random sample and a unit radius. In particular it was shown there that for the standard Gaussian distribution, all the expected Betti numbers of order $k \geq 1$ vanish and the resulting \v{C}ech complex becomes contractible. In what follows, we continue working on the same question  under the conditions of Theorem \ref{t:main.light} $(ii)$ from a more comprehensive viewpoint. 

Theorem \ref{two:theorem} establishes the existence of two sequences of balls with different radii,  the smaller ones ultimately containing so many   points that they can be covered by a union of balls with radius $r_n$ and centred on the points. On the other hand, ultimately there are no  points outside a larger balls. The main point, however, is that  the differences between the radii of the two balls decays at the rate of $o(r_n)$. We conclude, therefore, that, at least when $n$ is large, a union of balls with radius $r_n$ centred at points in $\mathcal{P}_n$ becomes contractible, and accordingly, the geometric objects that we have been studying up until now, which are isolated from other points, fail to  exist anywhere in all of 
$\bbr^d$.

In order to get a clear picture, we shall  add extra assumptions to the setup of Theorem \ref{t:main.light} $(ii)$. In particular, we assume that 
\begin{align}
\psi &\in RV_{v} \ \ \text{for some } 1 < v < \infty\,,  \label{e:reg.psi}  \\[5pt]
L &\equiv C \ (\text{suitable normalising constant}).  \label{e:const.L}
\end{align}
and 
\begin{equation}  \label{e:reg.r}
(r_n) \text{ is a regularly varying sequence which decreases to $0$ as } n \to \infty\,.  
\end{equation}
The reason why we  need  $v>1$ in \eqref{e:reg.psi} is as follows. By Proposition 2.5 in \cite{resnick:2007}, which establishes the  regular variation of the derivative of a regularly varying function,  we now have 
$$
a(z) = 1/\psi^{\prime}(z) \in RV_{1-v}\,.
$$
The setup of Theorem \ref{t:main.light} $(ii)$ requires that  $a(z) \to 0$ as $z \to \infty$, so the regular variation exponent of $\psi$ cannot be less than $1$. Finally we observe that $\psiinv \in RV_{1/v}$ and $a \circ \psiinv \in RV_{(1-v)/v}$, and both functions are eventually monotone.
\begin{theorem}
\label{two:theorem}
Assume the conditions of Theorem \ref{t:main.light} $(ii)$, as well as  \eqref{e:reg.psi}, \eqref{e:const.L}, and \eqref{e:reg.r}. Furthermore, assume that 
\begin{equation}  \label{e:stronger.version}
\frac{a\circ \psiinv (\log n)}{r_n}\, \log \log n \to 0 \ \ \text{as } n \to \infty\,.
\end{equation}
Then, there exist two sequences $(R_n^{(0)}, \, n\geq 1)$ and $(R_n^{(1)}, \, n\geq 1)$ such that, as $n \to \infty$,
$$
\P \Biggl\{ \, B\bigl(0; R_n^{(0)} \bigr) \subset \hspace{-10pt} \bigcup_{X \in \mathcal{P}_n \cap B(0; R_n^{(0)})} \hspace{-15pt} B(X; r_n)\,, \ \ \ \mathcal{P}_n \cap B\bigl( 0; R_n^{(1)} \bigr)^c = \emptyset \, \Biggr\} \to 1\,,
$$
and
\begin{equation}  \label{e:diff.radii}
r_n^{-1} \bigl( R_n^{(1)} - R_n^{(0)} \bigr) \to 0\,.
\end{equation}
\end{theorem}
\begin{figure}[!t]
  \label{f:thm}
\includegraphics[width=7cm]{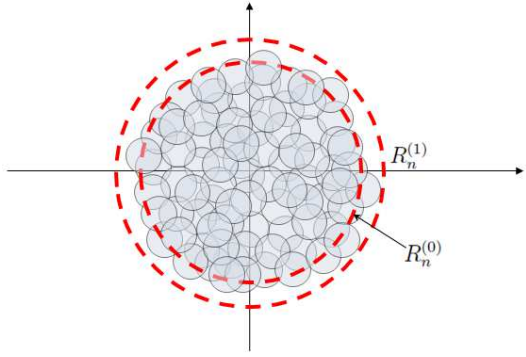}
\caption{ Asymptotically, the smaller ball of radius $R_n^{(0)}$ is covered by a union of balls with radius $r_n$ centred around the points in $\Pn$. There are no  points outside a larger ball with radius $R_n^{(1)}$. The difference between $R_n^{(1)}$ and $R_n^{(0)}$ vanishes at the rate of $o(r_n)$. }
\end{figure}
\begin{remark}
Typically, the solution to \eqref{e:normalizing.light} is given by $R_{k,n} = \psiinv(\log n + b_{k,n})$, where $b_{k,n}/\log n \to 0$ as $n \to \infty$; (cf.\ \eqref{e:log.grow.ex}.) Because of the regular variation of $a \circ \psiinv$, the condition $a(R_{k,n})/r_n \to 0$ in Theorem \ref{t:main.light} $(ii)$ is equivalent to $a \circ \psiinv(\log n)/r_n \to 0$. In essence, therefore,  assuming \eqref{e:stronger.version} adds a stronger condition to Theorem \ref{t:main.light} $(ii)$. We have not been able to determine whether or not the same result  can be obtained when $a \circ \psiinv(\log n)/r_n \to 0$ holds but \eqref{e:stronger.version} is no longer true.
\end{remark}
\begin{proof}
First of all, we claim that $(r_n) \in RV_0$. To see this, note that it will suffice to show that $\log r_n / \log n \to 0$ as $n \to \infty$, since $(r_n)$ is a regularly varying sequence. (cf.\  Proposition 2.6 $(i)$ in \cite{resnick:2007}.) Since $a \circ \psiinv \in RV_{(1-v)/v}$, it is clear that $a \circ \psiinv (\log n) \log n \to \infty$, and so, 
$$
\frac{\log r_n}{\log n} = \frac{a \circ \psiinv (\log n)}{r_n}\, \frac{r_n \log r_n}{a \circ \psiinv (\log n) \log n} \to 0\,, \ \ \ n \to \infty\,.
$$

For positive numbers $g$ and $\rho$, let $\mathcal{Q}_n(g, \rho)$ be a family of cubes with grid $gr_n$ that are contained in $B(0; \rho)$. Fix $g>0$ sufficiently  small so that  
$$
\bigl\{ \, Q \cap \mathcal{P}_n \neq \emptyset \ \, \text{for all } Q \in \mathcal{Q}_n(g, \rho)\, \bigr\} \subset \Bigl\{ \, B\bigl(0; \rho \bigr) \subset \bigcup_{X \in \mathcal{P}_n \cap B(0; \rho)} B(X; r_n) \, \Bigr\}
$$
for all $\rho>0$ and $n \geq 1$. We define $R_n^{(0)}$ and $R_n^{(1)}$ as follows:  
$$
R_n^{(0)} = \psiinv(A_n)\,, \ \ A_n = \log n + d \log r_n - \log \log r_n^{-1} \psiinv (\log n) - \delta\,,
$$
where $\delta$ is a positive constant such that
\begin{equation}  \label{e:rest.delta}
d - e^{\delta} g^d C < 0\,,
\end{equation}
and 
$$
R_n^{(1)} = \psiinv(B_n),
$$
where
$$ B_n = \log n + (d-1) \log \psiinv (\log n) + \log a \circ \psiinv (\log n) + \log \log n.
$$
We need to prove that, as $n \to \infty$,
\begin{equation}  \label{e:1st.event}
\P \bigl\{ \, Q \cap \mathcal{P}_n = \emptyset \ \, \text{for some } Q \in \mathcal{Q}_n(g, R_n^{(0)})\, \bigr\} \to 0\,, 
\end{equation}
and
\begin{equation}  \label{e:2nd.event}
\P \Bigl\{ \mathcal{P}_n \cap B \bigl( 0; R_n^{(1)} \bigr)^c = \emptyset \Bigr\} \to 1\,.
\end{equation}
The probability in \eqref{e:1st.event} is estimated from above by 
\beq
\sum_{Q \in \mathcal{Q}_n(g, R_n^{(0)})} \P\{Q \cap \mathcal{P}_n = \emptyset \} 
&=& \sum_{Q \in \mathcal{Q}_n(g, R_n^{(0)})} \exp \Bigl\{ -n \int_Q f(x) dx \Bigr\}
\notag  \\
\label{e:inner.circle}
&\leq& \sum_{Q \in \mathcal{Q}_n(g, R_n^{(0)})} \exp \bigl\{ -n (g r_n)^d f(R_n^{(0)}e_1) \bigr\}\\
 &\leq& \left( \frac{R_n^{(0)}}{gr_n} \right)^d \exp \bigl\{ -g^d nr_n^d f(R_n^{(0)}e_1) \bigr\}.  \notag
\eeq
By virtue of the inequality $R_n^{(0)} \leq \psiinv (\log n)$ and \eqref{e:rest.delta}, we have, as $n \to \infty$,
$$
d \log r_n^{-1} R_n^{(0)} - g^d nr_n^d f(R_n^{(0)}e_1) \leq ( d-e^{\delta} g^d C) \log r_n^{-1} \psiinv (\log n) \to -\infty
$$
from which the rightmost term in \eqref{e:inner.circle} vanishes as $n \to \infty$. 

Next, we   turn to proving \eqref{e:2nd.event}. Since 
$$
\P \Bigl\{ \mathcal{P}_n \cap B \bigl( 0; R_n^{(1)} \bigr)^c = \emptyset \Bigr\} = \exp \left\{ - n \int_{||x|| \geq R_n^{(1)}} f(x) dx \right\}\,,
$$
it is enough to show $n \int_{||x|| \geq R_n^{(1)}} f(x) dx \to 0$, as $n \to \infty$. By the polar coordinate transform with $J(\theta) = |\partial x / \partial \theta|$, we can write
\begin{align*}
n \int_{||x|| \geq R_n^{(1)}} f(x) dx &= s_{d-1} n a \bigl(R_n^{(1)}\bigr) \bigl( R_n^{(1)} \bigr)^{d-1} f \bigl( R_n^{(1)}e_1 \bigr) \\
&\quad \, \times \int_0^{\infty} \left( 1 + \frac{a(R_n^{(1)})}{R_n^{(1)}} \rho \right)^{d-1} \frac{f \Bigl( \bigl( R_n^{(1)} + a(R_n^{(1)}) \rho \bigr) e_1 \Bigr)}{f\bigl(R_n^{(1)}e_1 \bigr)} d\rho\,,
\end{align*}
where $s_{d-1}$ is a surface area of the $(d-1)$-dimensional unit sphere in $\bbr^d$. The dominated convergence theorem guarantees that the integral above converges to $\int_0^{\infty} e^{-\rho} d\rho = 1$. Therefore, we only have to verify
that
$$
n a \bigl(R_n^{(1)}\bigr) \bigl( R_n^{(1)} \bigr)^{d-1} e^{-\psi \bigl(R_n^{(1)}\bigr)} \to 0\,, \ \ \ \text{as } n \to \infty\,.
$$
Substituting $R_n^{(1)} = \psiinv(B_n)$, we have
$$
n a \bigl(R_n^{(1)}\bigr) \bigl( R_n^{(1)} \bigr)^{d-1} e^{-\psi \bigl(R_n^{(1)}\bigr)} = \frac{a \circ \psiinv(B_n)}{a \circ \psiinv (\log n)}\, \left( \frac{\psiinv (B_n)}{\psiinv (\log n)} \right)^{d-1} (\log n)^{-1}\,.
$$
Since $B_n / \log n \to 1$, it follows from the uniform convergence of regularly varying functions (cf.\ Proposition 2.4 in \cite{resnick:2007}) that 
$$
\frac{a \circ \psiinv(B_n)}{a \circ \psiinv (\log n)} \to 1\,, \ \ \ \frac{\psiinv (B_n)}{\psiinv (\log n)}  \to 1\,.
$$
So the proof of \eqref{e:2nd.event} is complete. 

It  remains to establish \eqref{e:diff.radii}. The mean value theorem yields
\begin{align*}
r_n^{-1} \bigl(R_n^{(1)} - R_n^{(0)} \bigr) &= r_n^{-1} (\psiinv)^{\prime} (t_n)\, (B_n - A_n)  \\
&= \frac{a \circ \psiinv(t_n)}{a \circ \psiinv(\log n)}\, \frac{a \circ \psiinv(\log n)}{r_n}\, (B_n - A_n)\,,
\end{align*}
where $t_n$ lies in between $A_n$ and $B_n$. Since $A_n/\log n \to 1$ and $B_n/\log n \to 1$, we have $t_n/\log n \to 1$, and thus, 
$$
\frac{a \circ \psiinv (t_n)}{a \circ \psiinv(\log n)} \to 1 \ \ \text{as } n \to \infty\,.
$$
To finish the argument, we have to establish the following three limits:
\begin{align}
&\frac{a \circ \psiinv(\log n)}{r_n}\, \log r_n \to 0\,,  \label{e:finish1} \\[5pt]
&\frac{a \circ \psiinv(\log n)}{r_n}\, \log a \circ \psiinv(\log n) \to 0\,,  \label{e:finish2} \\[5pt]
&\frac{a \circ \psiinv(\log n)}{r_n} \log \psiinv (\log n) \to 0\,. \label{e:finish3}
\end{align}
Since $\log a \circ \psiinv(\log n) < \log r_n$ for sufficiently large $n$, \eqref{e:finish1} is implied by \eqref{e:finish2}. By virtue of the regular variation of $a \circ \psiinv$ and $\psiinv$, we have that
$$
\left( \frac{\log a \circ \psiinv (\log n)}{\log \log n} \right) \ \text{and } \left( \frac{\log \psiinv (\log n)}{\log \log n} \right)
$$
are bounded sequences. Both \eqref{e:finish2} and \eqref{e:finish3}  now follow from
\eqref{e:stronger.version}, and so we are done.
 \end{proof}

\section{Applications}  \label{s:application}

\subsection{Limit Theorems for Betti Numbers}
\label{subsec:Betti}

The results of the previous two sections show the existence of a sequences of annuli containing  different kinds of geometric objects. This subsection will further examine this layer structure and the topological properties of the objects they include,  relying on the notion of Betti numbers to quantify the topology. Our aim is to derive  limit theorems for the Betti numbers of the \v{C}ech complex built over $\mathcal{P}_n = \{ X_1,\dots,X_{N_n} \}$, where $(X_i, \, i \geq 1)$ is an i.i.d sample drawn from a spherically symmetric distribution, and $N_n=|\mathcal{P}_n|$ is a Poisson random variable with mean $n$ and is independent of $(X_i, \, i \geq 1)$. For $k \geq 3$, we take $h_n\:(\bbr^d)^k \to \{ 0,1 \}$ as in \eqref{e:betti.number}; viz. 
$$
h_n(x_1,\dots,x_k) = \one \Bigl( \beta_{k-2} \left( \check{C} \bigl( \{x_1,\dots,x_k\}, r_n \bigr) \right) = 1 \Bigr)\,, \ \ \ x_1,\dots,x_k \in \bbr^d,
$$
and define $g_n\: (\bbr^d)^k \to \{ 0,1 \}$ as in \eqref{e:def.gn}. 
Also, we define
$$
\widehat{S}_{k,n} = \sum_{\bi \in \IPk} g_n(\mathcal{X}_{\bi}, \mathcal{P}_n)\, \one \bigl( \|X_{i_j}\| \geq R_{k,n}, \ j=1,\dots,k\, \bigr)
$$
for $\mathcal{X}_{\bi} = (X_{i_1},\dots,X_{i_k})$ with $\bi = (i_1,\dots,i_k) \in \IPk$. The definition of  $R_{k,n}$ depends on whether the underlying density has a heavy tail or a light tail. Specifically, if the underlying density has a regularly varying tail as in \eqref{e:RV.tail}, then \eqref{e:normalizing.heavy} defines the $R_{k,n}$, while \eqref{e:normalizing.light} determines the $R_{k,n}$ if the density is given by \eqref{e:density.light}. We are interested in the behaviour of the Betti numbers 
$$\beta_{k-2}\Bigl( \check{C} \bigl( \mathcal{P}_n \cap B(0; R_{k,n})^c, r_n \bigr) \Bigr).
$$
 \begin{figure}[!t]
  \label{f:betti-cech}
\includegraphics[width=8cm]{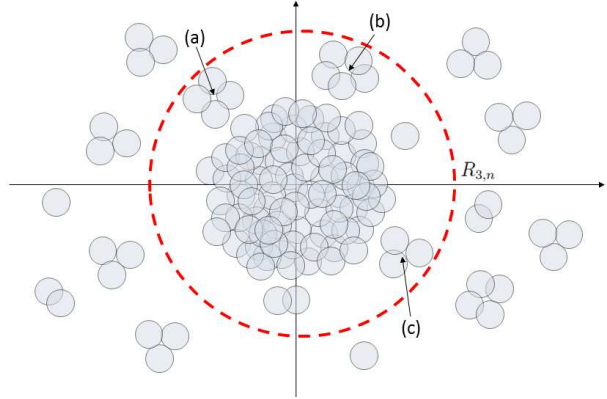}
\caption{{ For $k=3$, $d=2$, the Betti number $\beta_{1}\Bigl( \check{C} \bigl( \mathcal{P}_n \cap B(0; R_{3,n})^c, r_n \bigr) \Bigr)$ counts (one-dimensional) cycles outside of  $B(0; R_{3,n})$, while ignoring cycles inside the ball (e.g.\ (a), (b), and (c)). }}
\end{figure}
 The most relevant study to this subsection is \cite{adler:bobrowski:weinberger:2014},  in which the asymptotics of the expected Betti numbers were discussed.
 We, however, go well beyond this, by establishing Poisson limits for  the Betti numbers. As originally given in \cite{adler:bobrowski:weinberger:2014}, we shall  provide one useful inequality to elucidate the relation between $\widehat{S}_{k,n}$ above and the Betti numbers:
$$
\widehat{S}_{k,n} \leq \beta_{k-2}\Bigl( \check{C} \bigl( \mathcal{P}_n \cap B(0; R_{k,n})^c, r_n \bigr) \Bigr) \leq \widehat{S}_{k,n} + L_{k,n},
$$
where
$$
L_{k,n} = \hspace{-10pt}\sum_{\bi \in \mathcal I_{|\Pn|, k+1}} \widetilde{h}_n (\X_\bi)\, \one \bigl( \, \|X_{i_j}\| \geq R_{k,n}, \ j=1,\dots,k+1 \, \bigr),
$$
with 
$$
\widetilde{h}_n (x_1,\dots,x_{k+1}) = \one \bigl( \, \check{C}( \{ x_1,\dots,x_{k+1} \}, r_n ) \ \text{is connected}\, \bigr).  
$$
The limit theorem below demonstrates that $L_{k,n}$ tends to zero in probability, and as a result, $\widehat{S}_{k,n}$ and $\beta_{k-2}\bigl( \check{C} \bigl( \mathcal{P}_n \cap B(0; R_{k,n})^c, r_n \bigr) \bigr)$ asymptotically coincide.

\begin{theorem}  \label{c:betti.heavy}
Under the assumptions and notations of Theorem \ref{t:main.heavy}, 
$$
\beta_{k-2}\Bigl( \check{C} \bigl( \mathcal{P}_n \cap B(0; R_{k,n})^c, r_n \bigr) \Bigr) \Rightarrow Poi \left( \frac{s_{d-1}}{(\alpha k - d)k!}\, \int_{(\bbr^d)^{k-1}}  h(0,\by) d\by \right)
$$
where $s_{d-1}$ is the surface area of the $(d-1)$-dimensional unit sphere in $\bbr^d$. 
\end{theorem}

\begin{proof}
First of all,  note that the variable $\widehat{S}_{k,n}$ defined above can also be written, in terms of the point process  \eqref{e:def.pp.heavy1}, as 
\begin{align*}
\widehat{S}_{k,n} = &\sum_{\bi \in \IPk} g_n (\mathcal{X}_{\bi}, \mathcal{P}_n)\, \epsilon_{R_{k,n}^{-1} \mathcal{X}_{\bi}} \Bigl( \bigl\{ \bx \in (\bbr^d)^k: \|x_i\| \geq 1, \, \ i=1,\dots, k \bigr\} \Bigr)\,.
\end{align*}
Appealing to Theorem \ref{t:main.heavy}, 
\begin{align*}
\widehat{S}_{k,n} &\Rightarrow Poi\, \Bigl( \nu_k \bigl\{ \bx \in (\bbr^d)^k: \|x_i\| \geq 1, \, \ i=1,\dots, k \bigr\} \Bigr)  \\     
&=Poi \left( \frac{s_{d-1}}{(\alpha k - d)k!}\, \int_{(\bbr^d)^{k-1}}  h(0,\by) d\by \right) \ \ \ \text{in } \bbr_+\,.
\end{align*}
If we can now show that $L_{k,n} \to 0$ in probability, then the previous line suffices to prove the theorem. 

To this end, note that, in view of Theorem \ref{t:main.heavy}, the point process
\begin{equation}  \label{e:dummy.pp}
\sum_{\bi \in \mathcal I_{|\Pn|, k+1}} \widetilde{h}_n(\X_\bi)\, \epsilon_{R_{k+1,n}^{-1} \X_\bi }\, (\cdot)
\end{equation}
converges to a non-trivial Poisson random measure. Also, it is easy to see from \eqref{e:normalizing.heavy} that $R_{k,n} / R_{k+1,n} \to \infty$ as $n \to \infty$. Therefore, replacing the scaling constants $R_{k+1,n}$ in \eqref{e:dummy.pp} by $R_{k,n}$, the corresponding point process converges to zero in probability. Thus $L_{k,n} \stackrel{p}{\to} 0$ follows,  as required.
\end{proof}

\begin{theorem}  
Assume the conditions of Theorem \ref{t:main.light}. \\
$(i)$ \  If $a(R_{k,n})/r_n \to c \in (0,\infty]$ as $n \to \infty$, then 
$$
\beta_{k-2}\Bigl( \check{C} \bigl( \mathcal{P}_n \cap B(0; R_{k,n})^c, r_n \bigr) \Bigr) \Rightarrow 
$$
$$
Poi \Biggl( \frac{1}{k!}\, \int_{(\bbr^d)^{k-1}} \int_{S^{d-1}} 
\int_{\rho \geq 0\,, \ \rho + c^{-1} \langle \theta, y_i \rangle \geq 0\,, \ i=1,\dots,k-1} \hspace{-60pt} e^{ -k\rho - c^{-1} \sum_{i=1}^{k-1} \langle \theta,y_i \rangle }\, h(0,\by)\, d\rho J(\theta) d\theta  d\by \Biggr)\,.
$$
$(ii)$ \ If $a(R_{k,n})/r_n \to 0$ as $n \to \infty$, then 
$$
\beta_{k-2}\Bigl( \check{C} \bigl( \mathcal{P}_n \cap B(0; R_{k,n})^c, r_n \bigr) \Bigr) \stackrel{p}{\to} 0\,.
$$
\end{theorem}

\begin{proof}
For the proof of $(i)$, note first that  Theorem \ref{t:main.light} $(i)$, implies that
\begin{equation}   \label{e:transform.betti}
\sum_{\bi \in \IPk} g_n (\mathcal{X}_{\bi}, \mathcal{P}_n)\, \epsilon_{\Xsupn_\bi} \Bigl( \bigl( [0,\infty]^d \bigr)^k \Bigr)  \Rightarrow Poi\, \biggl( \nu_k \Bigl\{ \bigl( [0,\infty]^d \bigr)^k \Bigr\} \biggr) .
\end{equation}
Due to the symmetry of the integral with respect to $\theta \in S^{d-1}$, 
\begin{align*}
&\nu_k \Bigl\{ \bigl( [0,\infty]^d \bigr)^k \Bigr\}  \\
&= \frac{1}{k!}\, \int_{(\bbr^d)^{k-1}} \int_{\theta \succeq 0} \int_{\rho \geq 0\,, \ \rho + c^{-1} \langle \theta, y_i \rangle \geq 0\,, \ i=1,\dots,k-1} \hspace{-60pt} e^{ -k\rho - c^{-1} \sum_{i=1}^{k-1} \langle \theta,y_i \rangle }\, h(0,\by)\, d\rho J(\theta) d\theta  d\by  \\      
&=\frac{1}{2^d k!}\, \int_{(\bbr^d)^{k-1}} \int_{S^{d-1}} \int_{\rho \geq 0\,, \ \rho + c^{-1} \langle \theta, y_i \rangle \geq 0\,, \ i=1,\dots,k-1} \hspace{-60pt} e^{ -k\rho - c^{-1} \sum_{i=1}^{k-1} \langle \theta,y_i \rangle }\, h(0,\by)\, d\rho J(\theta) d\theta  d\by.
\end{align*}
{ Further, due to the spherical symmetry of the density, the left-hand side in \eqref{e:transform.betti} has the same weak limit as $2^{-d} \widehat{S}_{k,n}$.}
Consequently, $\widehat{S}_{k,n}$ weakly converges to 
$$
 Poi \left( \frac{1}{k!}\, \int_{(\bbr^d)^{k-1}} \int_{S^{d-1}} \int_{\rho \geq 0\,, \ \rho + c^{-1} \langle \theta, y_i \rangle \geq 0\,, \ i=1,\dots,k-1} \hspace{-10pt} \hspace{-60pt} e^{ -k\rho - c^{-1} \sum_{i=1}^{k-1} \langle \theta,y_i \rangle }\, h(0,\by)\, d\rho J(\theta) d\theta  d\by \right)\,.
$$
To complete the proof of $(i)$ it now suffices to verify  that $L_{k,n} \to 0$ in probability as $n \to \infty$, which follows along the same lines as in the proof of Theorem \ref{c:betti.heavy}. 

To show $(ii)$, note  that 
\begin{align*}
&\beta_{k-2}\Bigl( \check{C} \bigl( \mathcal{P}_n \cap B(0; R_{k,n})^c, r_n \bigr) \Bigr) \\
&\leq \sum_{\bi \in \IPk} h_n (\mathcal{X}_{\bi})\, \one \bigl( \, \|X_{i_j}\| \geq R_{k,n}\,, \ j=1,\dots,k\, \bigr)  + L_{k,n}\,.
\end{align*}
 The above expression has the same weak limit as 
$$
2^d \sum_{\bi \in \IPk} h_n (\mathcal{X}_{\bi})\, \epsilon_{\Xsupn_\bi} \Bigl( \bigl( [0,\infty]^d \bigr)^k \Bigr) + L_{k,n}\,.
$$ 
Once again, $L_{k,n} \to 0$ in probability, and so the assertion follows from Theorem \ref{t:main.light} $(ii)$. 
\end{proof}

\subsection{Limit Theorems for Partial Maxima}

We now turn to describing the limiting behaviour of the maximum distance from the origin of the random points constituting the geometric objects that we have been studying so far. 
More specifically, we  consider the `maxima process'  
\begin{equation}   \label{e:comp.maxima}
\bigvee_{\substack{1 \leq i_1 < \dots < i_k \leq |\mathcal{P}_n|t\,, \\[2pt] g_n(\mathcal{X}_{\bi}, \mathcal{P}_n) = 1}} \hspace{-10pt} \frac{\| X_{i_1} \|  - d_{k,n}}{c_{k,n}}\,, \ \ \ t \in [0,1]\,,
\end{equation}
where $a \vee b = \max \{  a, b\}$ for $a, b \in \bbr$, and $c_{k,n}>0$ and $d_{k,n} \in \bbr$ are the normalising sequences  of the previous sections. 

For a concrete example, suppose that $h_n\:(\bbr^d)^k \to \{ 0,1 \}$ is defined as in \eqref{e:geometric.graph}, where $\Gamma$ is a connected graph with $k$ vertices. Then, a $k$-tuple of random points can contribute to the maxima process, only if  its points serve as the 
 vertices of graph isomorphic to  $\Gamma$.
 
 Note that   the process \eqref{e:comp.maxima} only requires the computation of the maximum of $\|X_{i_1}\|$ (suitably scaled and centred). However, in view of \eqref{e:close.enough}, all the components in $\mathcal{X}_{\bi} = (X_{i_1},\dots,X_{i_k})$ must be close to each other. Therefore, the results below will be robust as to which component is chosen from $\mathcal{X}_{\bi}$. 
 
We start with the regularly varying tail case, which is essentially a corollary of Theorem 
\ref{t:main.heavy}. The limit in this case
  is a \textit{time-scaled extremal Fr\'echet process} (cf.\ \cite{owada:samorodnitsky:2015b}). The main difference between this and a classical extremal Fr\'echet process is that the former exhibits dependence in the max-increments, while the latter does not. 

\begin{theorem}  \label{c:heavy.max}
Under the assumptions of Theorem \ref{t:main.heavy}, let $(j_l, s_l)$ represent the points of a Poisson random measure with mean measure $\widetilde{\nu}_k \times \widetilde{\lambda}_k$, where $\widetilde{\nu}_k(A) = \nu_k \bigl(A \times (\bbr^d)^{k-1} \bigr)$ for a measurable set $A \subset \bbr^d$ ($\nu_k$ is defined in \eqref{e:def.nu.heavy}), and $\widetilde{\lambda}_k(B) = \lambda_k \bigl( \{(z_1,\dots,z_k) \in L_k\,, \, z_k \in B  \} \bigr)$ for measurable $B \subset [0,1]$. Then, 
$$
\bigvee_{\substack{1 \leq i_i < \dots < i_k \leq |\mathcal{P}_n| t\,, \\[2pt]  g_n ( \mathcal{X}_{\bi}, \mathcal{P}_n ) = 1}} \frac{\|X_{i_1}\|}{R_{k,n}} \Rightarrow \bigvee_{s_l \leq t} \|j_l\|\,, \qquad \text{in } D [0,1]\,,
$$
where $D [0,1]$ is the space of right continuous functions from $[0,1]$ into $\bbr$ with left limits. The limiting process is a time-scaled extremal $(\alpha k -d)$-Fr\'echet process with finite-dimensional laws determined as follows: for $0 = t_0 \leq t_1 < \dots < t_K \leq 1$, $\eta_i \geq 0$, $i=1,\dots,K$,
\begin{align} 
\P  &\Bigl( \, \bigvee_{s_l \leq t_i} \|j_l\| \leq \eta_i\,, \ i=1,\dots,K \Bigr)  \label{e:max.stable.fidi} \\
&= \exp \biggl\{ -\frac{s_{d-1}}{(k!)^2 (\alpha k -d)} \int_{(\bbr^d)^{k-1}} \hspace{-5pt} h(0,\by) d\by \sum_{i=1}^K (t_i^k - t_{i-1}^k)\, \Bigl( \bigwedge_{i \leq  j \leq K} \eta_j \Bigr)^{-(\alpha k -d)} \biggr\}\,. \notag
\end{align}
\end{theorem}
\begin{proof}
Restricting the domain of the point process convergence  in Theorem \ref{t:main.heavy}, we have
\begin{align*}
\sum_{\bi \in \IPk} &g_n(\mathcal{X}_{\bi}, \mathcal{P}_n)\, \epsilon_{\bigl( R_{k,n}^{-1} X_{i_1}, \ i_k/|\mathcal{P}_n| \bigr)}\, (\cdot)  \\
&\Rightarrow \sum_{l} \epsilon_{(j_l, s_l)}\, (\cdot) \ \ \ \text{in } M_p \Bigl( \bigl([-\infty,\infty]^d \setminus \{ 0 \} \bigr) \times [0,1] \Bigr)\,. 
\end{align*}
The functional $T\: M_p\bigl( \bigl([-\infty,\infty]^d \setminus \{ 0 \} \bigr) \times [0,1] \bigr) \to D [0,1]$ defined by $T (\sum_l \epsilon_{(z_l,\tau_l)}) = \bigvee_{\tau_l \leq \cdot}\|z_l\|$, is almost surely continuous. (cf.\  p214 of \cite{resnick:1987}.) Applying the continuous mapping theorem immediately gives the required weak convergence. 

What remains is to establish the precise form of the limit process, as in \eqref{e:max.stable.fidi}. To show this, note first that 
\begin{align*}
\P  &\Bigl( \, \bigvee_{s_l \leq t_i} \|j_l\| \leq \eta_i\,, \ i=1,\dots,K \Bigr)   \\
&= \exp \biggl\{ - \sum_{i=1}^K \tilde\nu_k \Bigl( \bigl\{ z\ \in \bbr^d: \| z\| > \bigwedge_{i \leq j \leq K} \eta_j \bigr\} \Bigr)\, \tilde\lambda_k \bigl( (t_{i-1}, t_i] \bigr)  \biggr\}\,.
\end{align*}
By  \eqref{e:def.nu.heavy} and \eqref{e:homo.exp} we have that
\begin{align*}
\tilde\nu_k &\Bigl( \bigl\{ z\ \in \bbr^d: \| z\| > \bigwedge_{i \leq j \leq K} \eta_j \bigr\} \Bigr) \\
&= \frac{s_{d-1}}{k! (\alpha k -d)} \int_{(\bbr^d)^{k-1}} \hspace{-5pt} h(0,\by) d\by\, \Bigl( \bigwedge_{i \leq  j \leq K} \eta_j \Bigr)^{-(\alpha k -d)},
\end{align*}
and 
$$
\tilde\lambda_k \bigl( (t_{i-1}, t_i] \bigr) = (k!)^{-1} (t_i^k - t_{i-1}^k)\,,
$$
from which \eqref{e:max.stable.fidi} now follows. 
\end{proof}

We now turn to the case of light-tailed densities. Note firstly that a similar, but simpler, argument than the proof of Theorems \ref{t:general.thm} and \ref{t:main.light} shows that under the conditions in Theorem \ref{t:main.light} $(i)$, the point process
$$
 \sum_{  \bi\in \IPk} g_n( \Xi,  \Pn) \, \epsilon_{\Bigl( \bigl( a(R_{k,n})^{-1} \bigl( \| X_{i_j} \| - R_{k,n} \bigr), \, j=1,\dots,k \bigr),\, \bi/ |\Pn|\Bigr)}(\cdot)
$$
 converges weakly to a Poisson random measure with mean measure $\mu_k \times \lambda_k$. As usual,  $\lambda_k$ is the $k$-dimensional Lebesgue measure concentrated on $L_k$, and $\mu_k$ is given by 
\begin{align*}
\mu_k &\bigl( (a_0, b_0] \times \dots \times (a_{k-1}, b_{k-1}] \bigr) \\
&= \frac{1}{k!} \int_{(\bbr^d)^{k-1}} \int_{S_{d-1}} \int_{a_0}^{b_0} \one \bigl( a_i < \rho + c^{-1} \langle \theta,y_i \rangle \leq b_i, \ i=1,\dots,k-1 \bigr)   \\
&\qquad \, \times \exp \Bigl\{ -k\rho - c^{-1} \sum_{i=1}^{k-1} \langle \theta, y_i \rangle \Bigr\}\, h(0,\by)\, d\rho J(\theta) d\theta d\by, 
\end{align*}
where $a_i, b_i$, $i=0,\dots,k-1$ are one-dimensional real vectors with $-\infty < a_i \leq b_i \leq \infty$, and $J(\theta) = |\partial x / \partial \theta|$ is the Jacobian. 
Exploiting this result and mimicking the proof of Theorem \ref{c:heavy.max}, one can prove the following result, in which the 
limit is well described as a \textit{time-scaled extremal Gumbel process}. 

\begin{theorem}
Assume the conditions in Theorem \ref{t:main.light}, and let $(j_l, s_l)$ be the points of a Poisson random measure with mean measure $\widetilde{\mu}_k \times \widetilde{\lambda}_k$, where $\widetilde{\mu}_k(A) = \mu_k \bigl(A \times (\bbr^d)^{k-1} \bigr)$ for a measurable set $A \subset \bbr$ and $\widetilde{\lambda}_k(B) = \lambda_k \bigl( \{(z_1,\dots,z_k) \in L_k\,, \, z_k \in B  \} \bigr)$ for measurable $B \subset [0,1]$. \\
$(i)$ \ If $a(R_{k,n})/r_n \to c \in (0,\infty]$ as $n \to \infty$, then 
\begin{equation}  \label{e:max.light}
\bigvee_{\substack{1 \leq i_i < \dots < i_k \leq |\mathcal{P}_n| t\,,  \\[2pt] g_n ( \mathcal{X}_{\bi}, \mathcal{P}_n ) = 1}} \hspace{-10pt} \frac{\|X_{i_1}\|-R_{k,n}}{a(R_{k,n})} \Rightarrow \bigvee_{s_l \leq t} j_l \ \ \ \text{in } D [0,1]\,.  
\end{equation}
The finite dimensional laws of the limiting process are  as follows: for $0=t_0 \leq t_1 < \dots < t_K \leq 1$, $\eta_i \in \bbr$, $i=1,\dots,K$,  
\begin{align*} 
\P  &\Bigl( \, \bigvee_{s_l \leq t_i} j_l \leq \eta_i\,, \ i=1,\dots,K \Bigr) \\
&= \exp \biggl\{ -\frac{1}{(k!)^2 k}  \int_{(\bbr^d)^{k-1}} \int_{S_{d-1}} e^{-c^{-1} \sum_{i=1}^{k-1} \langle \theta, y_i \rangle} h(0,\by) J(\theta) d\theta d\by \\
&\quad \qquad \times \sum_{i=1}^K (t_i^k - t_{i-1}^k)\, e^{-k \bigwedge_{i \leq j \leq K}\eta_j} \biggr\}\,. 
\end{align*}
$(ii)$ \ If $a(R_{k,n})/r_n \to 0$ as $n \to \infty$, then the left-hand side in \eqref{e:max.light} converges to $0$ in probability. 
\end{theorem}

\subsection{Limit Theorems for Partial Sums}

Continuing on from the previous subsection, we now consider limit theorems on partial sums. In particular, we shall  focus on a stable limit case. As seen in a variety of related studies such as \cite{davis:hsing:1995} and \cite{davis:resnick:1985}, proving stable limit theorems via a point process approach has become the gold  standard. In order to obtain stable limits, however, the underlying random variables constructing partial sums must have infinite second moments. For this reason, we shall  assume a regularly varying tail for  the underlying density, and further, 
that the homogeneity exponent $\alpha k - d$ in \eqref{e:homo.exp}  lies in the interval $(0,2)$.  Combining this constraint with $\alpha > d$ and $k \geq 2$, we need to treat only the case $1 < \alpha < 1.5$, $k=2$, and $d=1$.

\begin{theorem}  \label{t:clt}
Under the conditions of Theorem \ref{t:main.heavy}, assume that $1< \alpha < 1.5$, $k=2$, and $d=1$. Suppose additionally that an indicator $h\:\bbr^2 \to \{0,1 \}$ is not only shift invariant as in \eqref{e:invariance1} but also symmetric in the sense that
\begin{equation}  \label{e:rot.inv}
h(x_1,x_2) = h(-x_1,-x_2) \ \ \text{for all } x_1,x_2 \in \bbr\,.
\end{equation}
Then, $R_{2,n}^{-1} \sum_{\bi \in \IP2} g_n(\mathcal{X}_{\bi}, \mathcal{P}_n) X_{i_1}$  converges  weakly to a symmetric $(2\alpha-1)$-stable law. 
\end{theorem}

\begin{proof}
Restricting the domain of point process convergence shown in Theorem \ref{t:main.heavy}, we find that in the space $M_p \bigl([-\infty,\infty] \setminus \{ 0 \} \bigr)$, 
$$\sum_{\bi \in \IP2} g_n(\mathcal{X}_{\bi}, \mathcal{P}_n)\, \epsilon_{R_{2,n}^{-1} X_{i_1}} (\cdot)
$$  converges weakly to a Poisson random measure with intensity $\widetilde{\nu}_2$, where $\widetilde{\nu}_2(A) = \nu_2 (A \times \bbr)$ for measurable $A \subset \bbr$. 

Under our parameter restrictions, the homogeneity exponent in \eqref{e:homo.exp} is $2\alpha - 1$, and thus the limiting Poisson random measure can be represented in law by 
$$\sum_{j=1}^{\infty} \epsilon_{C_{\alpha} r_j \Gamma_j^{-(2\alpha - 1)^{-1}}},$$
 where $(r_j, \, j \geq 1)$ is a sequence of i.i.d. Rademacher random variables taking $+1$ and $-1$ with probability $1/2$, $\Gamma_j$ is the $j$-th jump time of a unit rate Poisson process, and 
$$
C_{\alpha} = \left( \frac{1}{2\alpha-1} \int_{\bbr} h(0,y)dy \right)^{1/(2\alpha-1)}.
$$ 
Here, $(r_j)$ and $(\Gamma_j)$ are taken to be independent. Notice that due to its symmetry, $C_{\alpha} \sum_{j=1}^{\infty} r_j \Gamma_j^{-(2\alpha-1)^{-1}}$ converges almost surely and has a symmetric $(2\alpha-1)$-stable 
distribution. For more information about series representation of stable laws, see Section 1.4 of \cite{samorodnitsky:taqqu:1994}. It is known that, for every $\delta>0$, the functional 
$T_{\delta}: M_p \bigl( [-\infty,\infty] \setminus \{ 0 \} \bigr) \to \bbr$
defined by 
$$
T_{\delta} \bigl( \sum_{l} \epsilon_{z_l} \bigr) = \sum_l z_l\, \one \bigl( |z_l| > \delta \bigr)
$$
is almost surely continuous. (cf.\ Section 7.2.3 in \cite{resnick:2007}.) 
Applying the continuous mapping theorem, we have, as $n \to \infty$, 
\begin{align*}
R_{2,n}^{-1} &\sum_{\bi \in \IP2} g_n \bigl(\mathcal{X}_{\bi}, \mathcal{P}_n \bigr) X_{i_1}\, \one \bigl(\, |X_{i_1}| > R_{2,n} \delta \bigr)  \\
&\Rightarrow C_{\alpha} \sum_{j=1}^{\infty} r_j \Gamma_j^{-(2\alpha-1)^{-1}}\, \one (\, C_{\alpha} \Gamma_j^{-(2\alpha-1)^{-1}} > \delta ).  
\end{align*}
As $\delta \downarrow 0$, we have 
$$
C_{\alpha} \sum_{j=1}^{\infty} r_j \Gamma_j^{-(2\alpha-1)^{-1}}\, \one ( \, C_{\alpha} \Gamma_j^{-(2\alpha-1)^{-1}} > \delta ) \Rightarrow C_{\alpha} \sum_{j=1}^{\infty} r_j \Gamma_j^{-(2\alpha-1)^{-1}}. 
$$
Hence, it remains to show that, for every $\eta > 0$,
$$
\lim_{\delta \downarrow 0} \limsup_{n \to \infty} \P  \biggl\{ \, \Bigl| \sum_{\bi \in \IP2}  g_n\bigl( \mathcal{X}_{\bi},\mathcal{P}_n \bigr) X_{i_1}\, \one \bigl(\, |X_{i_1}| \leq R_{2,n} \delta \bigr) \Bigr| > \eta R_{2,n} \biggr\} = 0
$$
However, by the Cauchy-Schwarz inequality, this will follow immediately if we can show that
\begin{equation}  \label{e:2nd.order.conv}
\lim_{\delta \downarrow 0} \limsup_{n \to \infty} R_{2,n}^{-2}\, \E  \Bigl\{ \sum_{\bi \in \IP2}  g_n \bigl( \mathcal{X}_{\bi},\mathcal{P}_n \bigr) X_{i_1}\, \one \bigl(\, |X_{i_1}| \leq R_{2,n} \delta \bigr) \Bigr\}^2 = 0.
\end{equation}

We can write
\begin{align*}
&R_{2,n}^{-2}\, \E  \Bigl\{ \sum_{\bi \in \IP2}  g_n \bigl( \mathcal{X}_{\bi},\mathcal{P}_n \bigr) X_{i_1}\, \one \bigl(\, |X_{i_1}| \leq R_{2,n} \delta \bigr) \Bigr\}^2  \\     
&= R_{2,n}^{-2}\, \sum_{m=2}^{\infty} \P  \bigl\{ |\mathcal{P}_n| = m \bigr\}\, \E  \Bigl\{ \sum_{\bi \in \mathcal{I}_{m,2}}  g_n \bigl(\mathcal{X}_{\bi}\,, \mathcal{X}_m \bigr) X_{i_1}\, \one \bigl(\, |X_{i_1}| \leq R_{2,n} \delta \bigr) \Bigr\}^2\,,
\end{align*}
where $\mathcal{X}_m = \{ X_1,\dots,X_m \}$ for $m \geq 2$. 

Now introduce a triangular array of i.i.d. Rademacher variables $\bigl(r_{\bi}, \, 1 \leq i_1 < i_2 < \infty \bigr)$, which are independent of $(X_i, \, i \geq 1)$. Then, by virtue of the symmetry of  the $X_i$, it follows that, for all $n \geq 1$ and $m \geq 2$,
\begin{align*}
\Bigl( g_n ( \mathcal{X}_{\bi},\mathcal{X}_m )\, &X_{i_1} \one \bigl( |X_{i_1}| \leq R_{2,n} \delta \bigr), \, \bi \in \mathcal{I}_{m,k} \Bigr)  \\     
&\stackrel{d}{=} \Bigl( g_n ( r_{\bi} \mathcal{X}_{\bi}\,, \,  r_{\bi} \mathcal{X}_m )\, r_{\bi} X_{i_1} \one \bigl( |X_{i_1}| \leq R_{2,n} \delta \bigr), \, \bi \in \mathcal{I}_{m,k} \Bigr)\,.
\end{align*}
We now observe that 
\begin{align*}
\E  &\Bigl\{ \sum_{\bi \in \mathcal{I}_{m,2}}  g_n \bigl(\mathcal{X}_{\bi}\,, \mathcal{X}_m \bigr) X_{i_1}\, \one \bigl(\, |X_{i_1}| \leq R_{2,n} \delta \bigr) \Bigr\}^2  \\     
&= \E  \Bigl\{ \sum_{\bi \in \mathcal{I}_{m,2}}  g_n \bigl(r_{\bi} \mathcal{X}_{\bi}\,, \, r_{\bi} \mathcal{X}_m \bigr)\, r_{\bi} X_{i_1}\, \one \bigl(\, |X_{i_1}| \leq R_{2,n} \delta \bigr) \Bigr\}^2  \\
&=\sum_{\bi \in \mathcal{I}_{m,2}} \sum_{\bj \in \mathcal{I}_{m,2}} \E  \biggl\{ X_{i_1} X_{j_1}\, \one \bigl(\, |X_{i_1}| \leq R_{2,n} \delta\,, \ |X_{j_1}| \leq R_{2,n} \delta \, \bigr)  \\     
&\quad \, \times \E  \Bigl\{ r_{\bi} r_{\bj}\, g_n ( r_{\bi} \mathcal{X}_{\bi}\,, \, r_{\bi} \mathcal{X}_m )\, g_n ( r_{\bj} \mathcal{X}_{\bj}\,, \, r_{\bj} \mathcal{X}_m ) \, \Bigl| \, \mathcal{X}_m \Bigr\} \biggr\}
\end{align*}
If $\bi \neq \bj$, $r_{\bi}$ and $r_{\bj}$ are independent, and so 
\begin{align*}
\E  &\Bigl\{ r_{\bi} r_{\bj}\, g_n ( r_{\bi} \mathcal{X}_{\bi}\,, \, r_{\bi} \mathcal{X}_m )\, g_n ( r_{\bj} \mathcal{X}_{\bj}\,, \, r_{\bj} \mathcal{X}_m ) \, \Bigl| \, \mathcal{X}_m \Bigr\}  \\     
&=  \E  \Bigl\{ r_{\bi}\, g_n ( r_{\bi} \mathcal{X}_{\bi}\,, \, r_{\bi} \mathcal{X}_m )\, \Bigl| \, \mathcal{X}_m \Bigr\}\, \E  \Bigl\{ r_{\bj}\, g_n ( r_{\bj} \mathcal{X}_{\bj}\,, \, r_{\bj} \mathcal{X}_m )\, \Bigl| \, \mathcal{X}_m \Bigr\} \\     
&= 4^{-1}\Bigl( g_n (\mathcal{X}_{\bi}\,, \, \mathcal{X}_m ) - g_n (-\mathcal{X}_{\bi}, -\mathcal{X}_m )  \Bigr)\, \Bigl(  g_n (\mathcal{X}_{\bj}\,, \, \mathcal{X}_m ) - g_n (-\mathcal{X}_{\bj}, -\mathcal{X}_m )  \Bigr) \\     
&=0\,,
\end{align*}
where the last equality follows from \eqref{e:rot.inv}. 

This implies that all cross terms will vanish, and so
\begin{align*}
 \E  &\Bigl\{ \sum_{\bi \in \mathcal{I}_{m,2}}  g_n \bigl(\mathcal{X}_{\bi}\,, \mathcal{X}_m \bigr) X_{i_1}\, \one \bigl(\, |X_{i_1}| \leq R_{2,n} \delta \bigr) \Bigr\}^2 \\
&\qquad =\ \sum_{\bi \in \mathcal{I}_{m,2}} \E  \Bigl\{ g_n ( r_{\bi} \mathcal{X}_{\bi}\,, \, r_{\bi} \mathcal{X}_m )\, X_{i_1}^2\, \one \bigl(\, |X_{i_1}| \leq R_{2,n} \delta \bigr) \Bigr\}  \\     
&\qquad = \ \sum_{\bi \in \mathcal{I}_{m,2}} \E  \Bigl\{ g_n ( \mathcal{X}_{\bi}, \mathcal{X}_m )\, X_{i_1}^2\, \one \bigl(\, |X_{i_1}| \leq R_{2,n} \delta \bigr) \Bigr\} \\     
&\qquad  \leq\ {m \choose  2} \E \Bigl\{ h_n(X_1,X_2)\, X_1^2\, \one \bigl( |X_1| \leq R_{2,n} \delta \bigr) \Bigr\}\,.
\end{align*}
Finally, we have that
\begin{align*}
R_{2,n}^{-2}\, &\E  \Bigl\{ \sum_{\bi \in \IP2} g_n (\mathcal{X}_{\bi}, \mathcal{P}_n)\, X_{i_1} \one \bigl( |X_{i_1}| \leq R_{2,n} \delta \bigr) \Bigr\}^2 \\     
&\leq R_{2,n}^{-2}\, \sum_{m=2}^{\infty} \P  \bigl\{ |\mathcal{P}_n| = m \bigr\} \begin{pmatrix} m \\ 2 \end{pmatrix} \E  \Bigl\{ h_n(X_1,X_2)\, X_1^2\, \one \bigl( |X_1| \leq R_{2,n} \delta \bigr) \Bigr\} \\     
&= \frac{n^2}{2 R_{2,n}^2}\, \int_{\bbr^2} h_n(x_1, x_2)\, x_1^2\, \one \bigl( |x_1| \leq R_{2,n} \delta \bigr) f(x_1) f(x_2) dx_1 dx_2  \\     
&= \frac{n^2 r_n}{R_{2,n}^2} \int_0^{R_{2,n} \delta} x^2 f(x) \int_{\bbr} f\bigl(x |1 + r_n y/x|\bigr)\, h(0,y) dy dx  \\     
&\leq \frac{n^2 r_n}{R_{2,n}^2} \int_0^{R_{2,n} \delta} x^2 f(x)^2 \int_{\bbr} \sup_{m \geq 1} \frac{f\bigl(x |1 + r_m y/x|\bigr)}{f(x)}\, h(0,y) dy dx\,.
\end{align*}
Here, the equality in the fourth line is obtained by the change  of variables $x_1 \leftrightarrow x$, $x_2 \leftrightarrow x + r_n y$. Since $(r_n)$ is a bounded sequence, it follows from the uniform convergence of regularly varying functions of negative exponent (cf.\ Proposition 2.4 in \cite{resnick:2007}) that
$$
\sup_{m \geq 1} \frac{f\bigl(x |1 + r_m y/x|\bigr)}{f(x)} \to 1\,, \qquad  \text{as } x \to \infty,
$$
uniformly in $y \in \bbr$ with $h(0,y)=1$. 

Thus, as $x \to \infty$,
$$
U(x) \equiv \int_{\bbr} \sup_{m \geq 1} \frac{f\bigl(x |1 + r_m y/x|\bigr)}{f(x)}\, h(0,y) dy \to \int_{\bbr} h(0,y) dy \in (0,\infty)\,.
$$
This fact implies that $x^2 f(x)^2 U(x)$ is regularly varying with exponent $2-2\alpha$. By an application of Karamata's theorem (e.g.\ Theorem 2.1 in \cite{resnick:2007}),
\begin{align*}
\frac{n^2 r_n}{R_{2,n}^2} \int_0^{R_{2,n} \delta} x^2 f(x)^2 U(x) dx 
&\sim \frac{n^2 r_n}{3-2\alpha} R_{2,n} \delta^3 f(R_{2,n} \delta)^2 U(R_{2,n} \delta) \\     
&\sim n^2 r_n R_{2,n} f(R_{2,n})^2\, \frac{\delta^{3-2\alpha}}{3-2\alpha}\, \int_{\bbr} h(0,y)dy \\     
&\to \frac{\delta^{3-2\alpha}}{3-2\alpha}\, \int_{\bbr} h(0,y)dy \ \ \ \text{as } n \to \infty\,.
\end{align*}
Since $\delta^{3-2\alpha} \to 0$ as $\delta \downarrow 0$, \eqref{e:2nd.order.conv} follows. 
\end{proof}

\section{Appendix}

\subsection{Proof of Theorem \ref{t:general.thm}}

For the completion of Theorem \ref{t:general.thm}, we need several important ingredients, all of which belong to the  `Palm theory' of Poisson point processes. They are useful when computing expectations related to Poisson point processes. We recall abbreviations \eqref{Xk:def}, \eqref{Xi:def}, together with \eqref{In:def}. 

\begin{lemma}(Palm theory for Poisson point processes, \cite{arratia:goldstein:gordon:1989}; see also Theorem 1.6 in \cite{penrose:2003})  \label{l:palm.penrose}
Let $(X_i, \, i \geq 1)$ be i.i.d. $\bbr^d$-valued random variables with common density $f$, and for $n \geq 1$, let $\mathcal{P}_n$ denote a Poisson point process with intensity $n f$. Let $u_n(\mathcal{Y}\,, \mathcal{X})$ be a measurable function defined for $\mathcal{Y}=(y_1,\dots,y_k)$, $y_i \in \bbr^d$ and a finite subset $\mathcal{X} \supset \mathcal{Y}$ of $d$-dimensional real vectors. Then, 
$$
\E  \biggl\{ \sum_{\bi \in \IPk} u_n(\mathcal{X}_{\bi}\,, \mathcal{P}_n) \biggr\} = \frac{n^k}{k!}\, \E  \bigl\{ u_n ( \X_k\,, \, \X_k \cup \mathcal{P}_n^{\prime} ) \bigr\}\,,
$$
where $\mathcal{P}_n^{\prime}$ is an independent copy of $\mathcal{P}_n$. 
\end{lemma}

\begin{lemma}  \label{l:palm.lebesgue1}
Under the same notation as Lemma \ref{l:palm.penrose}, let $u_n(\mathcal{Y})$ be a measurable function defined for $\mathcal{Y} = (y_1,\dots,y_k)$, $y_i \in \bbr^d$. Let $\lambda_k$ be the $k$-dimensional Lebesgue measure concentrated on the upper diagonal part of the unit cube; viz.\ $L_k = \bigl\{ (z_1,\dots,z_k) \in [0,1]^k: 0 \leq z_1 \leq \dots \leq z_k \leq 1 \bigr\}$. Then, for a measurable set $B \subset L_k$, 
$$
\E \biggl\{ \sum_{\bi \in \IPk} u_n (\mathcal{X}_{\bi}) \, \one \bigl( \bi/|\mathcal{P}_n| \in B \, \bigr) \biggr\} \sim \frac{n^k}{k!} \E \bigl\{ u_n(\X_k)\bigr\} \, \lambda_k(B),  \ \ \ \text{as}\ n \to \infty\,.
$$
\end{lemma}

\begin{proof}
\begin{align*}
&\E  \biggl\{ \sum_{\bi \in \IPk} u_n (\mathcal{X}_{\bi}) \, \one \bigl( \bi/|\mathcal{P}_n| \in B \, \bigr) \biggr\} \\     
&= \sum_{m=k}^{\infty} \P  \bigl\{ |\mathcal{P}_n| = m \bigr\} \E  \biggl\{ \sum_{\bi \in \IPk} u_n(\mathcal{X}_{\bi}) \, \one \bigl( \bi/|\mathcal{P}_n| \in B \, \bigr) \Bigl| \ |\mathcal{P}_n| = m \biggr\}  \\      
&= \frac{n^k}{k!}\, \E \bigl\{ u_n(\X_k) \bigr\} \sum_{m=k}^{\infty} e^{-n} \frac{n^{m-k}}{(m-k)!} \begin{pmatrix} m \\ k \end{pmatrix}^{-1} \# \bigl\{ \bi \in \mathcal{I}_{m,k} \cap mB \bigr\}\,.
\end{align*}
The proof then follows from the fact that  
$$
\begin{pmatrix} m \\ k \end{pmatrix}^{-1} \# \bigl\{ \bi \in \mathcal{I}_{m,k} \cap mB \bigr\} \to \lambda_k(B) \ \ \text{as } m \to \infty\,.
$$
\end{proof}

\begin{lemma}  \label{l:palm.lebesgue2}
Under the same notation as Lemma \ref{l:palm.penrose}, let $u_n(\mathcal{Y})$ be a non-negative measurable function defined for all finite subsets $\mathcal{Y} = (y_1,\dots,y_k)$, $y_i \in \bbr^d$. Suppose further that $n^k \E u_n(\X_k) \to C$, $n \to \infty$ for some constant $C>0$. Then, for a measurable set $B \subset L_k$, 
$$
\E  \, \biggl| \sum_{\bi \in \IPk} u_n(\mathcal{X}_{\bi}) \, \one \bigl( \bi/|\mathcal{P}_n| \in B \bigr) - \sum_{\bi \in \mathcal{I}_{n,k}} u_n(\mathcal{X}_{\bi}) \, \one \bigl( \bi/n \in B \bigr)  \biggr| \to 0\,, \ \ \ n \to \infty\,.
$$
\end{lemma}

\begin{proof}
We can proceed as follows. 
\begin{align*}
&\E  \, \biggl| \sum_{\bi \in \IPk} u_n(\mathcal{X}_{\bi}) \, \one \bigl( \bi/|\mathcal{P}_n| \in B \bigr) - \sum_{\bi \in \mathcal{I}_{n,k}} u_n(\mathcal{X}_{\bi}) \, \one \bigl( \bi/n \in B \bigr)  \biggr| \\        
&= \sum_{m=0}^{\infty} \P  \bigl\{ |\mathcal{P}_n| = m \bigr\} \Bigl|\, \#\bigl\{ \bi \in \mathcal{I}_{m,k} \cap mB \bigr\} - \#\bigl\{ \bi \in \mathcal{I}_{n,k} \cap nB \bigr\} \, \Bigr| \, \E \bigl\{ u_n(\X_k) \bigr\} \\     
&\sim C n^{-k} \sum_{m=0}^{\infty} \P  \bigl\{ |\mathcal{P}_n| = m \bigr\} \, \Bigl|\, \#\bigl\{ \bi \in \mathcal{I}_{n,k} \cap nB \bigr\} - \#\bigl\{ \bi \in \mathcal{I}_{m,k} \cap mB \bigr\} \, \Bigr| \\     
&\leq C n^{-k} \sum_{m=0}^{\infty} \P  \bigl\{ |\mathcal{P}_n| = m \bigr\} \, \Bigl|\,  \#\bigl\{ \bi \in \mathcal{I}_{n,k} \cap nB \bigr\} - \frac{n^k}{k!} \lambda_k(B) \, \Bigr| \\     
&\quad \, + C n^{-k} \sum_{m=0}^{\infty} \P  \bigl\{ |\mathcal{P}_n| = m \bigr\} \, |\, n^k - m^k | \, \frac{\lambda_k(B)}{k!} \\     
&\quad \, + C n^{-k} \sum_{m=0}^{\infty} \P  \bigl\{ |\mathcal{P}_n| = m \bigr\} \, \Bigl|\, \frac{m^k}{k!} \lambda_k(B) - \#\bigl\{ \bi \in \mathcal{I}_{m,k} \cap mB \bigr\}  \, \Bigr| \\     
&\equiv J_1 + J_2 + J_3\,.
\end{align*}
Evidently, $J_1 \to 0$ as $n \to \infty$. It is also easy to check that 
$$
J_2 = C \, \frac{\lambda_k(B)}{k!} \, \E  \, \left| \left(\frac{|\mathcal{P}_n|}{n}\right)^k - 1 \right| \to 0 \ \ \text{as } n \to \infty\,.
$$
As for $J_3$, given $\epsilon>0$, there is an integer $N \geq 1$ such that 
$$
\Bigl|\, \frac{k!}{m^k}\, \#\bigl\{ \bi \in \mathcal{I}_{m,k} \cap mB \bigr\} -  \lambda_k(B) \, \Bigr| < \epsilon \ \ \text{for all } m \geq N\,.
$$
Then, there exists a constant $C^{\prime} \geq C$ such that for all $m \geq N$, 
$$
J_3 \leq C^{\prime}n^{-k} \sum_{m=0}^{N-1} \P  \bigl\{ |\mathcal{P}_n| = m \bigr\} \, \frac{m^k}{k!} + C \epsilon n^{-k} \sum_{m=N}^{\infty} \P  \bigl\{ |\mathcal{P}_n| = m \bigr\} \, \frac{m^k}{k!} 
$$
Obviously, the first term on the right-hand side vanishes. The second term is bounded by 
$C \epsilon (k!)^{-1} n^{-k} \E  |\mathcal{P}_n|^k$. Since $\sup_{n \geq 1} n^{-k} \E  |\mathcal{P}_n|^k < \infty$ and $\epsilon$ is arbitrary, we conclude that $J_1 + J_2 + J_3 \to 0$ as $n \to \infty$. 
\end{proof}

Before starting the proof of Theorem \ref{t:general.thm}, we note that we shall often use, without further comment, the fact that, for every $k \geq 1$, 
$$
{n \choose  k} \  \sim \ \frac{n^k}{k!}\,, 
$$
in the sense that the ratio of the two sides tends to 1 as $n\to\infty$.

\begin{proof}[Proof of Theorem \ref{t:general.thm}]
For the convergence of $\regN_n^{(k)}$, according to Kallenberg's theorem (see Proposition 3.22 in \cite{resnick:1987}), it suffices to show that
\beq
\E\left\{\regN_n^{(k)}(R)\right\} \ \to \ \E\left\{N^{(k)}(R)\right\},  \label{e:Kallen1} 
\eeq
and
\beq
\P \left\{ \regN_n^{(k)}(R) = 0 \right\} \ \to \ \P \left\{ N^{(k)}(R) = 0 \right\},\label{e:Kallen2}
\eeq
for every disjoint union of measurable sets of the form $R=\bigcup_{p=1}^m (A_p \times B_p)$, where $A_p$ is relatively compact in $E_k$ with $\nu_k(\partial A_p) = 0$ (i.e.\ the boundary of $A_p$ has $\nu_k$-measure $0$), and $B_p$ is a measurable set in $L_k$. 

For the proof of \eqref{e:Kallen1}, we can set, without loss of generality, $m=1$, and write $A=A_1$, $B=B_1$. By virtue of Lemma \ref{l:palm.lebesgue1} in the Appendix, together with \eqref{e:conv.in.exp},
\begin{align*}
\E  \bigl\{ \regN_n^{(k)} (A \times B) \bigr\} &\sim \frac{n^k}{k!} \P  \Bigl\{ h_n(\X_k) = 1\,, \ \ \Xsupnk \in A \ \Bigr\} \lambda_k(B)  \\
&\to \nu_k(A) \lambda_k(B) = \E \bigl\{ N^{(k)}(A \times B) \bigr\}\,.
\end{align*}

Now we proceed to \eqref{e:Kallen2}. Setting 
$$
\xi_{\bi, n} = h_n(\mathcal{X}_{\bi}) \ \epsilon_{\bigl( \Xsupn_\bi, \ \mathbf{i} /n \bigr)} (R)\,, \ \ \bi =\{ i_1,\dots,i_k \} \in \mathcal{I}_{n,k}\,,
$$
we see that 
\begin{align*}
\Bigl| \P  \bigl\{ &\regN_n^{(k)}(R) = 0 \bigr\} - \P  \bigl\{ N^{(k)}(R) = 0 \bigr\} \Bigr|  \\
&\leq \Bigl| \P  \bigl\{ \regN_n^{(k)}(R) = 0 \bigr\} - \P  \bigl\{ \sum_{\bi \in \mathcal{I}_{n,k}} \xi_{\bi, n} = 0 \bigr\} \Bigr| \\
&\quad \, + \Bigl| \P  \bigl\{ \sum_{\bi \in \mathcal{I}_{n,k}} \xi_{\bi, n} = 0 \bigr\} - \P  \bigl\{ N^{(k)}(R) = 0 \bigr\} \Bigr|  \equiv I_1 + I_2\,.
\end{align*}
We  recall the following basic inequality: for integer-valued random variables $X$ and $Y$ defined on the same probability space,
$$
\bigl| \P \{X=0\} - \P \{Y=0\} \bigr| \ \leq \ \E  |X-Y|\,.
$$
Combining this inequality and Lemma \ref{l:palm.lebesgue2} in the Appendix proves $I_1 \to 0$ as $n \to \infty$. 
To demonstrate that $I_2 \to 0$ as $n \to \infty$, we introduce the total variation distance: for real-valued random variables $X$ and $Y$ defined on the same probability space $(\Omega, \mathcal{F}, \P )$, 
$$
d_{\text{TV}} (X,Y) \equiv \sup_{A \in \mathcal{F}} |\P \{X \in A\} - \P \{Y \in A\}|\,.
$$
Using this norm, 
\beq
\label{dtv:equn}
&&\biggl| \P \Bigl\{ \sum_{\bi \in \mathcal{I}_{n,k}} \xi_{\bi, n} = 0 \Bigr\} - \P \{N^{(k)}(R) = 0\} \biggr|   \notag \\     
&&\qquad\qquad\leq d_{\text{TV}} \biggl(\sum_{\bi \in \mathcal{I}_{n,k}} \xi_{\bi, n}, \ Poi\Bigl(\E \Bigl\{\sum_{\bi \in \mathcal{I}_{n,k}} \xi_{\bi, n} \Bigr\}\Bigr) \biggr) \\   \notag
&&\qquad \qquad\qquad + d_{\text{TV}} \biggl( Poi\Bigl(\E \Bigl\{ \sum_{\bi \in \mathcal{I}_{n,k}} \xi_{\bi, n} \Bigr\}\Bigr), \ N^{(k)}(R) \biggr)\,.
\eeq
Since the total variation distance between two Poisson variables can be bounded by the difference of their means (see Lemma 7.3 in \cite{bobrowski:adler:2014}), we have
$$
d_{\text{TV}} \biggl( Poi\Bigl(\E \Bigl\{\sum_{\bi \in \mathcal{I}_{n,k}} \xi_{\bi, n} \Bigr\}\Bigr), \ N^{(k)}(R) \biggr) \leq \biggl| \E \Bigl\{ \sum_{\bi \in \mathcal{I}_{n,k}} \xi_{\bi, n} \Bigr\} - \E \bigl\{ N^{(k)}(R) \bigr\}\biggr| .
$$
Thus the final term in \eqref{dtv:equn} converges to zero as $\to\infty$. 

To handle the first term on the right hand side of \eqref{dtv:equn}, note first that each $\xi_{\bi,n}$ is a Bernoulli random variable. 
To handle dependencies, create a graph with vertices the indices $\bi \in \mathcal{I}_{n,k}$ by placing an edge 
 between $\bi$ and $\bj$ (write $\bi \sim \bj$) if $\bi$ and $\bj$ share at least one component (i.e.\ if $|\bi \cap \bj|>0$). Then, $(\mathcal{I}_{n,k}, \sim)$ provides a \textit{dependency graph} with respect to $(\xi_{\bi, n}, \, \bi \in \mathcal{I}_{n,k})$; that is, for any two disjoint subsets $I_1$, $I_2$ of $\mathcal{I}_{n,k}$ with no edges connecting $I_1$ and $I_2$, $(\xi_{\bi, n}, \, \bi \in I_1)$ is independent of $(\xi_{\bi, n}, \, \bi \in I_2)$. Therefore, we are able to apply the so-called Poisson approximation theorem (see \cite{arratia:goldstein:gordon:1989} and also, Theorem 2.1 in \cite{penrose:2003}):
\begin{align*}
d_{\text{TV}} &\biggl(\sum_{\bi \in \mathcal{I}_{n,k}} \xi_{\bi, n}, \ Poi\Bigl(\E \Bigl\{ \sum_{\bi \in \mathcal{I}_{n,k}} \xi_{\bi, n} \Bigr\} \Bigr) \biggr) \\
&\leq 3 \left\{ \sum_{\bi \in \mathcal{I}_{n,k}} \sum_{\bj \in {N}_{\bi}} \E  \{\xi_{\bi, n}\} \E \{ \xi_{\bj, n} \}+ \sum_{\bi \in \mathcal{I}_{n,k}} \sum_{\bj \in  {N}_{\bi}\setminus \{ \bi \}} \E  \{\xi_{\bi, n} \xi_{\bj, n} \}\right\},
\end{align*}
where $N_{\bi} = \bigl\{ \bj \in \mathcal{I}_{n,k}: |\bi \cap \bj|>0 \bigr\}$. Now, for large enough $n$,
$$
\E  \{\xi_{\bi, n} \} \leq \sum_{p=1}^m \P  \bigl\{ h_n(\X_k)=1, \ \ \Xsupnk \in A_p \,  \bigr\} \leq 2 \begin{pmatrix} n \\ k \end{pmatrix}^{-1} \sum_{p=1}^m \nu_k(A_p),
$$
and thus
$$
\sum_{\bi \in \mathcal{I}_{n,k}} \sum_{\bj \in N_{\bi}} \E \{ \xi_{\bi, n} \}\E \{ \xi_{\bj, n}\} \leq \begin{pmatrix} n \\ k \end{pmatrix} \left( \begin{pmatrix} n \\ k \end{pmatrix} - \begin{pmatrix} n-k \\ k \end{pmatrix} \right) 4 \begin{pmatrix} n \\ k \end{pmatrix}^{-2} \Bigl( \sum_{p=1}^m \nu_k(A_p) \Bigr)^2.
$$
Here, it is clear that the right-hand side vanishes as $n \to \infty$. 

For $\bi$, $\bj \in \mathcal{I}_{n,k}$ with $|\bi \cap \bj| = l \in \{ 1,\dots,k-1 \}$,
\begin{align*}
\E  \{\xi_{\bi,n} \xi_{\bj,n} \}&\leq \sum_{p=1}^m \sum_{p^{\prime}=1}^m \P  \biggl\{ h_n(\X_k)=1, \ \ \Xsupnk \in A_p, \\
&\quad \, h_n(X_1,\dots,X_l,X_{k+1},\dots,X_{2k-l})=1, \\     
&\quad \, \Bigl( c_{k,n}^{-1}(X_j - d_{k,n}S(X_j))\,, \ j=1,\dots,l, k+1, \dots, 2k-l \Bigr) \in A_{p^\prime} \, \biggr\}\,.
\end{align*}
Regardless of the definition of $E_k$, there exists a compact set $K_1 \subset E_{2k-l}$ such that for all $p, p^{\prime} = 1,\dots,m$,
\begin{multline*}
\bigl\{ (x_1,\dots,x_{2k-l}) \in (\bbr^d)^{2k-l}: (x_1,\dots,x_k) \in A_p\,, \\
 (x_1,\dots,x_l,x_{k+1},\dots,x_{2k-l}) \in A_{p^{\prime}} \bigr\} \subset K_1
\end{multline*}
Thus, for some constant $C_1>0$, 
\begin{align*}
&\sum_{\bi \in \mathcal{I}_{n,k}} \sum_{\bj \in N_{\bi}\setminus \{ \bi \}} \E \{ \xi_{\bi, n} \xi_{\bj, n} \}\\     
&= \sum_{l=1}^{k-1} \begin{pmatrix} n \\ k \end{pmatrix} \begin{pmatrix} k \\ l \end{pmatrix} \begin{pmatrix} n-k \\ k-l \end{pmatrix} \E \{ \xi_{\bi, n} \xi_{\bj, n} \}\one ( |\bi \cap \bj| = l ) \\     
&\leq C_1 \sum_{l=1}^{k-1} n^{2k-l} \P  \Bigl\{ h_n(\X_k)=1, \ h_n(X_1,\dots,X_l,X_{k+1},\dots,X_{2k-l})=1,  \\      
&\qquad \qquad \, \ \  \X_{2k-l}^{(n)} \in K_1 \, \Bigr\} \to 0 \ \ \text{as } n \to \infty\,,
\end{align*}
where the last convergence follows from \eqref{e:anti.cluster} and now, \eqref{e:Kallen2} is proved as required. 

In order to prove that $\tildeN_n^{(k)}$ has the same weak limit as $\regN_n^{(k)}$, we only have to verify that for every non-negative continuous function $f: E_k \times L_k \to \bbr_+$ with compact support,
\begin{align*}
\regN_n^{(k)}&(f) - \tildeN_n^{(k)}(f) \\     
&= \sum_{\bi \in \IPk} \bigl( h_n(\mathcal{X}_{\bi}) - g_n(\mathcal{X}_{\bi},  \mathcal{P}_n) \bigr) f \bigl( \Xsupn_\bi, \ \bi/|\mathcal{P}_n| \bigr)  \stackrel{p}{\to} 0\,.
\end{align*}
Let $K_2$ be a compact set in $E_k$ so that the support of $f$ is contained in $K_2 \times L_k$. Noting that $\|f\|_{\infty} = \sup_{(x,y) \in K_2 \times L_k}f(x,y) < \infty$, 
\begin{align*}
\regN_n^{(k)}&(f) - \tildeN_n^{(k)}(f) \\     
&\leq \|f\|_{\infty} \sum_{\bi \in \IPk} \one \Bigl( h_n(\mathcal{X}_{\bi}) = 1, \ g_n(\mathcal{X}_{\bi}\,, \mathcal{P}_n) = 0, \ \ \Xsupn_\bi \in K_2 \, \Bigr)\,.
\end{align*}
In view of the Palm theory in Lemma \ref{l:palm.penrose} in the Appendix, we need to show that
$$
n^k \P  \bigl\{ h_n(\X_k) = 1, \ g_n( \X_k\,, \X_k \cup \mathcal{P}_n^{\prime}) = 0, \ \  \Xsupnk \in K_2 \, \bigr\} \to 0\,,
$$
where $\mathcal{P}_n^{\prime}$ is an independent copy of $\mathcal{P}_n$. Combining \eqref{e:conv.in.exp} and \eqref{e:component} completes the proof. 
\end{proof}

\subsection{Proof of Theorem \ref{t:main.heavy}}

\begin{proof}[Proof of Theorem \ref{t:main.heavy}]
By Theorem \ref{t:general.thm}, we only establish that the  three  convergence conditions, \eqref{e:conv.in.exp}, \eqref{e:anti.cluster}, and \eqref{e:component}, with $c_{k,n}$, $d_{k,n}$ replaced by $R_{k,n}$ and $0$, respectively, are satisfied. Since the proof of \eqref{e:conv.in.exp} is very similar to (and actually even easier than) that of \eqref{e:component}, we only check \eqref{e:anti.cluster} and \eqref{e:component}. We shall start with \eqref{e:component}. In view of the Portmanteau theorem for vague convergence (e.g.\ Proposition 3.12 in \cite{resnick:1987}) we need to show 
\beq \label{e:goal1.heavy}
\begin{pmatrix} n \\ k \end{pmatrix} \P  \bigl\{ \, g_n ( \X_k, \, \X_k \cup \mathcal{P}_n^{\prime}) = 1\,, \  R_{k,n}^{-1} \X_k \in K \,  \bigr\} \stackrel{v}{\to} \nu_k(K)
\eeq
for all relatively compact $K$ in $E_k= \bigl([-\infty,\infty]^d\bigr)^k \setminus \{ \mathbf{0} \}$ for which $\nu_k(\partial K)=0$. Without loss of generality, we can take $K=(a_0,b_0]\times \dots \times (a_{k-1},b_{k-1}] \subset E_k$, where $a_i$, $b_i$, $i=0,\dots,k-1$ are $d$-dimensional real vectors. Recall that any relatively compact set in $E_k$ is bounded away from the origin, and as assumed in \eqref{e:close.enough}, $h(x_1,\dots,x_k)=1$  only when  $x_1,\dots,x_k$ are all close enough to each other. Therefore, we can and shall  assume that each $(a_i,b_i]$ is bounded away from the origin. In particular, we assume that there exists an $\eta>0$ such that 
\begin{equation}  \label{e:bdd.away.0}
(a_i,b_i] \subset \bigl\{ x \in \bbr^d: \|x\| \geq \eta \bigr\}\,, \ \ i=0,\dots,k-1\,.
\end{equation}
{Consequently, we have
\begin{align*}  
\begin{pmatrix} n \\ k \end{pmatrix} &\P  \bigl\{ \, g_n ( \X_k, \, \X_k \cup \mathcal{P}_n^{\prime} ) = 1\,,  \ a_{i-1} \prec R_{k,n}^{-1} X_i \preceq b_{i-1}, \ i=1,\dots,k \  \bigr\}  \\
&= \begin{pmatrix} n \\ k \end{pmatrix} \E  \biggl\{ h_n(\X_k) \, \one \Bigl( \ a_{i-1} \prec R_{k,n}^{-1} X_i \preceq b_{i-1}\,, \ i=1,\dots,k \Bigr) \\
&\quad \, \times \P  \biggl\{ G \bigl( \X_k, \, r_n \bigr) \text{ is an isolated component of }   G \bigl( \X_k \cup \mathcal{P}_n^{\prime}, \, r_n \bigr) \, \Bigl| \, \X_k \biggr\} \biggr\}  \\
&= \begin{pmatrix} n \\ k \end{pmatrix} \int_{(\bbr^d)^k} h_n(x_1, \dots, x_k) \, \one ( a_{i-1} \prec R_{k,n}^{-1} x_i \preceq b_{i-1}\,, \ i=1,\dots,k ) \\
&\quad \, \times \exp \bigl\{ -np(x_1,\dots,x_k; \, r_n) \bigr\} f(x_1) \dots f(x_k) d\bx\,.
\end{align*} }
Let $I_k$ denote the last integral. The change of variables $x_1 \leftrightarrow x$, $x_i \leftrightarrow x + r_n y_{i-1}$, $i=2,\dots,k$, together with the location invariance of $h$, yields
\begin{align*}
I_k = &\begin{pmatrix} n \\ k \end{pmatrix} r_n^{d(k-1)} \hspace{-5pt} \int_{a_0 \prec R_{k,n}^{-1}x \preceq b_0} \hspace{-10pt} f(x) \int_{(\bbr^d)^{k-1}} \hspace{-10pt} h(0,\by) \prod_{i=1}^{k-1} \one \bigl( a_i \prec R_{k,n}^{-1} (x+r_ny_i) \preceq b_i \bigr)\, \\     
& \times f(x+r_ny_i) \, \exp \bigl\{ -np(x, x+r_n\by; \, r_n) \bigr\} d\by dx\,.
\end{align*}
Applying the polar coordinate transform $x \leftrightarrow (r,\theta)$ with $J(\theta) = |\partial x / \partial \theta|$,  and changing variables  by setting  $\rho = r/R_{k,n}$, gives
\begin{align*}
I_k = &\begin{pmatrix} n \\ k \end{pmatrix} r_n^{d(k-1)} R_{k,n}^d f(R_{k,n} e_1)^k \\
&\times  \int_{(\bbr^d)^{k-1}} \int_{S^{d-1}} \int_0^{\infty} \one (a_0 \prec \rho \theta \preceq b_0 )\, \rho^{d-1} f(R_{k,n} e_1)^{-1} f(R_{k,n} \rho e_1)  \\     
&\times  \prod_{i=1}^{k-1} \one \bigl( a_i \prec \rho \theta + r_n y_i / R_{k,n} \preceq b_i \bigr)\, f(R_{k,n}e_1)^{-1} f\bigl(R_{k,n} \|\rho \theta + r_n y_i / R_{k,n}\| e_1 \bigr) \\     
&\times \exp \bigl\{ -np(R_{k,n} \rho \theta, R_{k,n} \rho \theta + r_n\by; \, r_n) \bigr\}\, h(0,\by) d\rho J(\theta) d\theta d\by, 
\end{align*}
where $J(\theta) = |\partial x / \partial \theta|$ is the usual Jacobian and $S^{d-1}$ denotes the $(d-1)$-dimensional unit sphere in $\bbr^d$. 
Recalling \eqref{e:normalizing.heavy} and the assumption that $f$ has a regularly varying tail of exponent $-\alpha$, we see that for all $\rho > 0$, $\theta \in S^{d-1}$ and $y_i \in \bbr^d$, $i=1,\dots,k-1$, 
\begin{align}  
\begin{pmatrix} n \\ k \end{pmatrix} &r_n^{d(k-1)} R_{k,n}^d f(R_{k,n}e_1)^k\, \frac{f(R_{k,n} \rho e_1)}{f(R_{k,n}e_1)} \prod_{i=1}^{k-1} \one \bigl(a_i \prec \rho \theta + r_ny_i/R_{k,n} \preceq b_i\bigr)\,  \label{e:inside.conv}  \\     
&\times \frac{f\bigl(R_{k,n} \|\rho \theta + r_n y_i / R_{k,n}\| e_1 \bigr)}{f(R_{k,n}e_1)}  \notag \\
&\qquad \qquad \, \to \frac{1}{k!} \rho^{-\alpha k} \prod_{i=1}^{k-1} \one \bigl(a_i \prec \rho \theta \preceq b_i\bigr)\,, \quad
\text{as } n \to \infty\,.  \notag
\end{align}
Subsequently, we shall  show that
\begin{equation}  \label{e:vanish.comp}
np(R_{k,n} \rho \theta, R_{k,n} \rho \theta + r_n\by; \, r_n) \to 0\,, \ \ \ n \to \infty
\end{equation}
for every $\rho>0$, $\theta \in S^{d-1}$ and $y_i \in \bbr^d$, $i=1,\dots,k-1$. By the change of variables, 
\begin{align}
&np(R_{k,n} \rho \theta, R_{k,n} \rho \theta + r_n\by; \, r_n) \label{e:neg.comp} \\      
&= n r_n^d f(R_{k,n}e_1) \int_{B(0; \, 2) \cup \bigcup_{i=1}^{k-1} B(y_i; \, 2)} f(R_{k,n} e_1)^{-1} f \bigl( R_{k,n} \|\rho \theta + r_nz/R_{k,n}\| e_1 \bigr) dz\,. \notag
\end{align}
Appealing to the Potter bound (e.g.\ Theorem 1.5.6 in \cite{bingham:goldie:teugels:1987}), for every $\xi \in (0,\alpha)$, there is $C_1>0$ such that for sufficiently large $n$,
\begin{align*}
&\sup_{z \in B(0; \, 2) \cup \bigcup_{i=1}^{k-1} B(y_i; \, 2)} \frac{f \bigl( R_{k,n} \|\rho \theta + r_n z /R_{k,n}\|e_1 \bigr)}{f(R_{k,n}e_1)} \\     
&\leq C_1 \hspace{-20pt} \sup_{z \in B(0; \, 2) \cup \bigcup_{i=1}^{k-1} B(y_i; \, 2)} \Bigl( \, \|\rho \theta + r_nz/R_{k,n}\|^{-(\alpha + \xi)} + \|\rho \theta + r_nz/R_{k,n}\|^{-(\alpha - \xi)} \Bigr)  \\     
&\leq C_1 \bigl(2\rho^{-(\alpha + \xi)} + 2\rho^{-(\alpha - \xi)}\bigr)\,.
\end{align*}
Therefore, the supremum over $n$ of the integral in \eqref{e:neg.comp} is finite. On the other hand, \eqref{e:normalizing.heavy} ensures that $n r_n^d f(R_{k,n}e_1) \to 0$ as $n \to \infty$ and hence, the convergence in \eqref{e:vanish.comp} follows. 

Note that 
\begin{align*}
\frac{1}{k!} &\int_{(\bbr^d)^{k-1}} \int_{S^{d-1}} \int_0^{\infty} \rho^{d-1-\alpha k} \prod_{i=0}^{k-1} \one \bigl(a_i \prec \rho \theta \preceq b_i\bigr)\, h(0,\by) d\rho J(\theta) d\theta d\by  \\
&= \frac{1}{k!} \int_{(\bbr^d)^{k-1}} h(0,\by) d\by \int_{a_i \prec x \preceq b_i \ i=0,\dots,k-1} \|x\|^{-\alpha k} dx  \\
&= \nu_k \bigl( (a_0,b_0] \times \dots \times (a_{k-1},b_{k-1}] \bigr)\,.
\end{align*}
Thus, the proof of \eqref{e:goal1.heavy} can be finished, provided that the convergence in \eqref{e:inside.conv} holds under the integral sign. Indeed, one more application of the Potter bound, together with \eqref{e:bdd.away.0}, verifies the following: for every $\xi \in (0,\alpha-d)$, there exist $C_2$, $C_3>0$ such that
\begin{align*}
\one (a_0 \prec  \rho \theta \preceq b_0 ) \, \frac{f(R_{k,n} \rho e_1)}{f(R_{k,n} e_1)}     
\leq C_2 \, \one (\rho \geq \eta)\, ( \rho^{-(\alpha + \xi)} + \rho^{-(\alpha -\xi)})\,,
\end{align*}
and 
\begin{align*}
\prod_{i=1}^{k-1} &\one \bigl(a_i \prec \rho \theta + r_n y_i/R_{k,n} \preceq b_i\bigr) \, \frac{f(R_{k,n} \|\rho \theta + r_n y_i / R_{k,n}\| e_1)}{f(R_{k,n} e_1)} \\     
&\leq C_3( \eta^{-(\alpha+\xi)(k-1)} + \eta^{-(\alpha-\xi)(k-1)})\,. 
\end{align*} 
Since $\int_{\eta}^{\infty} (\rho^{d-1-\alpha+\xi} + \rho^{d-1-\alpha-\xi}) d\rho < \infty$, the dominated convergence theorem justifies the convergence under the integral sign, and so we have completed the first part of the proof; viz.\ \eqref{e:component}  is satisfied.

Next, we shall  prove \eqref{e:anti.cluster}. For $l=1,\dots,k-1$ and every compact set $K \subset E_{2k-l}=\bigl([-\infty,\infty]^d\bigr)^{2k-l} \setminus \{ \mathbf{0} \}$, we can assume without loss of generality that there exists $\eta^{\prime} > 0$ such that 
$$
K \subset \bigl\{ (x_1,\dots,x_{2k-l}) \in (\bbr^d)^{2k-l}: \|x_i\| \geq \eta^{\prime}\,, \ i=1,\dots, 2k-l \, \bigr\}\,.
$$
Proceeding in the same manner as above, the probability in \eqref{e:anti.cluster} is bounded by  
\begin{align*}
&n^{2k-l} \P  \Bigl\{ h_n(\X_k) = 1\,, \ h_n(X_1,\dots,X_l,X_{k+1},\dots,X_{2k-l})=1\,, \\     
&\quad \, \quad \, \quad \|X_i\| \geq \eta^{\prime} R_{k,n}, \ i=1,\dots,2k-l \,  \Bigr\}  \\
&=n^{2k-l} r_n^{d(2k-l-1)} R_{k,n}^d f(R_{k,n}e_1)^{2k-l} \\
&\quad \, \times  \int_{(\bbr^d)^{2k-l-1}}\int_{S^{d-1}} \int_{\eta^{\prime}}^{\infty} \rho^{d-1} f(R_{k,n}e_1)^{-1} f(R_{k,n} \rho e_1)  \\     
&\quad \, \times  \hspace{-5pt} \prod_{i=1}^{2k-l-1} \one \bigl(\|\rho \theta + r_n y_i / R_{k,n}\| > \eta^{\prime} \bigr)\, f(R_{k,n}e_1)^{-1} f \bigl( R_{k,n} \|\rho \theta + r_n y_i / R_{k,n}\| e_1 \bigr)  \\     
&\quad \, \times h(0,y_1,\dots,y_{k-1})  h(0,y_1,\dots,y_{l-1},y_k, \dots, y_{2k-l-1})\, d\rho J(\theta) d\theta d\by\,.
\end{align*}
The asymptotic order of the last integral is $\mathcal{O} \bigl( n^{2k-l} r_n^{d(2k-l-1)} R_{k,n}^d f(R_{k,n} e_1)^{2k-l} \bigr)$, which tends to $0$ as $n \to \infty$ for every $l = 1,\dots,k-1$, and so we are done.
\end{proof}

\subsection{Proof of Theorem \ref{t:main.light}}

For the proof of Theorem \ref{t:main.light}, we need a preliminary lemma. 

\begin{lemma}(Lemma 5.2 in \cite{balkema:embrechts:2004})  \label{l:upper.bound.light}
Given $(R_{k,n}, \, n \geq 1)$ as in \eqref{e:normalizing.light}, let $(q_m(n), \, m \geq 0, \, n \geq 1)$ be defined by 
$$
q_m(n) = a(R_{k,n})^{-1} \Bigl( \psiinv \bigl(\psi(R_{k,n}) + m \bigr) - R_{k,n} \Bigr)\,,  
$$
equivalently, 
$$
\psi \bigl( R_{k,n} + a(R_{k,n}) q_m(n) \bigr) = \psi (R_{k,n}) + m
$$
(since $\psi$ is an increasing function, the inverse function $\psiinv$ is well-defined everywhere). 
Then, given $\epsilon > 0$, there is an integer $N_{\epsilon} \geq 1$ such that
$$
q_m(n) \leq e^{m\epsilon} / \epsilon \ \ \text{for all } n \geq N_{\epsilon}, m \geq 0\,.
$$
\end{lemma}

\begin{proof}[Proof of Theorem \ref{t:main.light}]{\color{white} .}
The proof is rather long, and so we break into two main units, one each for the two main cases, further dividing the proof of the first case into three sign-posted parts. Hopefully this will help the reader to navigate the next few pages.

\vskip0.1truein
\noindent \underline{\textit{Proof of statement $(i)$}} 

\vskip0.1truein
\noindent \underline{\textit{Part 1}:}
As argued in the proof of Theorem \ref{t:main.heavy}, we need only check \eqref{e:anti.cluster} and \eqref{e:component}. As for \eqref{e:component}, we need to show that 
\beq  \label{e:start.prob} 
\begin{pmatrix} n \\ k \end{pmatrix} \P  \bigl\{ \, g_n ( \X_k, \, \X_k \cup \mathcal{P}_n^{\prime} ) = 1\,, \ 
\Xsupnk \in K \  \bigr\} \to \nu_k(K)
\eeq
for all relatively compact $K$ in $E_k=\bigl((-\infty,\infty]^d\bigr)^k$ for which $\nu_k(\partial K)=0$. Without loss of generality, it suffices to  check the case $K=I_0 \times \dots \times I_{k-1}$, where $I_i = (a_i,b_i] \subset \bbr^d$ and $a_i$, $b_i$ are $d$-dimensional vectors such that $-\infty \prec a_i \preceq b_i \preceq \infty$. Define 
$$
H = S^{d-1} \cap \bigl\{ (z_1,\dots,z_d) \in \bbr^d: z_1 + \dots + z_d \geq 0 \bigr\},
$$
and  let $H^c = S^{d-1} \setminus H$. Let $\Theta_1 = X_1 / \|X_1\|$. Since $a(R_{k,n})^{-1} R_{k,n} S(X_1) \to \infty$ a.s. (in a componentwise sense), we have  that, for large enough $n$, $\Theta_1 \in H^c$ and $a(R_{k,n})^{-1} \bigl(X_1 - R_{k,n}S(X_1)\bigr) \succeq a_0 (\succ -\infty)$ do not occur simultaneously. 
Thus, the left-hand side in \eqref{e:start.prob} equals 
$$  
\begin{pmatrix} n \\ k \end{pmatrix} \P  \Bigl\{ \, g_n ( \X_k, \, \X_k \cup \mathcal{P}_n^{\prime} ) = 1\,, \ \Theta_1 \in H\,,   \
\Xsupnk \in \prod_{i=1}^k I_{i-1}\,   \Bigr\} + o(1)
$$
as $n \to \infty$. Letting $J_k$ denote the leading term in the above,  we can write 
 \begin{align*}
J_k &= \begin{pmatrix} n \\ k \end{pmatrix} \E  \biggl\{ h_n(\X_k) \, \one \bigl( \, \Xsupnk \in \prod_{i=1}^k I_{i-1}\,, \ \ \Theta_1 \in H \bigr) \\
&\quad \, \times \P  \biggl\{ G \bigl( \X_k, \, r_n \bigr) \text{ is an isolated component of } G \bigl( \X_k \cup \mathcal{P}_n^{\prime}, \, r_n \bigr) \, \Bigl| \, \X_k \biggr\} \biggr\}  \\
&=\begin{pmatrix} n \\ k \end{pmatrix} \int_{(\bbr^d)^k}h_n(x_1,\dots,x_k) \, \\
&\quad \, \times \one \Bigl( \, a(R_{k,n})^{-1} \bigl( x_i - R_{k,n}S(x_i) \bigr) \in I_{i-1}\,,  i=1,\dots,k, \ \ x_1 / \|x_1\| \in H \Bigr) \\
&\qquad\qquad \times \exp \bigl\{ -np(x_1,\dots,x_k; r_n) \bigr\}\, f(x_1)\dots f(x_k) d\bx,
\end{align*} 
where the definition of $p$ was given in \eqref{e:comp.func}.

The change of variables $x_1 \leftrightarrow x$ and $x_i \leftrightarrow x+r_ny_{i-1}$, $i=2,\dots,k$ yields
\begin{align*}
J_k = &\begin{pmatrix} n \\ k \end{pmatrix} r_n^{d(k-1)} \int_{\bbr^d} \one \Bigl(\, a(R_{k,n})^{-1} \bigl(x - R_{k,n} S(x)\bigr) \in I_0\,, \ x/\|x\| \in H \Bigr)\, f(x) \\     
&\quad\times \int_{(\bbr^d)^{k-1}} h(0,\by) \, \prod_{i=1}^{k-1} \, \one \Bigl(\, a(R_{k,n})^{-1} \bigl(x+r_ny_i - R_{k,n} S(x+r_ny_i)\bigr) \in I_i \bigr)  \\     
&\quad\qquad\times f(x+r_ny_i) \exp \bigl\{ -np(x, x+r_n\by; r_n) \bigr\} d\by dx.
\end{align*}
Further calculation by the polar coordinate transform $x \leftrightarrow (r,\theta)$ with $J(\theta) = |\partial x / \partial \theta|$ and the change of variable $\rho = a(R_{k,n})^{-1}(r-R_{k,n})$ gives
\begin{align*}
J_k = &\begin{pmatrix} n \\ k \end{pmatrix} r_n^{d(k-1)} a(R_{k,n}) R_{k,n}^{d-1} f(R_{k,n}e_1)^k \\
&\times \int_{(\bbr^d)^{k-1}} \int_{S^{d-1}} \int_0^{\infty}  \prod_{i=1}^6 L_i\, h(0,\by)\,  d\rho J(\theta) d\theta d\by\,,    
\end{align*}
where
\begin{align*}
L_1 &= \one \bigl( \, \rho \theta - a(R_{k,n})^{-1} R_{k,n} \bigl( S(\theta) - \theta \bigr) \in I_0, \ \, \theta \in H \, \bigr)\,,  \\     
L_2 &= \left( 1 + \frac{a(R_{k,n})}{R_{k,n}} \rho \right)^{d-1}, \\     
L_3 &= f(R_{k,n}e_1)^{-1} f\Bigl( \bigl( R_{k,n} + a(R_{k,n}) \rho \bigr) e_1 \Bigr)\,,  \\
L_4 &= \prod_{i=1}^{k-1} \, \one \Bigl( \, \rho \theta + \frac{r_n}{a(R_{k,n})} y_i \\
&\qquad\qquad\qquad \, - \frac{R_{k,n}}{a(R_{k,n})} \left[ S\Bigl( \bigl(R_{k,n}+a(R_{k,n})\rho \bigr)\theta + r_ny_i \Bigr) - \theta \right] \in I_i \, \Bigr)\,, \\
L_5 &= \prod_{i=1}^{k-1} f(R_{k,n}e_1)^{-1} f\Bigl( \|(R_{k,n}+a(R_{k,n})\rho)\theta + r_ny_i\|e_1 \Bigr)\,,  \\     
L_6 &= \exp \Bigl\{ -np\bigl((R_{k,n}+a(R_{k,n})\rho)\theta, (R_{k,n}+a(R_{k,n})\rho)\theta+r_n\by; r_n \bigr)\Bigr\}\,.
\end{align*}

\vskip0.1truein
\noindent \underline{\textit{Part 2}:} For the application of the dominated convergence theorem, one needs to compute the limit for each $L_i$, $i=1,\dots,6$, while also establishing finite upper bounds for each term. We shall  begin with the indicators $L_1$ and $L_4$, which are trivially bounded. The fact that $a(R_{k,n})^{-1} R_{k,n} \to \infty$ ensures that
\begin{align*}
L_1 \to \one ( \, \rho \theta \in I_0\,, \ \, \theta \succeq 0 )\,,  \ \ \ n \to \infty\,.
\end{align*}
Observe also that
$$
L_1 \leq \one ( \rho \geq -M^{\prime} ) \ \ \text{for some } M^{\prime} \geq 0\,.
$$
To see this in more detail, choose $M^{\prime} \geq 0$ such that $\min_{1 \leq j \leq d} a_0^{(j)} \geq -M^{\prime}$ (the superscript $(j)$ denotes `$j$th component` of a given vector). Fix $\rho$ and $\theta$ such that $L_1 = 1$. Then for all $j$ with $\theta^{(j)}>0$,
$$
\rho \geq \rho \theta^{(j)} = \rho \theta^{(j)} - a(R_{k,n})^{-1} R_{k,n} \bigl(S(\theta)^{(j)} -  \theta^{(j)} \bigr) \geq a_0^{(j)} \geq -M^{\prime}\,.
$$
Note that $\theta \in H$ guarantees that at least one component in $\theta$ must be positive. 

 In what follows, we prove the assertion when $M^{\prime} = 0$. The proof for a general $M^{\prime}$ is notationally more complicated, but essentially the same.

Before moving to $L_4$, note the following useful expansion, which will be  applied repeatedly in what follows: For each $i=1,\dots,d$, 
\begin{equation}  \label{e:expansion}
\Bigl|\Bigl|\bigl( R_{k,n} + a(R_{k,n})\rho \bigr)\theta + r_n y_i \Bigr|\Bigr| = R_{k,n} + a(R_{k,n}) \left( \rho + \frac{\langle \theta, y_i \rangle + \gamma_n(\rho, \theta, y_i)}{a(R_{k,n})/r_n} \right),
\end{equation}
so that $\gamma_n(\rho, \theta, y_i) \to 0$ uniformly in $\rho \geq 0$, $\theta \in S^{d-1}$, and $\|y_i\| \leq M$ ($M$ is determined in \eqref{e:close.enough}). 

Turning now to  $L_4$, we shall  rewrite the expression within the indicator as follows:
\begin{align}
\rho \theta + \frac{r_n}{a(R_{k,n})} \left( y_i + \frac{\alpha_n}{\beta_n} \right), \label{e:expression.L4} 
\end{align}
where 
\begin{multline*}
\alpha_n = \bigl( \langle \theta, y_i \rangle + \gamma_n(\rho, \theta, y_i) \bigr) \theta + \frac{R_{k,n}}{r_n} \left( 1 + \frac{a(R_{k,n})}{R_{k,n}} \rho \right) \theta \\
- m \left( \frac{R_{k,n}}{r_n} \left( 1 + \frac{a(R_{k,n})}{R_{k,n}} \rho \right) \theta + y_i \right)
\end{multline*}
with 
$$
m(x) = \bigl( |x^{(1)}|, \dots, |x^{(d)}| \bigr)\,,  \  \ \ x=(x^{(1)}, \dots, x^{(d)}) \in \bbr^d\,,
$$
and 
$$
\beta_n = 1 + \frac{a(R_{k,n})}{R_{k,n}} \left( \rho + \frac{\langle \theta, y_i \rangle + \gamma_n(\rho, \theta, y_i)}{a(R_{k,n})/r_n} \right).
$$
As seen above, $L_1 \to \one ( \rho \theta \in I_0, \ \, \theta \succeq 0 )$ as $n \to \infty$, so it is enough to discuss the convergence of \eqref{e:expression.L4} for every $\rho \geq 0$, $\theta \succ 0$, and $\|y_i\| \leq M$. Then, for large enough $n$,
$$
\frac{R_{k,n}}{r_n} \left( 1 + \frac{a(R_{k,n})}{R_{k,n}} \rho \right) \theta + y_i \succ 0.
$$
Thus, we have  that $\alpha_n \to \langle \theta, y_i \rangle \theta - y_i$ and $\beta_n \to 1$ as $n \to \infty$. Now we have 
$$
\rho \theta + \frac{r_n}{a(R_{k,n})} \left( y_i +  \frac{\alpha_n}{\beta_n} \right) \to \bigl( \rho + c^{-1} \langle \theta, y_i \rangle \bigr) \theta
$$
and in conclusion, for every $\theta \succ 0$, 
$$
L_4 \to \prod_{i=1}^{k-1} \, \one \Bigl( \bigl( \rho + c^{-1} \langle \theta, y_i \rangle \bigr) \theta \in I_i \Bigr)\,.
$$

Regarding $L_2$, it is clear that for every $\rho \geq 0$, $L_2 \to 1$ as $n \to \infty$ and it is also easy to check that $L_2 \leq 2 (\rho \vee 1)^{d-1}$.

As for $L_3$, we write 
$$
L_3 = L(R_{k,n})^{-1} L\bigl( R_{k,n} + a(R_{k,n}) \rho \bigr) \exp \Bigl\{ -\psi \bigl( R_{k,n} + a(R_{k,n})\rho \bigr) + \psi(R_{k,n}) \Bigr\}\,.
$$
An elementary calculation (e.g.\  p142 in \cite{embrechts:kluppelberg:mikosch:1997}) shows that 
\begin{equation}  \label{e:unif.conv.a}
\frac{a(R_{k,n})}{a\bigl( R_{k,n} + a(R_{k,n}) v \bigr)} \to 1
\end{equation}
uniformly on bounded $v$-sets. Namely $1/a$ is flat for $a$. Therefore, for every $\rho  \geq 0$, 
\begin{align*}
\exp \Bigl\{ -\psi \bigl( R_{k,n} + a(R_{k,n})\rho \bigr) + \psi(R_{k,n}) \Bigr\} \to e^{-\rho}, \ \ \ n \to \infty\,.
\end{align*}
Since $L$ is flat for $a$, it follows that $L_3 \to e^{-\rho}$ as $n \to \infty$ for every $\rho \geq 0$. For the upper bound of $L_3$ on $\{ \rho \geq 0 \}$, we apply Lemma \ref{l:upper.bound.light}. Choosing $\epsilon \in \bigl(0, (d+\gamma k)^{-1} \bigr)$ and recalling that $\psi$ is non-decreasing,
\begin{align*}
&\exp \Bigl\{ -\psi \bigl( R_{k,n} + a(R_{k,n})\rho \bigr) + \psi(R_{k,n}) \Bigr\} \, \one (\rho \geq 0 )  \\     
&= \sum_{m=0}^{\infty} \, \one \bigl( q_m(n) \leq \rho < q_{m+1}(n) \bigr)\, \exp \Bigl\{ -\psi \bigl( R_{k,n} + a(R_{k,n})\rho \bigr) + \psi(R_{k,n}) \Bigr\}  \\     
&\leq \sum_{m=0}^{\infty} \, \one \bigl( 0 \leq \rho \leq \epsilon^{-1} e^{(m+1)\epsilon} \bigr)\, e^{-m}
\end{align*}
for all $n \geq N_{\epsilon}$. \\
On the other hand, using the bound in \eqref{e:poly.upper}, on $\{ \rho \geq 0 \}$, we have, for sufficiently large $n$,
\begin{align*}
L(&R_{k,n})^{-1} L \bigl( R_{k,n} + a(R_{k,n}) \rho \bigr)  \leq C \left( 1 + \frac{a(R_{k,n})}{R_{k,n}} \rho \right)^{\gamma} \leq 2C (\rho \vee 1)^{\gamma}.
\end{align*}
Multiplying these bounds together, we have 
\begin{equation*}
L_3 \one (\rho \geq 0) \leq 2 C (\rho \vee 1)^{\gamma} \sum_{m=0}^{\infty} \, \one \bigl( 0 \leq  \rho \leq \epsilon^{-1} e^{(m+1)\epsilon} \bigr)\, e^{-m}.
\end{equation*}

Next, we turn to $L_5$. First, denote
\begin{equation}  \label{e:consise.taylor}
\xi_n(\rho, \theta, y) = \frac{\langle \theta, y_i \rangle + \gamma_n(\rho, \theta, y)}{a(R_{k,n})/r_n}\,. 
\end{equation}
Since $c = \lim_{n \to \infty}a(R_{k,n})/r_n$ is strictly positive, 
\begin{equation}  \label{e:no.blow.up}
A = \sup_{\substack{n \geq 1, \ \rho \geq 0, \\[2pt] \theta \in S^{d-1}, \ \|y\| \leq M}} \bigl| \xi_n(\rho, \theta, y) \bigr| < \infty\,.
\end{equation}
Using the expansion \eqref{e:expansion}, we can write 
\begin{align*}
L_5 = &\prod_{i=1}^{k-1} L(R_{k,n})^{-1} L \Bigl( R_{k,n} + a(R_{k,n}) \bigl( \rho +  \xi_n(\rho,\theta,y_i) \bigr) \Bigr) \\
&\times \exp \left\{ -\int_0^{\rho + \xi_n(\rho,\theta,y_i)} \frac{a(R_{k,n})}{a \bigl( R_{k,n} + a(R_{k,n} )r \bigr)}\, dr \right\}.
\end{align*}
Due to the uniform convergence \eqref{e:unif.conv.a} and \eqref{e:no.blow.up}, for every $\rho \geq 0$, $\theta \in S^{d-1}$, and $\|y_i\| \leq M$, 
$$
\int_0^{\rho + \xi_n(\rho,\theta,y_i)} \frac{a(R_{k,n})}{a \bigl( R_{k,n} + a(R_{k,n}) r \bigr)}\, dr \to \rho + c^{-1} \langle \theta, y_i \rangle\,, \ \ \ n \to \infty
$$
and  
$$
L(R_{k,n})^{-1} L \Bigl( R_{k,n} + a(R_{k,n}) \bigl( \rho +  \xi_n(\rho,\theta,y_i) \bigr) \Bigr) \to 1\,, \ \ \ n \to \infty\,.
$$
We thus conclude that
$$
L_5 \to \exp \bigl\{ -  (k-1) \rho - c^{-1} \sum_{i=1}^{k-1} \langle \theta, y_i \rangle  \bigr\}
$$
for every $\rho \geq 0$, $\theta \in S^{d-1}$, and $\|y_i\| \leq M$, $i=1,\dots,k-1$. 

To provide  an appropriate upper bound for $L_5$ on $\{ \rho \geq 0 \}$, note that, for large enough $n$, 
\begin{align*}
&\prod_{i=1}^{k-1} \exp \left\{ -\int_0^{\rho + \xi_n(\rho,\theta,y_i)} \frac{a(R_{k,n})}{a \bigl( R_{k,n} + a(R_{k,n} )r \bigr)}\, dr \right\} \, \one ( \rho \geq 0 )  \\     
&\leq \exp\, \biggl\{ \, \sum_{i=1}^{k-1} \int_{-A}^0 \frac{a(R_{k,n})}{a \bigl( R_{k,n} + a(R_{k,n})r \bigr)}\, dr\, \one \bigl(  - A \leq \rho + \xi_n(\rho,\theta,y_i) \leq 0 \bigr) \\     
&\quad \, -\sum_{i=1}^{k-1} \int_0^{\rho + \xi_n(\rho,\theta,y_i)} \frac{a(R_{k,n})}{a \bigl( R_{k,n} + a(R_{k,n})r \bigr)}\, dr \, \one \bigl( \rho + \xi_n(\rho,\theta,y_i)>0 \bigr) \biggr\}  \\     
&\leq e^{2 A (k-1)}\,. 
\end{align*}
It follows from \eqref{e:flat} and \eqref{e:poly.upper} that there exists $C_1 \geq 1$ such that, on $\{ \rho \geq 0 \}$,
\begin{align*}
&L(R_{k,n})^{-1} L \Bigl( R_{k,n} + a(R_{k,n}) \bigl( \rho + \xi_n(\rho,\theta,y_i) \bigr) \Bigr)  \\     
&= L(R_{k,n})^{-1} L \Bigl( R_{k,n} + a(R_{k,n}) \bigl( \rho + \xi_n(\rho,\theta,y_i) \bigr) \Bigr) \, \one \bigl( \rho + \xi_n(\rho,\theta,y_i) \geq - A \bigr)  \\     
&\leq 2 + L(R_{k,n})^{-1} L \Bigl( R_{k,n} + a(R_{k,n}) \bigl( \rho + \xi_n(\rho,\theta,y_i) \bigr) \Bigr) \, \one \bigl( \rho + \xi_n(\rho,\theta,y_i) > 0  \bigr)  \\     
&\leq 2 + C \left( 1 + \frac{a(R_{k,n})}{R_{k,n}} \bigl( \rho + \xi_n(\rho,\theta,y_i) \bigr) \right)^{\gamma}\,    \\     
&\leq 2 + 2C \Bigl( 1 \vee  \bigl( \rho + \xi_n (\rho,\theta,y_i) \bigr) \Bigr)^{\gamma}   \\     
&\leq C_1^{(k-1)^{-1}} (\rho \vee 1)^{\gamma}
\end{align*}
for sufficiently large $n$.
This in turn implies that
$$
\prod_{i=1}^{k-1} L(R_{k,n})^{-1} L \Bigl( R_{k,n} + a(R_{k,n}) \bigl( \rho + \xi_n(\rho,\theta,y_i) \bigr) \Bigr) \leq C_1\, (\rho \vee 1)^{\gamma(k-1)}
$$
on $\{ \rho \geq 0 \}$, and further, 
$$
L_5 \leq C_1\,  e^{2 A (k-1)}\, (\rho \vee 1)^{\gamma (k-1)}
$$
holds on $\{ \rho \geq 0 \}$.

Finally, we turn to  $L_6$.
\begin{align*}
-\log L_6 &= np\bigl((R_{k,n}+a(R_{k,n})\rho)\theta, (R_{k,n}+a(R_{k,n})\rho)\theta+r_n\by; r_n \bigr)  \\     
&= n r_n^d f(R_{k,n}e_1) \int_{B(0;2) \cup \bigcup_{i=1}^{k-1} B(y_i; 2)} f(R_{k,n}e_1)^{-1} \\
&\quad \, \times f\biggl( \Bigl( R_{k,n}+a(R_{k,n})\bigl(\rho + \xi_n(\rho,\theta,z) \bigr) \Bigr) e_1 \biggr) dz\,. 
\end{align*}
As argued in the derivation of the bound for $L_5$, the supremum of integrand over $n$ can be bounded by a constant multiple of $(\rho \vee 1)^{\gamma}$. On the other hand, \eqref{e:normalizing.light} guarantees $nr_n^d f(R_{k,n}e_1) \to 0$ as $n \to \infty$ and hence, $L_6 \to 1$ as $n \to \infty$. 

From the argument thus far, it follows that for every $\rho \geq 0$, $\theta \in S^{d-1}$, and $\|y_i\| \leq M$, $i=1,\dots,d$,
\begin{align*}
\prod_{i=1}^6 L_i \to &e^{-  k \rho - c^{-1} \sum_{i=1}^{k-1} \langle \theta, y_i \rangle } \\
&\times \one \bigl( \rho \theta \in I_0\,, \ \, \theta \succeq 0\,, \ \, (\rho + c^{-1} \langle \theta, y_i \rangle) \theta \in I_i, \, i=1,\dots,k-1 \bigr)\,.
\end{align*}
To finish the argument, it remains to check the $L^1$-integrability on $\{ \rho \geq 0 \}$ of the upper bound for $\prod_{i=1}^6 L_i$. As we have seen so far,  
\begin{align*}
&\prod_{i=1}^6 L_i \, \one ( \rho \geq 0 ) \\
&\leq 4 C C_1\, e^{2 A (k-1)}\, (\rho \vee 1)^{d-1+\gamma k} \sum_{m=0}^{\infty} \one \bigl( 0 \leq \rho \leq \epsilon^{-1} e^{(m+1)\epsilon} \bigr) e^{-m}
\end{align*}
for sufficiently large $n$. Indeed,
\begin{align*}
\int_{0}^{\infty} &(\rho \vee 1)^{d-1+\gamma k} \sum_{m=0}^{\infty} \one \bigl( 0 \leq \rho \leq \epsilon^{-1} e^{(m+1)\epsilon} \bigr) e^{-m} d \rho  \\         
&\leq \frac{e^{(d+\gamma k)\epsilon}}{\epsilon^{d+\gamma k}}\, \sum_{m=0}^{\infty} e^{- \bigl[ 1 - (d+\gamma k)\epsilon \bigr]m} 
\end{align*}
so that the right-hand side is finite because we took $0 < \epsilon < (d+\gamma k)^{-1}$. Applying the dominated convergence theorem as well as \eqref{e:normalizing.light}, it turns out that $J_k \to \nu_k(K)$, $n \to \infty$ as required.

\vskip0.1truein
\noindent \underline{\textit{Part 3}:} Next, we shall  prove that \eqref{e:anti.cluster} is satisfied. For $l=1,\dots,k-1$, and any compact set $K \subset E_{2k-l} = \bigl( (-\infty,\infty]^d\bigr)^{2k-l}$, there exists $B \geq 0$ such that 
$$
K \subset \bigl\{ (x_1,\dots,x_{2k-l}) \in (\bbr^d)^{2k-l}: x_i \succeq -B \one\,, \ i=1,\dots, 2k-l \, \bigr\}\,,
$$
where $\one$ is a $d$-dimensional vector with all entries $1$. Then, the probability in \eqref{e:anti.cluster} can be bounded by  
\begin{align*}
&n^{2k-l} \P  \bigl\{ h_n(\X_k) = 1\,, \ h_n(X_1,\dots,X_l,X_{k+1},\dots,X_{2k-l})=1\,, \  \X_{2k-l}^{(n)} \succeq -B \one  \bigr\} \\     
&=n^{2k-l} \int_{(\bbr^d)^{2k-l}} h_n(x_1,\dots,x_k) h_n(x_1,\dots,x_l,x_{k+1},\dots,x_{2k-l})  \\     
&\quad \, \times \one \Bigl(  x_i - R_{k,n} S(x_i) \succeq - a(R_{k,n}) B \one, \ i=1,\dots,2k-l\,, \ \, x_1/\|x_1\| \in H  \Bigr)  \\     
&\quad \, \times f(x_1)\dots f(x_{2k-l}) d\bx + o(1)\,, \ \ \ n \to \infty\,.
\end{align*}
The equality above follows from the fact that $X_1/\|X_1\| \in H^c$ and $ X_1 - R_{k,n}S(X_1) \geq -a(R_{k,n}) B \one$ do not occur simultaneously. \\
By precisely the same change of variables and polar coordinate transform as in the proof of \eqref{e:component}, the leading term of the last line above equals 
\begin{align*}
n^{2k-l} &r_n^{d(2k-l-1)} a(R_{k,n}) R_{k,n}^{d-1} f(R_{k,n}e_1)^{2k-l} \\
&\times \int_{(\bbr^d)^{2k-l-1}} \int_{S^{d-1}} \int_0^{\infty} \one \Bigl( \rho \theta - \frac{R_{k,n}}{a(R_{k,n})} \bigl( S(\theta) - \theta \bigr) \succeq -B\one\,, \ \theta \in H \Bigr) \\     
&\times \left( 1 + \frac{a(R_{k,n})}{R_{k,n}} \rho \right)^{d-1} f(R_{k,n}e_1)^{-1} f \Bigl( \bigl( R_{k,n} + a(R_{k,n}) \rho \bigr) e_1 \Bigr)  \\   
&\times \prod_{i=1}^{2k-l-1} \, \one \biggl( \, \rho \theta + \frac{r_n}{a(R_{k,n})} y_i \\
&\quad \, - \frac{R_{k,n}}{a(R_{k,n})} \Bigl[ S \Bigl( \bigl(R_{k,n} + a(R_{k,n}) \rho \bigr) \theta + r_n y_i \Bigr) - \theta \Bigr] \succeq - B \one \, \biggr) \\
&\times f(R_{k,n}e_1)^{-1} f \Bigl( \bigl| \bigl| \bigl( R_{k,n} + a(R_{k,n}) \rho \bigr) \theta + r_n y_i \bigr| \bigr| e_1 \Bigr)  \\     
&\times h(0,y_1,\dots,y_{k-1})\, h(0,y_1,\dots,y_{l-1},y_k,\dots,y_{2k-l-1})\, d\rho J(\theta) d\theta d\by\,.
\end{align*}
The asymptotic order of the above expression is 
$$
\mathcal{O} \bigl( n^{2k-l} r_n^{d(2k-l-1)} a(R_{k,n}) R_{k,n}^{d-1} f(R_{k,n}e_1)^{2k-l} \bigr)\,,
$$
which vanishes as $n \to \infty$ for every $l=1,\dots,k-1$, and so the proof of $(i)$ is complete. 
\vspace{10pt}

\noindent \underline{\textit{Proof of statement $(ii)$}}

\vskip0.1truein
\noindent
We need only show that, as $n \to \infty$,
$$
\regN_n^{(k)}(f) = \sum_{\bi \in \IPk} h_n(\mathcal{X}_{\bi})\, f \bigl( \Xsupn_\bi, \ \bi/|\mathcal{P}_n| \bigr) \stackrel{p}{\to} 0
$$
for every continuous non-negative function $f\:E_k \times L_k \to \bbr_+$ with compact support. Note that the support of $f$ is contained in 
$$
\bigl\{ (x_1,\dots,x_k) \in (\bbr^d)^k: \, x_i \succeq -B_1 \one\,, \ i=1,\dots,k\, \bigr\} \times L_k
$$
for some $B_1 \geq 0$. Therefore,
$$
\regN_n^{(k)}(f) \leq \|f\|_{\infty} \sum_{\bi \in \IPk} \, \one \bigl( h_n(\mathcal{X}_{\bi}) = 1\,, \ \Xsupn_\bi \succeq -B_1 \one \bigr).
$$
Observe that $X_{i_j} - R_{k,n}S(X_{i_j}) \succeq -a(R_{k,n}) B_1 \one$ implies 
$$
\|X_{i_j}\| - R_{k,n} \geq -a(R_{k,n})B_2
$$ 
for some $B_2 \geq B_1$. Hence, 
\begin{align*}
\regN_n^{(k)}(f) \leq &\|f\|_{\infty} \sum_{\bi \in \IPk} \, \one \bigl( h_n(\mathcal{X}_{\bi}) = 1\,, \\ 
&\|X_{i_j}\| - R_{k,n}  \geq -a(R_{k,n}) B_2\,, \ j=1,\dots,k\,  \bigr)\,.
\end{align*}
In view of Lemma \ref{l:palm.penrose} in the Appendix, what now needs to be verified is that
$$
n^k \P  \bigl\{ h_n(\X_k) = 1\,, \ \|X_i\| - R_n \geq -a(R_{k,n}) B_2\,, \ i=1,\dots,k\, \bigr\} \to 0\,.
$$
Once again, applying the same kind of change of variables and polar coordinate transform, together with \eqref{e:expansion},
\begin{align*}
&n^k \P  \bigl\{ h_n(\X_k) = 1\,, \ \|X_i\| - R_n \geq -a(R_{k,n}) B_2\,, \ i=1,\dots,k\, \bigr\} \\     
&= n^k r_n^{d(k-1)} a(R_{k,n}) R_{k,n}^{d-1} f(R_{k,n}e_1)^k \int_{(\bbr^d)^{k-1}} \int_{S^{d-1}} \int_{\rho \geq -B_2} \left( 1 + \frac{a(R_{k,n})}{R_{k,n}} \rho \right)^{d-1}  \\     
&\quad \times  f(R_{k,n}e_1)^{-1}\, f\Bigl( \bigl( R_{k,n} + a(R_{k,n}) \rho \bigr) e_1 \Bigr) \prod_{i=1}^{k-1} \, \one \bigl( \rho + \xi_n(\rho,\theta,y_i) \geq -B_2 \bigr) \\     
&\quad \times  f(R_{k,n}e_1)^{-1}\, f\biggl( \Bigl( R_{k,n}+a(R_{k,n})\bigl(\rho + \xi_n(\rho,\theta,y_i) \bigr) \Bigr) e_1 \biggr) \, h(0,\by)\, d\rho J(\theta)d\theta d\by  ,
\end{align*}
where the definition of $\xi_n(\rho,\theta,y_i)$ is given by \eqref{e:consise.taylor}.

As argued before, 
$$
\left( 1 + \frac{a(R_{k,n})}{R_{k,n}} \rho \right)^{d-1} \leq 2(\rho \vee 1)^{d-1}
$$
and, on $\{ \rho \geq -B_2 \}$,
$$
L(R_{k,n})^{-1} L \bigl( R_{k,n} + a(R_{k,n})\rho \bigr) \leq 2C (\rho \vee 1)^{\gamma}\,.
$$
Furthermore, since $a$ is eventually non-increasing, 
\begin{align*}
\exp \Bigl\{ -\psi \bigl( R_{k,n} + a(R_{k,n})\rho \bigr) + \psi(R_{k,n}) \Bigr\} \leq e^{-\rho}\,.
\end{align*}
As argued before, it follows from \eqref{e:flat} and \eqref{e:poly.upper} that there exists $C_2>0$ such that on $\{ \rho \geq -B_2\,, \ \rho + \xi_n(\rho,\theta,y_i) \geq -B_2  \}$,
$$
\prod_{i=1}^{k-1} L(R_{k,n})^{-1} L \Bigl( R_{k,n} + a(R_{k,n}) \bigl( \rho + \xi_n(\rho,\theta,y_i) \bigr) \Bigr) \leq C_2 (\rho \vee 1)^{\gamma(k-1)}.
$$
Since $a$ is eventually non-increasing,
\begin{align*}
\prod_{i=1}^{k-1} &\exp \Bigl\{ -\psi \Bigl( R_{k,n} + a(R_{k,n}) \bigl(\rho + \xi_n(\rho,\theta,y_i)\bigr) \Bigr) + \psi(R_{k,n}) \Bigr\}  \\     
&\leq \prod_{i=1}^{k-1} \exp \Bigl\{- \bigl( \rho + \xi_n(\rho,\theta,y_i) \bigr)\Bigr\}\,.
\end{align*}
Now, we have
\begin{align*}
&n^k \P  \bigl\{ h_n(\X_k) = 1\,, \ \|X_i\| - R_n \geq -a(R_{k,n}) B_2\,, \ i=1,\dots,k\, \bigr\}  \\     
&\leq 4 C C_2 \int_{(\bbr^d)^{k-1}} \int_{S^{d-1}} \int_{\rho \geq -B_2} (\rho \vee 1)^{d-1+\gamma k} e^{-\rho} \prod_{i=1}^{k-1} \exp \Bigl\{- \bigl( \rho + \xi_n(\rho,\theta,y_i) \bigr)\Bigr\}  \\     
&\quad \, \times \one \bigl( \rho + \xi_n(\rho,\theta,y_i) \geq -B_2 \bigr)\, h(0,\by) d\rho J(\theta) d\theta d\by\,.
\end{align*}
If we can show that
\begin{equation}  \label{e:application.dct}
\exp \Bigl\{- \bigl( \rho + \xi_n(\rho,\theta,y_i) \bigr)\Bigr\}  \, \one \bigl( \rho + \xi_n(\rho,\theta,y_i) \geq -B_2 \bigr) \to 0\,, \ \ \ n \to \infty
\end{equation}
for every $\rho \geq -B_2$, $\theta \in S^{d-1}$, and $\|y_i\| \leq M$, $i=1,\dots,d$, then the dominated convergence theorem finishes the proof. First, in the case of $\langle \theta, y_i \rangle < 0$, 
$$
\rho + \xi_n(\rho,\theta,y_i) \to -\infty\,,  \ \ \ n \to \infty
$$
and hence, as $n \to \infty$,
\begin{align*}
\exp &\Bigl\{- \bigl( \rho + \xi_n(\rho,\theta,y_i) \bigr)\Bigr\}  \, \one \bigl( \rho + \xi_n(\rho,\theta,y_i) \geq -B_2 \bigr) \\
&\leq e^{B_2} \, \one \bigl( \rho + \xi_n(\rho,\theta,y_i) \geq -B_2 \bigr) \to 0\,.
\end{align*}
Second, if $\langle \theta, y_i \rangle > 0$, then $\exp \bigl\{ -\bigl(\rho + \xi_n(\rho,\theta,y_i)\bigr) \bigr\} \to 0$. So in either case, \eqref{e:application.dct} is established.
\end{proof}

\noindent {\bf Acknowledgement.} The authors are  grateful to  two  referees and an  associate editor, all of whose comments led to a substantial improvement in the presentation of the paper.

\bibliographystyle{imsart-nameyear}

\bibliography{Takashi_ref}

\end{document}